\documentclass{amsart}

\usepackage{amsmath}
\usepackage{amsfonts}
\usepackage{amssymb}
\usepackage{amsthm}
\usepackage{comment}
\usepackage{epsfig}
\usepackage{psfrag}
\usepackage{mathrsfs}
\usepackage{amscd}
\usepackage[all]{xy}
\usepackage{rotating}
\usepackage{lscape}
\usepackage{amsbsy}
\usepackage{verbatim}
\usepackage{moreverb}
\usepackage{color}
\usepackage{bbm}
\usepackage{eucal}

\newtheorem{theorem}{Theorem}[subsection]
\newtheorem{prop}[theorem]{Proposition}
\newtheorem{lemma}[theorem]{Lemma}
\newtheorem{cor}[theorem]{Corollary}
\newtheorem{conj}[theorem]{Conjecture}
\newtheorem{problem}[theorem]{Problem}

\theoremstyle{definition}
\newtheorem{definition}[theorem]{Definition}

\newtheorem{observation}[theorem]{Observation}

\newtheorem{notation}[theorem]{Notation}
\newtheorem{convention}[theorem]{Convention}
\newtheorem{remark}[theorem]{Remark}
\newtheorem{example}[theorem]{Example}

\definecolor{orange}{rgb}{.95,0.5,0}
\definecolor{light-gray}{gray}{0.75}
\definecolor{brown}{cmyk}{0, 0.8, 1, 0.6}
\definecolor{plum}{rgb}{.5,0,1}

\DeclareMathOperator{\open}{{\sf open}}

\DeclareMathOperator{\inj}{{\sf inj}}
\DeclareMathOperator{\surj}{{\sf surj}}
\DeclareMathOperator{\creat}{{\sf creat}}
\DeclareMathOperator{\emb}{{\sf emb}}
\DeclareMathOperator{\rf}{{\sf ref}}
\DeclareMathOperator{\pcbl}{{\sf p.cbl}}

\DeclareMathOperator{\cylr}{{\sf Cylr}}
\DeclareMathOperator{\cylo}{{\sf Cylo}}

\DeclareMathOperator{\uno}{\mathbbm{1}}

\DeclareMathOperator{\isot}{{\sf Isot}}

\DeclareMathOperator{\bDelta}{\mathbf{\Delta}}
\DeclareMathOperator{\Set}{\sf Set}
\DeclareMathOperator{\Link}{\sf Link}
\DeclareMathOperator{\Unzip}{\sf Unzip}
\DeclareMathOperator{\Spaces}{\cS{\sf paces}}

\DeclareMathOperator{\Aut}{\sf Aut}
\DeclareMathOperator{\colim}{{\sf colim}}

\DeclareMathOperator{\limit}{{\sf lim}}
\DeclareMathOperator{\cone}{\sf cone}

\DeclareMathOperator{\Hom}{\sf Hom}

\DeclareMathOperator{\Fun}{{\sf Fun}}

\DeclareMathOperator{\Map}{{\sf Map}}

\DeclareMathOperator{\exit}{\sf Exit}
\DeclareMathOperator{\Exit}{\bcE{\sf xit}}
\DeclareMathOperator{\fexit}{\fE{\sf xit}}

\DeclareMathOperator{\Shv}{\sf Shv}

\DeclareMathOperator{\Cat}{{\sf Cat}}

\DeclareMathOperator{\gpd}{{\sf Gpd}}

\DeclareMathOperator{\cMfd}{{\sf c}\cM{\sf fd}}

\newcommand{\ov}{\overline}
\newcommand{\w}{\widetilde}
\newcommand{\lag}{\langle}
\newcommand{\rag}{\rangle}

\DeclareMathOperator{\Ar}{{\sf Ar}}

\DeclareMathOperator{\unzip}{\sf Unzip}

\DeclareMathOperator{\man}{\sf Man}

\DeclareMathOperator{\m}{\sf Mod}

\DeclareMathOperator{\shv}{\sf Shv}
\DeclareMathOperator{\psh}{\sf PShv}
\DeclareMathOperator{\Psh}{\sf PShv}
\DeclareMathOperator{\PShv}{\sf PShv}

\DeclareMathOperator{\lkan}{{\sf LKan}}

\DeclareMathOperator{\Diff}{{\sf Diff}}

\DeclareMathOperator{\op}{\mathsf{op}}

\DeclareMathOperator{\bu}{\cB\mathsf{un}}
\DeclareMathOperator{\Bun}{\cB\mathsf{un}}
\DeclareMathOperator{\cBun}{{\sf c}\cB\mathsf{un}}
\DeclareMathOperator{\bun}{\mathsf{Bun}}
\DeclareMathOperator{\cbun}{\mathsf{cBun}}

\DeclareMathOperator{\Bsc}{\cB {\sf sc}}

\DeclareMathOperator{\Snglr}{\cS{\sf nglr}}

\DeclareMathOperator{\st}{\mathsf{st}}

\DeclareMathOperator{\ev}{\mathsf{ev}}

\DeclareMathOperator{\cls}{\mathsf{cls}}
\DeclareMathOperator{\act}{\mathsf{act}}
\DeclareMathOperator{\cbl}{\mathsf{cbl}}

\DeclareMathOperator{\wcbl}{\mathsf{w.cbl}}

\DeclareMathOperator{\Mfd}{{\cM}\mathsf{fd}}

\DeclareMathOperator{\Emb}{\mathsf{Emb}}

\DeclareMathOperator{\strat}{\mathsf{Strat}}
\DeclareMathOperator{\Strat}{\cS\mathsf{trat}}
\DeclareMathOperator{\kan}{\mathsf{Kan}}

\DeclareMathOperator{\spaces}{\cS\mathsf{paces}}

\DeclareMathOperator{\spectra}{\cS\mathsf{pectra}}

\DeclareMathOperator{\sing}{\mathsf{Sing}}
\DeclareMathOperator{\set}{{\mathsf{Set}}}

\DeclareMathOperator{\vfr}{\sf vfr}
\DeclareMathOperator{\fr}{\sf fr}
\DeclareMathOperator{\sfr}{\sf sfr}

\DeclareMathOperator{\Bord}{\cB{\sf ord}}

\DeclareMathOperator{\Sing}{\mathsf{Sing}}

\def\ot{\otimes}

\DeclareMathOperator{\Fin}{\sf Fin}

\DeclareMathOperator{\oo}{\infty}

\DeclareMathOperator{\rfib}{\sf RFib}

\DeclareMathOperator{\Stri}{\boldsymbol\cS\mathsf{tri}}
\DeclareMathOperator{\trans}{\mathsf{Trans}}

\DeclareMathOperator{\tr}{\triangleright}
\DeclareMathOperator{\tl}{\triangleleft}

\newcommand{\ra}{\rightarrow}
\newcommand{\la}{\leftarrow}
\newcommand{\xra}{\xrightarrow}
\newcommand{\xla}{\xleftarrow}
\newcommand{\un}{\underline}

\def\cB{\mathcal B}\def\cC{\mathcal C}\def\cD{\mathcal D}
\def\cE{\mathcal E}\def\cF{\mathcal F}\def\cG{\mathcal G}
\def\cJ{\mathcal J}\def\cK{\mathcal K}\def\cL{\mathcal L}
\def\cM{\mathcal M}\def\cO{\mathcal O}\def\cP{\mathcal P}
\def\cR{\mathcal R}\def\cS{\mathcal S}\def\cT{\mathcal T}
\def\cU{\mathcal U}\def\cV{\mathcal V}\def\cX{\mathcal X}
\def\cY{\mathcal Y}

\def\AA{\mathbb A}\def\DD{\mathbb D}

\def\RR{\mathbb R}

\def\ZZ{\mathbb Z}

\def\sB{\mathsf B}\def\sC{\mathsf C}\def\sD{\mathsf D}
\def\sE{\mathsf E}

\def\sT{\mathsf T}

\def\bdelta{\mathbf\Delta}

\def\fC{\frak C}
\def\fE{\frak E}\def\fF{\frak F}
\def\fI{\frak I}

\def\bcE{\boldsymbol{\mathcal E}}

\def\oC{\ov{\sC}}

\begin{document}

\title{A stratified homotopy hypothesis}
\author{David Ayala}
\author{John Francis}
\author{Nick Rozenblyum}

\address{Department of Mathematics\\Montana State University\\Bozeman, MT 59717}
\email{david.ayala@montana.edu}
\address{Department of Mathematics\\Northwestern University\\Evanston, IL 60208}
\email{jnkf@northwestern.edu}
\address{Department of Mathematics\\University of Chicago\\Chicago, IL 60637}
\email{nick@math.uchicago.edu}
\thanks{DA was supported by the National Science Foundation under award 1508040.}
\thanks{JF was supported by the National Science Foundation under awards 1207758 and 1508040.}
\thanks{NR was supported by the National Science Foundation	while he was in residence at MSRI, during the Fall 2014 semester.}

\begin{abstract}
We show that conically smooth stratified spaces embed fully faithfully into $\oo$-categories. This articulates a stratified generalization of the homotopy hypothesis proposed by Grothendieck. As such, each $\oo$-category defines a stack on conically smooth stratified spaces, and we identify the descent conditions it satisfies. 
These include $\RR^1$-invariance and descent for open covers and blow-ups, analogous to sheaves for the h-topology in $\AA^1$-homotopy theory.
In this way, we identify $\oo$-categories as \emph{striation sheaves}, which are those sheaves on conically smooth stratified spaces satisfying the indicated descent. 
We use this identification to construct by hand two remarkable examples of $\oo$-categories: $\Bun$, an $\oo$-category classifying constructible bundles; and $\Exit$, the absolute exit-path $\oo$-category. 
These constructions are deeply premised on stratified geometry, the key geometric input being a characterization of conically smooth stratified maps between cones and the existence of pullbacks for constructible bundles.

\end{abstract}



\maketitle

\tableofcontents

\section*{Introduction}

The present work lies in a program whose aim is to fuse manifold topology with $\oo$-category theory. The following hypothesis guides our program: a manifold can be encoded as a moduli space for its stratifications. The manifold itself is a point-set geometric object, while the moduli space is $\oo$-categorical. The first step in our program was taken in work with Hiro Lee Tanaka in \cite{aft1}.
There, using a key notion of conical smoothness, foundations were laid for a theory of conically smooth stratified spaces so as to accommodate moduli.

\smallskip

The present work studies the interplay of geometry and homotopy theory of conically smooth stratified spaces. We show that the homotopy theory of conically smooth maps carries a universal property; this property forms a stratified generalization of the homotopy hypothesis put forward by Grothendieck in \cite{groth}, which we now recall.
To a manifold, one can associate an $\oo$-category whose objects are points of the manifold and whose morphisms are paths between points. The homotopy hypothesis asserts that this association is fully faithful: the space of maps between manifolds is homotopy equivalent to the space of functors between associated $\oo$-categories, whatever model for $\oo$-categories one selects. The $\oo$-categories this association produces are $\oo$-groupoids: every morphism is an equivalence.

\smallskip

This association of an $\oo$-groupoid to a manifold has a stratified generalization: the exit-path $\oo$-category of Lurie \cite{HA}. As proposed first by MacPherson and developed by Treumann \cite{treumann}, from a suitable stratified space $X$ one can define an entity $\exit(X)$ whose objects are points of the space and whose morphisms are those paths whose direction is restricted by the stratification: the paths can exit a stratum into a less deep stratum, but they cannot return. We prove that the space of conically smooth maps between conically smooth stratified spaces is homotopy equivalent to the space of functors between their exit-path $\oo$-categories. That is, the exit-path functor
\[\xymatrix{
\Strat\ar@{^{(}->}[rr]^\exit&&\Cat_{\oo}}
\]
is a fully faithful embedding of conically smooth stratified spaces, and conically smooth maps among them, into $\oo$-categories. 
This is our articulation of a stratified homotopy hypothesis.
The $\oo$-categories produced by this association have the property that each endomorphism of an object is an equivalence.

\smallskip

The proof of this result follows from a formulation of d\'evissage of stratified structures, like that alluded to by Grothendieck in \cite{esquisse}. In essence, ours states that one can understand conically smooth stratified spaces by a combined understanding of cones, blow-ups, manifolds with corners, and induction on depth. Further, the same is true in families by the strong regularity properties afforded by conical smoothness.

\smallskip

As a consequence, one can view an $\oo$-category as a pre-stack on conically smooth stratified spaces: given an $\oo$-category $\cC$ and a conically smooth stratified space $K$, a $K$-point of $\cC$ is a functor
\[\exit(K)\longrightarrow\cC\]
from the exit-path $\oo$-category of $K$ to $\cC$. One can then ask what descent properties this pre-stack is endowed with. The following is a list of locality properties that a space-valued presheaf on conically smooth stratified spaces might or might not possess; we make these precise in ~\S\ref{sec.striation} (see Definition~\ref{def.striation}, amplified by Remark~\ref{strong-striation}).

\begin{itemize}
\item {\bf Sheaf:} the presheaf satisfies descent for open covers.
\item {\bf Constructible:} the sheaf is locally constant on each stratum of each conically smooth stratified space.
\item {\bf Cone-local:} the constructible sheaf satisfies descent for blow-ups along deepest strata.
\item {\bf Consecutive:} the cone-local constructible sheaf has values on iterated cones determined by values on single cones.  
\item {\bf Univalent:} the consecutive cone-local constructible sheaf has \emph{underlying} locally constant sheaf which agrees with its \emph{maximal sub}-locally constant sheaf.

\end{itemize}

We define the collection of striation sheaves, $\Stri$, to consist of those presheaves on conically smooth stratified spaces which satisfy all of these conditions. The following is the first main theorem of this paper.
\begin{theorem}\label{thm.strat-hom-hyp}
There is an equivalence 
\[\Stri ~\simeq~\Cat_{\oo}\]
between striation sheaves and $\oo$-categories. This equivalence sends an $\oo$-category $\cC$ to the presheaf on $\Strat$ taking values
\[K \mapsto \Map(\exit(K),\cC)\]
where $\exit(K)$ is the exit-path $\oo$-category of the conically smooth stratified space $K$.
\end{theorem}

Striation sheaves offer a workable model for $\oo$-category theory with appealing theoretical features -- for instance, the $\ZZ/2$ action on $\Cat_{\oo}$ sending $\cC \mapsto \cC^{\op}$ acts on $\Stri$ by Poincar\'e duality of stratifications (see Remark \ref{rm.not-canonical}). However, we do not introduce striation sheaves in order to compete with quasi-categories as a foundational setting of choice. Rather, we use them to make interesting $\oo$-categories by hand from geometry.

\smallskip

The following perspective informs this use. Quasi-categories make for such an economical model for $\oo$-category theory in part by allowing for coherently, rather than strictly, associative compositions of morphisms. It is technically useful to allow for coherence, but it is technically difficult to specify it. As a result, quasi-categories of interest are rarely arrived at by concretely specifiying a simplicial set and verifying the inner horn filling conditions. Rather, they typically result from a sequence of three steps. First, one writes down a topological, simplicial, or otherwise enriched category concretely. Second, one converts it into a quasi-category via a nerve construction. Third, one performs a formal $\oo$-categorical maneuver (such as taking limits/colimits, presheaves, Dwyer--Kan localization, \&c) to effect a desired end result.

\smallskip

We posit that there are profound examples of $\oo$-categories which are impracticable to produce by this sequence, and we advance striation sheaves as a conduit through which manifold topology can form such profound examples. In the present work we construct several: the $\oo$-category $\Bun$ classifying constructible bundles, and an absolute exit-path $\oo$-category $\Exit$. These build toward a third example, the earlier mentioned moduli space of stratifications on a manifold, which is the subject of a future work. 

\smallskip

First, we give a construction of the exit-path $\oo$-category $\exit(X)$ of a conically smooth stratified space $X$ as a striation sheaf. It is the presheaf represented by $X$ itself. The exit-path quasi-category is one of the few quasi-categories which is actually written down by hand, and the verification of inner horn filling conditions by repeated subdivisions is quite involved (see Appendix A of~\cite{HA}). Our verification that $\exit(X)$ is a striation sheaf is perhaps easier because $\exit(X)$ is so geometric. For locality in geometry, covers are often more workable than subdivisions. Specifically, one can take open covers of simplices. Such a maneuver lives outside the world of quasi-categories and simplicial sets, but within that of striation sheaves.

\smallskip

Our foremost example of interest ensues from the geometry of {\it constructible bundles}, a notion introduced in this work's predecessor \cite{aft1}.\footnote{We recently learned that a spiritually similar notion was introduced by Ren\'e Thom in the 1960s; he called them gentle maps, and they have since been called Thom mappings by Mather \cite{mather.survey} and others. See Conjecture \ref{conj.thommaps}.} A constructible bundle $X\ra K$ is a conically smooth map of stratified spaces whose restriction to each stratum $X_{|K_q}\ra K_q$ is a stratified fiber bundle over an ordinary smooth manifold; a likewise condition is also required on links. We now describe the $\oo$-category, $\Bun$, which results from this notion. The objects are conically smooth stratified spaces. A morphism in $\Bun$ between between two conically smooth stratified spaces, $X_0 \ra X_1$, consists of: a third conically smooth stratified space $X$; a constructible bundle $f:X \ra [0,1]$ to the stratified interval whose stratification is given by $\{0\}\subset [0,1]$; identifications of the fibers $X_0 \cong X_{|\{0\}}$ and $X_1 \cong X_{|\{1\}}$. The following commutative diagram depicts a morphism from $X_0$ to $X_1$:
\[\xymatrix{
X_0 \ar@{^{(}->}[r]\ar[d]&X\ar[d]_f^{\cbl}&\ar@{_{(}->}[l]X_1\ar[d]
\\
\{0\}\ar[r]&[0,1]&\ar[l]\{1\}.\\}
\] 
Here $\cbl$ denotes the condition of $f$ being constructible. This implies, in particular, that the restriction $X_{|(0,1]}$ is isomorphic to the product bundle $X_1 \times (0,1]$.

One comes to grips with the intrinsic $\oo$-categorical nature of $\Bun$ in attempting to compose morphisms. Given two constructible bundles $X\ra [0,1]$ and $Y \ra [1, 2]$ with identifications $X_1 \cong Y_1$, one would like to glue the two intervals end-to-end and form a space $X\cup Y \ra [0,2]$ representing a morphism in $\Bun$ from $X_0$ to $Y_2$. However, the restriction $X\cup Y_{|(0,2]}$ need not be equivalent to the product bundle $Y_2\times (0,2]$. To solve this composition problem requires deforming $X\cup Y$ into a conically smooth stratified space mapping constructibly to the stratified interval $(\{0\}\subset [0,2])$. This deformation constitutes a resolution of singularities, retracting certain floating strata to $\{0\}$. Here, one confronts an inner horn-filling problem as a fundamental feature of the geometry of constructible bundles. The solution has several essential steps:
\begin{enumerate}
\item\label{horn.one} There exists an essentially unique constructible bundle $Z\ra \Delta^2$ with restrictions $Z_{|\Delta^{\{0<1\}}} \cong X$ and $Z_{|\Delta^{\{1<2\}}} \cong Y$: this requires a condition on links (Definition \ref{def.cbl}).
\item The restriction $Z_{|\Delta^{\{0<2\}}}\ra \Delta^{\{0<2\}}$ is again a constructible bundle: this is a special case of the pullback property of constructible bundles (Lemma \ref{cbls-pullback}).
\end{enumerate}

\smallskip

Contemplating this problem might cement two notions. First, $\Bun$ does not admit any obvious manifestation in a model of $\oo$-categories with strictly associative compositions, such as topological categories. Second, in verifying this horn-filling condition one would want to use open covers of the topological simplex $\Delta^2$ as a topological space, rather than subdivisions of the simplicial set $\Delta[2]$. Our theory of striation sheaves is tailored to such situations.

\smallskip

Due to the pullback property of Lemma \ref{cbls-pullback}, one can regard $\Bun$ in a natural way as a presheaf on conically smooth stratified spaces, where the value on a stratified space
\[\Bun(K)\]
consists of all constructible bundles over $K$.
Viewed in this way, the previous horn-filling condition of step (\ref{horn.one}) above becomes the consecutivity axiom for striation sheaves. The second main theorem of our work is the following.
\begin{theorem}\label{main2}
$\Bun$ is a striation sheaf.
\end{theorem}
As such, we regard $\Bun$ as an $\oo$-category by way of Theorem \ref{thm.strat-hom-hyp}. A $K$-point of this $\oo$-category, namely, a functor
\[\exit(K) \longrightarrow\Bun~,\]
exactly consists of a constructible bundle $X\ra K$.
The $\infty$-category $\Bun$ encodes a great deal.  From it one can extract spaces of conically smooth open embeddings, the models of moduli spaces of manifolds of Hatcher \cite{hatcher} and Waldhausen \cite{waldhausen}, and spaces of stratifications on manifolds; combinatorially, one can extract the category of based finite sets and the simplicial category $\bdelta$.

\smallskip

Our first two examples of striation sheaves, $\Bun$ and exit-path $\oo$-categories, are related to one another. $\Bun$ parametrizes an absolute version of the exit-path $\oo$-category, $\Exit$, which we now describe. An object of $\Exit$ is a conically smooth stratified space with a point. A morphism between two pointed stratified spaces $X_0$ and $X_1$ is a third stratified space $X$ with a constructible bundle $X\ra [0,1]$, identifications of the fibers with $X_0$ and $X_1$, together with the additional datum of a stratified section $[0,1] \ra X$ starting from the distinguished point of $X_0$ and ending at the distinguished point of $X_1$. This is depicted in the following diagram:
\[\xymatrix{
X_0 \ar@{^{(}->}[r]\ar[d]&X\ar[d]^{\cbl}&\ar@{_{(}->}[l]X_1\ar[d]\\
\ar@/^1pc/[u]\{0\}\ar[r]&[0,1]\ar@/^1pc/[u]&\ar@/_1pc/[u]\ar[l]\{1\}~.\\}
\]
We solve the existence of compositions in $\Exit$ in the same way as for $\Bun$, by realizing $\Exit$ as a presheaf on conically smooth stratified spaces. This description is likewise very natural: an object of $\Exit(K)$ is a constructible bundle to $K$ together with a stratified section, $X\rightleftarrows K$. The following corollary of Theorem~\ref{main2} realizes $\Exit$ as an absolute exit-path $\oo$-category, tying together the three principal examples of striation sheaves introduced in this work.

\begin{cor}
$\Exit$ is a striation sheaf. For a $K$-point of $\Bun$ defined by a constructible bundle $X\ra K$, there is a resulting commutative diagram
\[\xymatrix{
\exit(X)\ar[rr]^-{(X\underset{K}\times X \rightleftarrows X)}\ar[d]&&\Exit\ar[d]\\
\exit(K)\ar[rr]^-{(X\xra{\cbl} K)}&&\Bun\\}\]
which is a limit diagram of $\oo$-categories.
\end{cor}

\subsection*{Future works and conjectures}
This work is the second paper in a larger program, currently in progress. 
We now outline a part of this program, in order of logical dependency. 
This part consists of a number of papers, the last of which proves the cobordism hypothesis, after Baez--Dolan \cite{baezdolan}, Costello \cite{cycat}, Hopkins--Lurie (unpublished), and Lurie \cite{cobordism}.
\begin{enumerate}

\item[\bf \cite{aft1}:] {\bf Local structures on stratified spaces}, by the first two authors with Hiro Lee Tanaka, establishes a theory of stratified spaces based on the notion of conical smoothness.
This theory is tailored for the present program, and intended neither to supplant or even address outstanding theories of stratified spaces.
This theory of conically smooth stratified spaces and their moduli is closed under the basic operations of taking products, open cones of compact objects, restricting to open subspaces, and forming open covers, and it has a notion of derivative which, in particular, gives the following critical result: 
\begin{itemize}
\item[~] For the open cone $\sC(L)$ on a compact conically smooth stratified space $L$, taking the derivative at the cone-point implements a homotopy equivalence between \emph{spaces} of conically smooth automorphisms
\[
\Aut\bigl(\sC(L)\bigr)~\simeq~\Aut(L)~.
\] 
\end{itemize}
This work also introduces the notion of a constructible bundle, along with other classes of maps between stratified spaces.

\item[\bf Present:]  The present work proves that stratified spaces are parametrizing objects for $\infty$-categories.
Specifically, we construct a functor $\exit\colon \strat \to \Cat_\infty$ and show that the resulting restricted Yoneda functor $\Cat_\infty \to \Psh(\strat)$ is fully faithful.
The image is characterized by specific geometric descent conditions.
We call these presheaves \emph{striation sheaves}. 
We develop this theory so as to construct particular examples of $\infty$-categories by hand from stratified geometry: $\Bun$, $\Exit$, and variations thereof. 
As striation sheaves, $\Bun$ classifies constructible bundles, $\Bun \colon  K\mapsto \{X\xra{\sf cbl} K\}$,
while $\Exit$ classifies constructible bundles with a section.

\item[\bf \cite{emb1a}:]
{\bf Factorization homology I: higher categories} proves that vari-framed stratified $n$-manifolds are parametrizing objects for $(\infty,n)$-categories.  
Namely, we construct a tangent classifier $\sT \colon \Exit \to {\cV}{\sf ect}^{\sf inj}$ to an $\infty$-category of vector spaces and injections there among.
We use this to define $\infty$-categories $\cMfd_n^{\vfr}$ of \emph{vari-framed compact $n$-manifolds}, and $\cMfd_n^{\sfr}$ of \emph{solidly framed compact $n$-manifolds}. As a striation sheaf, $\cMfd_n^{\vfr}$ classifies proper constructible bundles equipped with a trivialization of their fiberwise tangent classifier, and $\cMfd_n^{\sfr}$ classifies proper constructible bundles equipped with an injection of their fiberwise tangent classifier into a trivial  $n$-dimensional vector bundle.
We then construct a functor $\fC\colon (\cMfd_n^{\vfr})^{\op} \to \Cat_{(\infty,n)}$ between $\infty$-categories, and use this to define \emph{factorization homology}. This takes the form of a functor between $\infty$-categories
\[
\int \colon \Cat_{(\infty,n)} \longrightarrow \Fun\bigl(\cMfd_n^{\vfr}, \Spaces\bigr)
\]
that we show is fully faithful. In this sense, vari-framed compact $n$-manifolds define parametrizing spaces for $(\infty,n)$-categories. Subsequent papers characterize the essential image of this functor, and establish likewise results for $(\oo,n)$-categories with adjoints, as they relate to solidly $n$-framed compact manifolds.

\item[\bf \cite{bord}:]
{\bf The cobordism hypothesis}, by the first two authors, proves the cobordism hypothesis.  
Namely, for $\cX$ a symmetric monoidal $(\infty,n)$-category with adjoints and with duals, the space of fully extended (framed) topological quantum field theories is equivalent to the underlying $\infty$-groupoid of $\cX$: 
\[
\Map^{\ot}(\Bord_n^{\fr}, \cX) ~\simeq~ \cX^\sim~.
\]  
We first prove the tangle hypothesis, one form of which states that, for $\ast \xra{\uno} \cC$ a pointed $(\infty,n+k)$-category with adjoints, there is a canonical identification 
$\Map^{\ast/}(\cT{\sf ang}_{n\subset n+k}^{\fr}, \cC) \simeq k{\sf End}_\cC(\uno)$ between the space of pointed functors and the space of $k$-endomorphisms of the point in $\cC$.  
The tangle hypothesis is proved in two steps.
The first step establishes versions of the factorization homology functors above in which the higher categories are replaced by \emph{pointed} higher categories, and the manifolds are replaced by possibly non-compact manifolds.  
The second step shows that the pointed $(\infty,n+k)$-category $\cT{\sf ang}^{\fr}_{n\subset n+k}$, as a copresheaf on $\Mfd_{n+k}^{\sfr}$, is represented by the object $\RR^k$. The cobordism hypothesis follows from the tangle hypothesis, represented by the equivalence $\Bord_n^{\fr} \simeq \varinjlim \Omega^k \RR^k$ as copresheaves on $\Mfd_n^{\sfr}$.

\end{enumerate}

We have just described how we ourselves do and intend to build on the present work. 
Next, we outline several problems and conjectures posed by the present work that we ourselves are not presently pursuing.

\begin{conj}
Topological exit-paths define a fully faithful functor
\[
\exit: \Strat^{C^0} \longrightarrow \Cat_{\oo}
\]
from an $\infty$-category of $C^0$ stratified spaces to that of $\oo$-categories.
\end{conj}

\begin{problem}
Find a strict model for $\cBun$. That is, construct a natural topological or $\kan$-enriched category (consequently, with strictly associative compositions) which is equivalent to $\cBun$ as an $\oo$-category. 

\end{problem}

\begin{remark}
By the results of this paper, a morphism in $\cBun$ from $X$ to $Y$ is equivalent to a {\it pre-constructible} bundle map $Y\ra X$, in the sense defined in \cite{aft1}. The $\oo$-category $\cBun^{\op}$ therefore presents an $\oo$-category of pre-constructible bundles. This latter notion does not offer an immediate strict model of $\cBun$, however, because the set of pre-constructible bundle maps $\{Y\ra X\}$ does not in any immediate way form a space with the homotopy type of $\cBun(X,Y)$. Our construction $\cBun^{\op}$ can therefore be thought of as a possible solution to the problem of defining an $\oo$-category whose objects are compact stratified spaces and whose morphisms are pre-constructible bundles among them.
\end{remark}

\begin{conj}\label{conj.thommaps}
Conjecture 1.5.3 in~\cite{aft1} holds that Whitney stratified spaces are examples of conically smooth stratified spaces.  
Supposing this conjecture is true, we conjecture that \emph{Thom mappings}, in the sense of~\cite{mather.survey}, are examples of constructible bundles in the sense of Definition~\ref{def.cbl}.

\end{conj}

\begin{conj}
The essential image of the fully faithful functor
\[
\exit : \Strat \longrightarrow \Cat_{\oo}
\]
consists of the finite $\oo$-categories in which each endomorphism is an equivalence.
\end{conj}

The next problem concerns constructing a natural stratification of a suitable topological space.

\begin{problem}\label{00}
Let $X$ be a locally compact Hausdorff topological space.
Consider the dualizing sheaf $\omega \in \Shv_{\spectra}(X)$ implementing Verdier duality on $X$.  
By Kan extension, this dualizing sheaf defines a limit-preserving functor
\[
\omega \colon \Shv(X)^{\op} \longrightarrow \spectra~.
\]
Consider the terminal localization
\[
\omega \colon \Shv(X)^{\op} \xra{~\rm localization~}  \ov{\Shv}(X)^{\op}  \xra{~\rm conservative~} \spectra
\]
through which this functor factors.
\begin{enumerate}
\item
Consider the full $\infty$-subcategory $\sE(X)^{\op}\subset \ov{\Shv}(X)$ consisting of completely compact objects.
Identify checkable intrinsic conditions on $X$ for which the canonical colimit preserving functor
\[
\Psh\bigl(\sE(X)^{\op} \bigr)\longrightarrow \ov{\Shv}(X)
\]
an equivalence between $\infty$-categories.

\item
Identify checkable intrinsic conditions on $X$ for which the $\infty$-category $\sE(X)$ has the property that each endomorphism in it is an equivalence; yet more, for which $\sE(X)$ is the exit-path $\infty$-category associated to a stratification $\w{X}$ of $X$.

\end{enumerate}
\end{problem}

Here is a suggested approach to the previous problem:
identify checkable intrinsic conditions on $X$ for which the following sequence of properties hold.
\begin{itemize}
\item For each $x\in X$, the ind-spectrum represented by the functor
\[
\{x\in U\underset{\rm open}\subset X\}^{\op}~\subset~{\sf Opens}(X)^{\op} \xra{~\omega~}\spectra
\]
is in fact a finite spectrum $\omega_x$ (in which case, necessarily, it agrees with the stalk of $\omega$ at $x$).

\item
The following relation on $X$ is an equivalence relation:
$x\sim y$ means there exists an open subset $x,y\in U\subset X$ for which the canonical maps between spectra $\omega_x \la \omega(U) \to \omega_y$ are equivalences.  

\item 
Provided the previous point, the following relation on the set of $X_{/\sim}$ is a partial order:
$[x]\leq [y]$ means there is an open subset $x,y\in U\subset X$ about representatives for which the canonical map between spectra $\omega_x \la \omega(U)$ is an equivalence.  
Write this poset as $P_X$.

\item 
Provided the previous two points, the canonical map between underlying sets $X \to P_X$ is continuous.
Provided so, this data $(X\to P_X)$ is a conical stratification in the sense of Definition~A.5.5 of~\cite{HA}.
\end{itemize}

\begin{remark}
Problem~\ref{00}(1), and the first part of~(2), can be posed for a general $\infty$-topos $\cX$ in place of $\Shv(X)$, provided the existence of a dualizing object $\omega\in \Shv_{\spectra}(\cX)$.
As so, we see this line of inquiry as potentially transportable to less point-set topological contexts, such as in finite characteristic algebraic geometry where there is partial evidence of Poincar\'e duality along specified loci (Poitou--Tate duality, for instance).

\end{remark}

\begin{problem}\label{prob.model}
Endow the ordinary categories of striation sheaves and transversality sheaves with model category structures. These model categories should be Quillen equivalent, in parallel, with the complete Segal space and quasi-category model categories.
\end{problem}

\begin{problem}
For $X$ a conically smooth stratified space with $\Aut(X)$ its space of automorphisms and $\Emb^\sim(X)$ its space of self-embeddings which are isotopic to automorphisms, then is the inclusion
\[
\Aut(X) \longrightarrow\Emb^\sim(X)
\]
a homotopy equivalence? 
This is shown under the assumption that $X$ is finitary in the proof of Theorem \ref{isaquasi-cat}.
\end{problem}

We make several remarks before concluding this portion of the introduction.

\begin{remark}\label{rm.not-canonical}
The identification $\Cat_\infty \simeq \Stri$ is not quite canonical.
The space of such identifications is a torsor for $\Aut(\Cat_\infty)\simeq \ZZ/2$, with the opposite identification implemented by considering enter-paths, as opposed to exit-paths. 
Consequently, this description of $\oo$-categories as sheaves on $\strat$ is balanced between opposites (like a 1-dimensional real vector space without a preferred generator). Our equivalence $\Cat_\infty\simeq \Stri$ determines a $\ZZ/2$-action on $\Stri$. 
This action is not obvious, because it is not inherited from an action on $\strat$.  
Nevertheless, there are some $\ZZ/2$-orbits in $\Stri$ consisting entirely of representables: these are the stratified spaces for which there is a Poincar\'e dual stratified space. An example of such is a polygonized smooth closed manifold of dimension at least $1$; a non-example is an unstratified closed manifold. The involution exchanges Poincar\'e duals, and so one can regard this involution on $\Stri$ as a form of Poincar\'e duality.
\end{remark}

\begin{remark}
After To\"en in \cite{toen}, there is a canonical equivalence between any two $\oo$-categories of $\oo$-categories up to taking opposites. Consequently, we will typically not use notation to distinguish between an $\oo$-category which has first been constructed in a particular model, such as simplicial categories, and the associated $\oo$-category in a different model, such as quasi-categories. For instance, $\Strat$ will stand for both a simplicial category and an $\oo$-category. When we make an argument using specific point-set features of a particular model (e.g., by using Kan fibrations of simplicial categories), we will use the name of that model. When we make a argument which works in every model for $\oo$-category theory, then we avoid use of a particular term and simply refer to $\oo$-categories.
\end{remark}

\begin{remark}
We work throughout internal to $\oo$-category theory. For instance, after  \cite{toen} and \cite{clark-chris}, we employ Rezk's complete Segal spaces as an internal construction of the $\oo$-category of $\oo$-categories, rather than as a model category. As such, we do not construct a model structure for striation sheaves or transversality sheaves, it not being necessary to do so for our larger purpose. Nevertheless, such model structures could be beneficial. See Problem~\ref{prob.model}.
\end{remark}

\begin{notation}
In the following, we encounter a number of examples of collections of topological objects which form both a discrete category of interest, such as the category $\strat$ of stratified spaces, and a quite different higher category of interest, such the $\oo$-category $\Strat$ of stratified spaces. In such cases, we will often use calligraphic typeface for the first letter of the higher category. In cases which will never be ambiguous we will stay with the plain sans serifed type. (These cases which will never be ambiguous include: the $\oo$-category $\Cat_{\oo}$ of small $\oo$-categories; the $\oo$-category of functors $\Fun(\cC, \cD)$ between two $\oo$-categories $\cC$ and $\cD$; the $\oo$-category $\psh(\cC):=\Fun(\cC^{\op}, \spaces)$ of space-valued presheaves on an $\oo$-category $\cC$.)
\end{notation}

In what remains of this introduction, we relate the present results and approach to previous work and then give a more technical linear overview of the sections of this work and the results therein.

\subsection*{Relation to previous works}
The theory of {\it conically smooth} stratified spaces, introduced in \cite{aft1} and further developed throughout this work, is designed to carry both the geometric features of the classical theory of stratifications, after Whitney and Thom--Mather, as well as have robust behavior in families, after the homotopy-theoretic stratifications of Siebenmann and Quinn. We detail these connections.

\smallskip

The geometric study of stratifications dates to the seminal works of Whitney, Thom, and Mather (\cite{whitney0}, \cite{whitney1}, \cite{whitney2}, \cite{thom}, and \cite{mather}), introduced toward the study of singular algebraic varieties and of dynamics of smooth maps. The theory of homotopy-theoretic stratifications was advanced by Siebenmann \cite{siebenmann} and Quinn \cite{quinn}. Object-wise, the geometric theory has fine geometric features, such as the openness of transversality due to Trotman \cite{trotman}.
Conversely, the latter topological theories has robust homotopy theoretic features: Siebenmann introduced the study of spaces of stratified maps between his locally-cone stratified sets, and showed this theory has well-behaved moduli. In particular, he proved these spaces of stratified maps are locally contractible, which can be interpreted as an isotopy extension theorem and
ultimately accommodates the existence of classifying spaces for fiber bundles. In contrast, Whitney stratified spaces have not been studied in families, because the naive notion of a map of Whitney stratified spaces, one which is smooth on strata separately, leads to pathologies. Consequently, Siebenmann's results have no historical counterpart in the geometric theory after Whitney. However, the homotopy-theoretric stratifications are insufficent for more geometry. Nonexistence of tubular neighborhoods (see \cite{normal}), absence of transversality, and difficulty in establishing the existence of pullbacks all hinder the coarser topological theory.

\smallskip

The notion of stratified spaces studied in this work modifies the established definitions of Whitney--Thom--Mather by adding a requirement of conical smoothness for stratifications and maps thereof. This notion is intrinsic and makes no recourse to an ambient smooth manifold. Conical smoothness ensures strong regularity properties along closed strata in a conically smooth stratified space, so that there exist tubular neighborhoods along singularity loci. This regularity implies the Whitney conditions as well as many of the differing excellent properties on both the geometric and homotopy-theoretic sides. On the geometric side, Trotman's openness of transversality can be deduced from the stratified inverse function theorem of \cite{aft1}. On the topological side, Siebenmann's stratified isotopy extension theorem follows in our theory by standard arguments on tubular neighborhoods.

\smallskip

Conical smoothness also cures pathologies in the geometric and homotopy-theoretic theories. It does so for a unified reason: on both sides these pathologies stem from pseudo-isotopy theory, and conical smoothness evades pseudo-isotopy. For instance, in Quinn's theory of homotopically stratified sets, there exist obstructions to the existence of tubular neighborhoods; counterexamples using pseudo-isotopy theory are constructed in \cite{htww}. Hughes proves an {\it approximate} tubular neighborhood theorem, valid only in higher dimensions, in \cite{hughes}; his result uses the h-cobordism theorem, which is again an application of pseudo-isotopy theory. Likewise, the same methods of \cite{htww} point out pathologies in the naive notion of families of Whitney stratified spaces: a naive bundle of Whitney stratified spaces need not itself be Whitney stratified. Examples are given by cones parametrized over a circle; unless the associated pseudo-isotopy is a diffeomorphism, the link of the cone-locus of the total space lacks a smooth structure. 
Conical smoothness averts pseudo-isotopic difficulties because of the identification, from \cite{aft1}, between the homotopy types of the space of conically smooth automorphisms of a cone $\sC(X)$ and the conically smooth automorphisms of the compact conically smooth stratified space $X$ which is the link of the cone-point. Without the conical smoothness, automorphisms of a cone is a pseudo-isotopy space for $X$, and consequently one can form pathological stratifications in families by gluing along pseudo-isotopies rather than along automorphisms; see the introduction of \cite{aft1} for a discussion of this construction. The authors of \cite{htww} state that their work concerns the glue that holds together stratified spaces; a conclusion of their work is that pseudo-isotopy permeates this glue. With our methods, one can glue without pseudo-isotopy by requiring conical smoothness; one applies the glue in layers rather than all at once.

\smallskip

A main avatar of this work, the exit-path $\oo$-category, was previously studied in depth by Lurie in \cite{HA}. This notion of exit-paths was first proposed by MacPherson, unpublished, and developed 2-categorically by Treumann in \cite{treumann} and, afterward, by Woolf \cite{woolf}. We use our enriched category $\Strat$ of stratified spaces and conically smooth maps to give a construction of the exit-path $\oo$-category of a conically smooth stratified space as a complete Segal space; Lurie's construction is a quasi-category. Our proof of the fully faithfulness of the functor $\exit:\Strat\ra \Cat_{\oo}$ implies, in particular, that the functor is conservative. As such, it implies a detection criterion for stratified homotopy equivalences: they are stratified maps which induce homotopy equivalences on all strata and all links of strata. This recovers a result first proved by Miller \cite{miller}.

\smallskip

Our study of the descent properties of stacks on conically smooth stratified spaces represented by $\oo$-categories bears a close analogy with Morel--Voevodsky's motivic homotopy theory~(\cite{morel-voevodsky}). We show that our enriched category $\Strat$ is the $\RR$-localization of the ordinary category of stratified spaces $\strat$, in clear analogy with $\AA^1$-localization. We further prove descent for blow-ups, a topological analogue of descent for the h-topology in algebraic geometry (see~\cite{SV}). The univalence property of a striation sheaf corresponds to Rezk's completeness condition, as formulated in \cite{rezk-n}, and Voevodsky's univalence axiom in homotopy type theory. Our construction of striation sheaves from transversality sheaves via the topologizing diagram is a topological analogue of Joyal--Tierney's equivalence between complete Segal spaces and quasi-categories in \cite{joyaltierney}, using Rezk's classifying diagram~(\cite{rezk}). Our topologizing diagram is inspired by and generalizes the method of Hatcher and Waldhausen for constructing moduli spaces of smooth manifolds in \cite{hatcher} and \cite{waldhausen}. To establish all these results makes use of a uniform system for decomposing stratified spaces in terms of links, blow-ups, and induction on depth of strata, which we conceive of as a d\'evissage for stratified structures like that continually mentioned by Grothendieck in his Esquisse d'un Programme \cite{esquisse}.

\smallskip

Lastly, our notion of a constructible bundle is a geometric refinement of the stratified systems of fibrations  introduced by Quinn in \cite{quinnends} for the study of h-cobordisms in families. (The constructible bundle is a further refinement of this notion, as there exists an additional stratification on the total space.)  In Quinn's theory, strata have regular neighborhoods which are open mapping cylinders; it is shown that demanding such for each pair of strata grants as much for links of links. This composability feature foreshadows the consecutivity of $\Bun$, the paramount technical property which allows for composing morphisms in $\Bun$ and so used to prove that it indeed forms an $\oo$-category. Lastly, the verification of univalence for $\Bun$ is a stratified generalization of the homotopy equivalence $\Diff(M)\simeq\Emb^\sim(M,M)$ between diffeomorphisms and self-embeddings of a manifold isotopic to a diffeomorphism.

\subsection*{Linear overview}

{\bf Section 1} recapitulates the fundamental features of conically smooth stratified spaces, including links and the unzipping construction, as developed in \cite{aft1}. Constructible closed stratified subspaces have regular neighborhoods; we make essential use of this technical feature throughout.

\smallskip

{\bf Section 2} studies $\Strat$, a simplicial enrichment of the discrete category of stratified spaces $\strat$. The construction of $\Strat$ is an application of the {\it topologizing diagram}, which produces space-valued presheaves from groupoid-valued presheaves on conically smooth stratified spaces. We introduce isotopy sheaves as a class of groupoid-valued presheaves which become constructible sheaves after applying the topologizing diagram. From this, we deduce that $\Strat$ is the Dwyer--Kan localization of $\strat$ with respect to stratified homotopy equivalences. We proceed to study homotopy colimits in $\Strat$. The key technical input is a description of the homotopy type of spaces of conically smooth maps between cones. Using this, we identify the following distinguished classes of homotopy colimit diagrams: open covering sieves, blow-ups for deepest strata; and double cone gluing diagrams.

\smallskip

{\bf Section 3} concerns the interaction of $\Strat$ with Rezk's theory of complete Segal spaces. We show that the standard stratification of the $n$-simplices defines a fully faithful functor
\[\xymatrix{
\bdelta\ar@{^{(}->}[r]^{\st}& \Strat}
\]
from the simplicial indexing category to the simplicial category of stratified spaces; this sends $[p]$ to the topological $p$-simplex $\Delta^p$ equipped with the standard stratification. We define the simplicial space $\exit(X)$ as the composite,
\[\xymatrix{
\exit(X)\colon \bdelta^{\op}\ar[rr]^{\st}&&\Strat^{\op}\ar[rr]^{\Strat(-,X)}&&\spaces~,}
\]
given by restricting to $\bdelta^{\op}$ the representable presheaf $X$. Using the analysis of homotopy colimits in $\Strat$ from \S2, we prove that $\exit(X)$ is a complete Segal space, hence an $\oo$-category. The equivalence between quasi-categories and complete Segal spaces exchanges $\exit(X)$ with Lurie's exit-path quasi-category of $X$. We show that the resulting functor $\exit: \Strat\ra \Cat_{\oo}$ preserves several distinguished classes of colimit diagrams.

\smallskip

{\bf Section 4} introduces striation sheaves. These are constructible sheaves on conically smooth stratified spaces that send the distinguished classes of colimit diagrams to limit diagrams in $\spaces$. We prove that the $\oo$-category of simplicial spaces $\psh(\bdelta)$ is equivalent to the $\oo$-category $\shv^{\sf cone,\cbl}(\strat)$ of those constructible sheaves on conically smooth stratified spaces which satisfy descent for blow-ups. Using this, we prove the further equivalence between the $\oo$-category of striation sheaves $\Stri$ and the $\oo$-category of $\oo$-categories $\Cat_{\oo}$. These proofs use all of our analysis of homotopy colimits of stratified spaces to show that $\Strat$ is generated under homotopy colimits of distinguished diagrams by the three element $\oo$-subcategory $\{\emptyset, \ast, \Delta^1\}$ consisting of the empty set, a point, and the standardly stratified 1-simplex. As a corollary, we deduce that the exit-path functor $\exit:\Strat\ra \Cat_{\oo}$ is fully faithful: 
\[
\Strat(X,Y)~\simeq~\Map\bigl(\exit(X),\exit(Y)\bigr)~.
\]
 
\smallskip

{\bf Section 5} introduces transversality sheaves as an efficient point-set means for constructing striation sheaves. We prove that the topologizing diagram sends transversality sheaves to striation sheaves; this elaborates on the construction of constructible sheaves from isotopy sheaves from \S2. For transversality sheaves, the topologizing diagram agrees with the classifying diagram of Rezk. This connects our passage between transversality sheaves and striation sheaves to that between quasi-categories and complete Segal spaces of Joyal--Tierney \cite{joyaltierney}.

\smallskip

{\bf Section 6} constructs $\bun$ and $\exit$ using the notion of a constructible bundle of stratified spaces. Toward this study, we establish a number of basic crucial properties of constructible bundles of stratified spaces. By extensive use of the results and techniques of \cite{aft1} we prove: that constructible bundles pull back; that constructible bundles compose; that constructible bundles can be recognized stratum locally in the source. These lemmas are used to prove that $\bun$ and $\exit$ are transversality sheaves, and thus that their topologizing diagrams $\Bun$ and $\Exit$ form striation sheaves, hence $\oo$-categories. The most technically involved verification is the consecutivity condition, which relies on a classification of isomorphism classes of constructible bundles over a cone. In doing so, we prove a conceptually appealing description of morphisms in $\Bun$. 
Namely, there is a surjection from the set of isomorphism classes of spans 
\[
X_0 \xla{~\pcbl~} L  \xra{~\open~}  X_1
\]
to the set of isomorphism classes of constructible bundles $X\to \Delta^1$, and thereafter a surjection to $\pi_0\Bun(\Delta^1)$. As such, $\Bun$ receives an essentially surjective functor from a type of Burnside $\oo$-category formed from the classes of proper constructible and open maps in $\Strat$.

We single out special $\oo$-subcategories of $\Bun$, consisting of closed, creation, refinement, and embedding morphisms. We exhibit a factorization system on $\Bun$ in terms of these morphisms.
 There are two cylinder functors, the open cylinder $\cylo$ and the reversed cylinder $\cylr$, which define monomorphisms $\Strat^{\open}\hookrightarrow \Bun$ and $(\Strat^{\pcbl})^{\op}\hookrightarrow \Bun$ from the $\oo$-categories of stratified spaces with open maps and from the $\oo$-category of stratified spaces with proper constructible bundles. Lastly, we identify the distinguished $\oo$-subcategories of $\Bun$ in terms of these cylinder functors.
\begin{itemize}
\item Embeddings: $\Bun^{\emb}$ is equivalent to the $\oo$-subcategory of $\Strat^{\open}$ consisting of those maps which are stratified open embeddings.
\item Refinements: $\Bun^{\rf}$ is equivalent  to the $\oo$-subcategory of $\Strat^{\open}$ consisting of those maps which are homeomorphisms of underlying topological spaces.
\item Closed morphisms: $\Bun^{\cls}$ is contravariantly equivalent to the $\oo$-subcategory of $\Strat^{\pcbl}$ consisting of those proper constructible bundles which are injective.
\item Creation morphisms: $\Bun^{\creat}$ is contravariantly equivalent to the $\oo$-subcategory of $\Strat^{\pcbl}$ consisting of those proper constructible bundles which are surjective.
\end{itemize}

\subsection*{Acknowledgements}
The present work would not have been conceived without Jacob Lurie's illuminating treatment of exit-paths and constructible sheaves in~\cite{HA}. We thank David Nadler for several conversations, and for pointing out the relation between constructible bundles and Thom's gentle maps. We thank the referees for their careful reading of this paper. JF thanks Jesse Ball.

\section{Conically smooth stratified spaces}\label{recall}
We recall some relevant notions among conically smooth stratified spaces as developed in~\cite{aft1}.
This section is very much an overview, so we point the unfamiliar reader to that reference for precise definitions and details.

\subsection{Stratified topological spaces}

Before defining the \emph{conically smooth} stratified spaces which occupy this work, we first review the baseline topological notions.
We will regard a poset $P$ as a topological space by declaring the subsets $P_{\leq p}$, for each element $p\in P$, to be closed.
A {\it stratified topological space (with stratifying poset $P$)} is a paracompact Hausdorff topological space $X$ with a continuous map $X\ra P$.
Typically we omit the stratifying poset from notation, and simply write $X$ in place of $(X\to P)$.  
For $p\in P$, the {\it $p$-stratum (of $X$)} is the preimage $X_p\subset X$ is the preimage of the element $p\in P$; a stratum of $X$ is a $p$-stratum for some $p\in P$.
The \emph{depth} of a stratified topological space is the depth of its image in its stratifying poset, by which we mean the maximum among lengths of strictly increasing sequences, should it exist.
A map between stratified topological spaces $X\ra Y$ is a commutative diagram of topological spaces
\[
\xymatrix{
X\ar[r]\ar[d]&Y \ar[d]\\
P\ar[r]& Q}
\]
where $P$ and $Q$ are the stratifying posets for $X$ and $Y$. Such maps are closed under composition, thereby organizing stratified topological spaces as a category.

\begin{example}
We observe an essential operation for generating new stratified topological spaces from old ones: taking cones. First, for $P$ is a poset, then the left-cone $P^{\tl}$ on $P$ is the poset defined by adjoining a new minimum element to $P$.
Note the identification $[n]^{\tl} = [n+1]$. For $X$ a stratified topological space indexed by $P$, then the open cone on $X$,
\[
\sC(X):=\{0\} \underset{\{0\}\times X}\coprod [0,1)\times X \longrightarrow \{0\} \underset{\{0\}\times P}\amalg [1]\times P = P^{\tl}~,
\]
is a stratified topological space indexed by $P^{\tl}$. 
Note that the cone $\sC(X)$ carries a natural action by multiplication of the nonnegative reals $\RR_{\geq 0}$, by scaling in the cone coordinate.
Note also that the stratifying poset $P^{\tl}$ has strictly greater depth than the poset $P$, provided the latter exists.
\end{example}

Open subsets of a stratified space inherit stratifications, as follows.
Let $X=(X\to P)$ be a stratified topological space, let $U\subset X$ be an open subset of its underlying topological space.
Consider the poset over $P$
\[
P_{|U}  \longrightarrow P
\]
for which, for each $p\in P$, the fiber over $p$ is the discrete partial order on the connected components of $U\cap X_p$, equipped with the partial ordering so that $(U\cap X_p)_\alpha \leq (U\cap X_q)_\beta$ means the intersection with the closure $(U\cap X_p)_\alpha \cap \ov{(U\cap X_q)_\beta}\neq \emptyset$ is not empty.
The \emph{inherited stratification} of $U$ is the continuous map $U \to P_{|U}$.
Note the morphism between stratified topological spaces
\[
(U \to P_{|U})\longrightarrow (X\to P)~.
\]
We say a morphism between stratified topological spaces is an \emph{open embedding} if it is isomorphic to a morphism of this form.
The composition of two such open embeddings is again an open embedding.
The category of {\it $C^0$ stratified space} is the smallest full subcategory of stratified topological spaces with the following properties:

	\begin{enumerate}
	\item The empty set $\emptyset$ is a $C^0$ stratified space stratified by the empty poset.
	\item If $X$ is a compact $C^0$ stratified space and its stratifying poset $P$ is finite, then the open cone $\sC(X)$ is a $C^0$ stratified space.
	\item If $X$ and $Y$ are $C^0$ stratified spaces, then the product stratified space $X\times Y$ is a $C^0$ stratified space. 
	\item  If $X$ is a $C^0$ stratified space and $U\hookrightarrow X$ is an open embedding between topological spaces, then $U$ is a $C^0$ stratified space.
	\item  If $X$ is a stratified topological space admitting an open cover by $C^0$ stratified spaces, then $X$ is a $C^0$ stratified space.
	\end{enumerate}	
Observe that the following define examples of $C^0$ stratified spaces:
\begin{itemize}
\item a singleton, $\ast= (\ast \to \ast)$;
\item a half-open interval, $[0,1) = \bigl([0,1))\to [1]\bigr)$, with stratification such that the $0$-stratum is precisely $\{0\}$;
\item Euclidean spaces, $\RR^i = (\RR^i\to \ast)$;
\item any $C^0$ manifold;
\item for $X$ a compact $C^0$ manifold, its open cone $\sC(X)$ is a $C^0$ stratified space.  

\end{itemize}
Note that, for $X$ a $C^0$ stratified space, the identity map $X \xra{\sf id} X$ is an open embedding; in other words, a pair of elements $p,q\in P$ in its stratifying poset are related if and only if the intersection with the closure $X_p\cap \ov{X_q}$ is not empty.  
It follows that $C^0$ stratified spaces and open embeddings among them is a category.

\subsection{Conical smoothness}\label{sec.stratified-spaces}

We now add a {\it conical smoothness} condition to our stratifications. The condition of conical smoothness will be present throughout this work. Our definition of this is, unfortunately and perhaps ineluctably, an inductive one.
The induction is on the \emph{depth} of the stratifying poset.

A \emph{conically smooth stratified space} is a $C^0$ stratified space $X$ that is equipped with a \emph{conically smooth atlas}
\[
\Bigl\{\RR^{i_\alpha}\times \sC(Z_\alpha) \hookrightarrow X\Bigr\}_\alpha
\]
by \emph{basics}.
We now explain these terms.
Each $Z_\alpha$ is a compact conically smooth stratified space of depth strictly less than that of $X$; so we assume, by virtue of our induction, that we have already defined what it means for $Z_\alpha$ to be equipped with a conically smooth atlas.
In general, a \emph{basic} is the data of a non-negative integer $i$ and a compact conically smooth stratified space $Z$; together, these data define the $C^0$ stratified space $\RR^i\times \sC(Z)$, which might also be referred to as a \emph{basic} when the conically smooth atlas on $Z$ is understood.
For $P$ the stratifying poset of $Z$, the stratifying poset of $\RR^i\times \sC(Z)$ is $P^{\tl}$; the \emph{cone-locus} of this basic is the stratum $\RR^i = \RR^i\times \ast$, which is that indexed by the adjoined minimum.
Continuing toward our definition of a conically smooth stratified space, a \emph{conically smooth atlas} is a collection of basics openly embedding as $C^0$ stratified spaces into $X$.  This collection satisfies three conditions: 
\begin{enumerate}
\item This collection of open embeddings forms a basis for the topology of $X$, in particular it forms an open cover of $X$.
\item The \emph{transition maps}, by which we mean the inclusions of open embeddings $\RR^i\times \sC(W) \hookrightarrow \RR^j\times \sC(Z)$ over $X$, are \emph{conically smooth}, defined below.
\item This collection is maximal with respect to (1) and (2).  
\end{enumerate}
To complete our definition of a conically smooth stratified space, it remains to explain the condition for an open embedding between basics to be \emph{conically smooth}.
So consider a $C^0$ open embedding $f\colon \RR^i\times \sC(Y) \hookrightarrow \RR^j\times \sC(Z)$ between basics.
This open embedding $f$ is \emph{conically smooth} in the sense of~\S3.3 of~\cite{aft1} if the following conditions are satisfied.
\begin{itemize}
\item {\bf Away from the cone-locus:}
Each element of the atlas $\psi\colon \RR^k\times \sC(W) \hookrightarrow Y$ determines a composite open embedding $\RR^i\times(0,\infty)\times \RR^k \times \sC(W) \hookrightarrow \RR^i\times \sC(Y) \hookrightarrow \RR^j\times \sC(Z)$.
By way of a smooth identification $(0,\infty)\cong \RR$, we recognize the domain of this composition as a basic.  
Using that $f$ is an open embedding between stratified topological spaces, necessarily this is the data of an open embedding $f\psi\colon \RR^{i+1+k}\times \sC(W) \hookrightarrow \RR^j\times \sC(Z)\smallsetminus \RR^j$ to the complement of the cone-locus.  
For each such $\psi$, the condition of \emph{conical smoothness} requires this open embedding $f\psi$ is a member of the atlas of the target, which carries meaning via the induction in the definition of a conically smooth stratified space.

\item {\bf Along the cone-locus:}
Should the preimage of the cone-locus be empty, then the above point entirely stipulates the conical smoothness condition on $f$.  
So suppose the preimage of the cone-locus is not empty.
Using that $f$ is an open embedding between stratified topological spaces, upon representing the values of $f = \bigl[(f_{\parallel},f_r,f_\theta)\bigr]$ as coordinates, then $f_r(p,0,y) = 0$ for all $(p,y)\in \RR^i\times Y$.  
For $0< k <\infty$, the condition that $f$ is \emph{conically $C^k$ (along the cone-locus)} requires the assignment
\begin{equation}\label{derivative}
(p,v,s,y)~\mapsto~\underset{t\to0^+} {\sf lim} \Bigl[\Bigl(\frac{f_{\parallel}(p+tv,ts,y) - f_{\parallel}(p,0,y)}{t}, \frac{f_r(p+tv,ts,y)}{t}, f_\theta(p+tv,ts,y) \Bigr)\Bigr] 
\end{equation}
to be defined as a map $\sD_{\parallel}f\colon \sT\RR^i\times \sC(Y) \to \sT\RR^j\times \sC(Z)$, required to be conically $C^{k-1}$ along the cone-locus (which carries meaning via induction on $k$).  
The map $f$ is \emph{conically smooth} exactly if $f$ is conically $C^k$ for all $k\geq 0$.  
\end{itemize}

Outlined above is a definition of a conically smooth stratified space, in the sense of~\cite{aft1}.  
We now, more briefly, outline the definition of a \emph{conically smooth} map between two; the assumption of conical smoothness on maps between stratified spaces will be present throughout this work.  
Like the definition of a conically smooth stratified space, a conically smooth map between two stratified spaces is offered by induction on depth.  
For $X$ and $Y$ conically smooth stratified spaces, a map $f\colon X\to Y$ between their underlying stratified topological spaces is \emph{conically smooth} if, for each pair of basic charts $\phi\colon \RR^i\times \sC(Y) \hookrightarrow X$ and $\psi\colon \RR^j\times \sC(Z)\hookrightarrow Y$ for which there is a containment $f(\phi(\RR^i)) \subset \psi(\RR^j)$ of the images of the cone-loci, the map $\psi^{-1}\circ f \circ \phi\colon \RR^i\times \sC(Y) \to \RR^j\times \sC(Z)$ is \emph{conically smooth}.
This latter use of the term means conically smooth \emph{away from the cone-locus}, which can be ensured to carry meaning via induction on depth, and conically smooth \emph{along the cone-locus} which means $f$ abides by expression~(\ref{derivative}).

\begin{example}
Note that smooth manifolds are precisely those conically smooth stratified spaces that have no strata of positive codimension. If $g:M\ra N$ is a smooth map between compact smooth manifolds, then the map of cones $\sC(g):\sC(M)\ra \sC(N)$ is conically smooth. If $h:\RR^i \ra \RR^j$ is a smooth map between Euclidean spaces, then the product map $h\times \sC(g): \RR^i \times \sC(M)\ra \RR^j\times \sC(N)$ is again conically smooth.
\end{example}

We distinguish the following important classes of conically smooth maps:
\begin{itemize}
\item {\bf Embedding:}
 $f$ is an \emph{embedding} if $f$ is an isomorphism onto its image.

\item {\bf Open embedding:}
 $f$ is an \emph{open embedding} if it is open as well as an embedding.
 
\item {\bf Refinement:}
We say $f$ is a \emph{refinement} if it is a homeomorphism of underlying topological spaces, and, for each stratum $X_p\subset X$, the restriction $f_|\colon X_p \to Y$ is an embedding.

\item {\bf Open:} 
 $f$ is \emph{open} if it is an open embedding on underlying topological spaces, and $f$ is a refinement onto its image.  

\item {\bf Proper:}
 $f$ is \emph{proper} if $f^{-1}C\subset X$ is compact for each compact subspace $C\subset Y$.  

\item {\bf Fiber bundle:}
 $f$ is a \emph{fiber bundle} if the collection of images $\phi(O) \subset Y$, indexed by pullback diagrams
\[
\xymatrix{
F\times O \ar[r]  \ar[d]
&
X  \ar[d]
\\
O \ar[r]^-{\phi}
&
Y
}
\]
in which the horizontal maps are open embeddings, is a basis for the topology of $Y$.  

\item {\bf Submersion:}
 $f$ is a \emph{submersion} if the collection of images $\psi(U\times O) \subset X$, indexed by diagrams
\[
\xymatrix{
U\times O \ar[r]^-{\psi}  \ar[d]
&
X  \ar[d]
\\
O \ar[r]
&
Y
}
\]
in which the horizontal maps are open embeddings, is a basis for the topology of $X$.

\item {\bf (Weakly) Constructible:}
 $f$ is a \emph{weakly constructible bundle} if, for each stratum $Y_q\subset Y$, the restriction $f_{|}\colon f^{-1}Y_q \to Y_q$ is a fiber bundle. The definition of a \emph{constructible bundle} is inductive based on depth: as the base case, a smooth map $f\colon X\ra Y$ between smooth manifolds is a constructible bundle if it is a fiber bundle; in the inductive step of the definition, a conically smooth map $f\colon X\ra Y$ between stratified spaces is a constructible bundle if it is a weakly constructible bundle and, additionally, if for each stratum $Y_q\subset Y$ the natural map
 \[
\Link_{f^{-1}Y_q}(X) \longrightarrow f^{-1}Y_q \underset{Y_q}\times \Link_{Y_q}(Y)
\]
is a constructible bundle.

\end{itemize}

\begin{example}
For $X$ a compact stratified space, 
the quotient map $X\times [0,1)\ra \sC(X)$ is a weakly constructible bundle, which is not necessarily a constructible bundle.
\end{example}

Conically smooth maps compose, as do each of the named classes of maps. (See Proposition~\ref{cbl-compose} for the {\bf constructible} case.) This yields the following variety of subcategories
\[
\strat^{\sf cbl} ~\subset~\strat~\supset~\strat^{\open}~\supset~\strat^{\emb}~,~\strat^{\sf ref}
\]
where the superscripts indicate an aforementioned class of maps, except ${\sf emb}$ which signifies \emph{open} embeddings.
The category $\strat$ admits finite products.

As with manifolds with boundary, there are conically smooth stratified spaces with boundary.  
A conically smooth stratified space \emph{with boundary} is a conically smooth stratified space $\ov{X}$ together with a closed sub-stratified space $\partial \ov{X} \subset \ov{X}$ for which there is a conically smooth open embedding $\partial \ov{X} \times \RR_{\geq 0}\hookrightarrow \ov{X}$ where here $\RR_{\geq 0}$ is the open cone on a point.  
The existence of collar-neighborhoods verifies that smooth manifolds with boundary are examples of stratified spaces with boundary. 
Note that a conically smooth stratified space might have many, or no, \emph{boundary} structures.  

\begin{example}
The closed cone on a compact conically smooth stratified space $\oC(L)$ is naturally a conically smooth stratified space with boundary, as seen by the inclusion $L\times \RR_{\geq 0} \cong L\times (0,1]\hookrightarrow \oC(L)$.
\end{example}

For $\partial \ov{X}\times \RR_{\geq 0} \hookrightarrow \ov{X}$ a conically smooth stratified space with boundary, we refer to $\partial\ov{X}$ as the \emph{boundary} of $\ov{X}$, and ${\sf int}(\ov{X}):= \ov{X}\smallsetminus \partial \ov{X}$ as the \emph{interior}.  
A conically smooth stratified space $X$ is {\it finitary} if it is the interior of a compact conically smooth stratified space with boundary.

\subsection{Links and regular neighborhoods}\label{sec-reg}

For $L$ a compact conically smooth stratified space, there is a pushout diagram in $\strat$
\begin{equation}\label{cone-pushout}
\xymatrix{
L  \ar[r]^-{\{0\}}  \ar[d]
&
L\times [0,1)  \ar[d]
\\
\ast  \ar[r]
&
\sC(L)~.
}
\end{equation}
There is likewise a pushout diagram upon applying $\RR^i\times -$, thereby witnessing each basic as a pushout. 
Say a stratum $X_0\subset X$ is a \emph{deepest} stratum if $X_0\subset X$ is closed as a subspace.
By design, the conical smoothness of an atlas for $X$ yields, for each deepest stratum $X_0\subset X$, a pushout diagram in $\strat$
\[
\xymatrix{
\Link_{X_0}(X)  \ar[r]  \ar[d]
&
\Unzip_{X_0}(X)  \ar[d]
\\
X_0  \ar[r]
&
X
}
\]
which restricts over each basic neighborhood $ \RR^i\times \sC(L) \hookrightarrow X$ of a point $x\in X_0$ as the square~(\ref{cone-pushout}).  
By iterating this construction, we have a likewise pushout diagram for $X_0\subset X$ not just a single deepest stratum, but a constructible closed sub-stratified space of $X$.
We refer to such a pushout square as a \emph{blow-up square (along $X_0\subset X$)}; and we refer to the upper left stratified space as the \emph{link (of $X_0$ in $X$)}, and the upper right stratified space as the \emph{unzip (of $X$ along $X_0$)}.  
Because it is the case locally, notice that each map in such a blow-up square is proper and constructible.

The existence of a collar-neighborhood of each face of a smooth manifold with corners ultimately grants the existence of a conically smooth tubular neighborhood of $X_0\subset X$ of a deepest stratum:
\[
\xymatrix{
X_0\subset \sC(\pi)~\ar@{^{(}->}[rr]^-{\sf open~emb}&&~X
}
\]
where $\sC(\pi)\to X_0$ is the fiberwise open cone on the link $\Link_{X_0}(X)\xra{\pi}X$ 
This map from the fiberwise cone is an open embedding between stratified spaces.
Thereafter this gives the existence of a conically smooth regular neighborhood of each stratum $X_p\subset X$ of an arbitrary conically smooth stratified space.
Consequently, for each conically smooth stratified space $X=(X\to P)$ and each constructible closed sub-stratified space $X_0\subset X$, there exists a conically smooth map 
\[
\xymatrix{
X_0\subset \sC(\pi_0)~\ar@{^{(}->}[rr]^-{\sf open}&&~X
}
\]
where $\Link_{X_0}(X)\xra{\pi_0} X_0$ is the projection from the link, and $\sC(\pi_0)$ is the fiberwise open cone. This map from the fiberwise cone is a refinement onto its image, which is an open subspace of $X$.  
In particular, the existence of conically smooth bump functions gives the existence of a conically smooth map $X\ra \RR_{\geq 0}$ extending the projection $\sC(\pi_0)\to [0,1)\hookrightarrow \RR_{\geq 0}$ followed by the standard inclusion.

\begin{lemma}[Shrinking Lemma]\label{regularbasis}
Each open cover $\cU_0$ of a conically smooth stratified space $K$ admits an open refinement $\cV_0$ with the property that, for each $V\in \cV_0$, the inclusion factors $V\xra{i}\DD^i\times \oC(L) \xra{\ov{\phi}} K$ in which $\ov{\phi}$ is a conically smooth embedding and $i$ is an isomorphism onto the interior $\RR^i\times \sC(L)\subset \DD^i\times \oC(L)$.

\end{lemma}

\begin{proof}
By design, conically smooth open embeddings from basics form a basis for the topology of $K$. 
Therefore each such $\cU_0$ admits a refinement by basics.  
Simply by scaling, conically smooth embeddings $\RR^i\times \sC(L) \hookrightarrow \RR^i\times\sC(L)$ that factor as the interior through a conically smooth embedding $\DD^i\times\oC(L)\to \RR^i\times\sC(L)$, form a local base for the topology about the origin $0\in \RR^i\times \sC(L)$.  
The result follows.

\end{proof}

\subsection{Sheaves}
Recall the notion of a (space-valued) presheaf on an $\oo$-category (or, in particular, an ordinary category).
\begin{definition}
Given an $\oo$-category $\cC$, the $\oo$-category of presheaves on $\cC$ is
\[
\psh(\cC) := \Fun(\cC^{\op},\spaces)~,
\]
the $\oo$-category of contravariant functors from $\cC$ to the $\oo$-category of spaces.
\end{definition}
There is a Grothendieck topology on $\strat$ induced from the standard Grothendieck topology on topological spaces via the \emph{underlying topological space} functor $\strat \to {\sf Top}$. A sieve $\cU\subset \strat_{/K}$ on a conically smooth stratified space $K$ is a \emph{covering sieve} if it has the following properties.
\begin{itemize}

\item {\bf Open:}
For each $(U\to K)\in \cU$ there is a morphism $(U\to K)\to (U_0\to K)$ in $\cU$ with $U_0\to K$ an open embedding of stratified spaces.

\item {\bf Surjective:}
For each $x\in K$, the object $(\{x\}\to K)$ belongs to $\cU$.  

\end{itemize}
The $\infty$-category of \emph{sheaves} on $\strat$ is the full $\infty$-subcategory
\[
\shv(\strat)~\subset~\Psh(\strat)
\]
consisting of those presheaves $\cF$ for which, for each covering sieve $\cU$ of a conically smooth stratified space $K$, the restriction of $\cF$ along the adjoint diagram
\[
(\cU^{\op})^{\tl}  \cong (\cU^{\tr})^{\op} \longrightarrow \strat^{\op} \xra{~\cF~} \spaces
\]
is a limit diagram. Equivalently, the canonical map $\cF(K) \xra{\simeq} \underset{U\in \cU} \limit \cF(U)$ is an equivalence between spaces for each covering sieve $\cU$ of a conically smooth stratified space $K$. 
\begin{remark}
Since conically smooth stratified spaces are, by definition, locally finite dimensional, a sheaf in the above sense will satisfy the the stronger condition of hyperdescent, i.e., descent for hypercovers (Definition 6.5.3.2 of \cite{HTT} or Definition 4.1 of \cite{dugger-isaksen}). See \S6.5.4 and Corollary 7.2.1.17 of \cite{HTT}.
\end{remark}

Being defined in terms of limits, the inclusion $\Shv(\strat)\to \Psh(\strat)$ preserves limits, and from presentability considerations the adjoint functor theorem can be applied, thereby producing a left adjoint
\[
\Psh(\strat)\longrightarrow \Shv(\strat)
\]
which is \emph{sheafification}.
The restriction of the forgetful functor $\strat_{|{\sf Top}^{\open}} \to {\sf Top}^{\open}$, to \emph{open embeddings} of topological spaces, is a Cartesian fibration; the Cartesian morphisms are \emph{open embeddings} of stratified spaces.
It follows that this collection of covering sieves on $\strat$ is a Grothendieck topology.  
Therefore the sheafification functor is left exact, and the aforementioned adjunction is (the opposite of) a geometric morphism between $\infty$-topoi; see \cite{HTT}.

In a standard manner, we will extend a presheaf on $\strat$ to subspaces of conically smooth stratified spaces. 
Namely, for $K$ a conically smooth stratified space, consider the poset ${\sf Sub}(K)$ of subspaces of the underlying topological space of $K$, ordered by inclusion.  
Taking images defines a functor $\strat_{/K}\to {\sf Sub}(K)$.  
This functor restricts as an equivalence of categories $\strat^{\emb}_{/K} \to {\sf Sub}^{\open}(K)$ from the subcategory of \emph{open embeddings} to $K$, to the subposet of \emph{open} subspaces of $K$.  
We thus have the composite functor
\[
\Psh(\strat) \xra{~{}_{|K}~} \Psh(\strat^{\emb}_{/K})~\simeq~\Psh\bigl({\sf Sub}^{\open}(K)\bigr)\xra{~\sf LKan~}\Psh\bigl({\sf Sub}(K)\bigr)
\]
where ${\sf LKan}$ is given by left Kan extension.
Explicitly, for $\cF$ a presheaf on $\strat$, the value of this composite functor on $\cF$ evaluates on a subspace $K_0\subset K$ as
\[
\cF\mapsto \Bigl( (K_0\subset K)\mapsto \underset{K_0\subset O\underset{\open}\subset K}\colim \cF(O)~=:~\cF(K_0)\Bigr)
\]
where the colimit is over the poset of open subspaces of $K$, each of which canonically inherits the structure of a conically smooth stratified space.  
We will not distinguish in notation between a presheaf $\cF$ on $\strat$ and its image under this composite functor. This way we can simply write $\cF(K_0)$ for value of the above assignment, as indicated.
The definition of a topological space is just so that the category indexing this colimit is filtered (as an $\infty$-category).  
For $\ov{\cU}_0$ a collection of subspaces of $K$ whose collections of interiors cover $K$ with $\ov{\cU}\subset \strat_{/K}$ the sieve on $K$ consisting of those conically smooth maps $X\to K$ that factor through some member of $\ov{\cU}_0$, then paracompactness of $K$ implies the diagram
\[
(\ov{\cU}^{\op})^{\tl}  \xra{~\sf image~} {\sf Sub}(K)^{\op}\xra{~\cF~} \Spaces
\]
is a limit diagram for each sheaf $\cF$ on $\strat$.

\section{Constructible sheaves}
We examine constructible sheaves on the site of conically smooth stratified spaces.  
These are sheaves on $\strat$ that restrict as locally constant sheaves on each stratum of each stratified space.

\subsection{Stratified homotopy}
We give a stratified analogue of a smooth homotopy. Note that a usual homotopy $X\times \RR \ra Y$ is equivalent, by adjunction, to a map $X\times \RR \ra Y\times \RR$ over the projections to $\RR$.

\begin{definition}\label{def.st-htpy}
Let $f,g\colon X\to Y$ be two conically smooth maps among conically smooth stratified spaces.
A \emph{stratified homotopy} from $f$ to $g$ is a conically smooth map $H:X\times \RR\ra Y\times \RR$ over $\RR$ whose restrictions are identified as $f = H_{|\{0\}}$ and $g = H_{|\{1\}}$.  
We write $f\simeq g$ if there exists a stratified homotopy from $f$ to $g$. A conically smooth map $f:X\ra Y$ is a \emph{stratified homotopy equivalence} if there exists a conically smooth map $Y\xra{g} X$ for which $1_X\simeq gf$ and $fg\simeq 1_Y$; $g$ is a \emph{stratified homotopy inverse} of $f$.  

\end{definition}

\begin{observation}\label{R-in-J}
By induction on $i\geq 0$, the projection $X\times \RR^i\to X$ is a stratified homotopy equivalence. 
\end{observation}

\begin{lemma}\label{homotopy-boundary}
For each conically smooth stratified space with boundary $\ov{X}$, the inclusion of the interior ${\sf int}(\ov{X}) \to \ov{X}$ is a stratified homotopy equivalence.

\end{lemma}
\begin{proof}
We first consider the case $\ov{X}=[0,1)$.  
Choose a smooth function $\phi\colon [0,1)\to [0,1)$ satisfying
\begin{itemize}
\item $\phi(0)>0$,
\item $\phi(s)=0$ for $s\geq \frac{1}{2}$,
\item $s+\phi(s)< 1$.
\end{itemize}
The map $[0,1)\xra{g}(0,1)$, given by $g(s) = s+\phi(s)$, is a homotopy inverse to the inclusion of the interior $(0,1)\to [0,1)$. Relevant homotopies are given by $H\colon [0,1)\times [0,1)\to [0,1)$ given by $H_t(s) = s+t\phi(s)$ and $H'\colon (0,1)\times[0,1) \to (0,1)$ given by $H'_t(s) = s+(1-t)\phi(s)$.  

Immediately after the above case, we see the result for the case $\ov{X} = \partial \ov{X}\times[0,1)$.
In the above argument, $f$ and $g$ and $H$ and $H'$ restrict as the identity map on $[\frac{1}{2},1)$. 
By the definition of a stratified space with boundary, the there is an open embedding $\partial\ov{X} \times [0,1)\hookrightarrow \ov{X}$ under $\partial \ov{X}$.
Consequently, this product case extends to the general case.

\end{proof}

\begin{observation}\label{htpy-closure}
For each pair of conically smooth stratified spaces $X$ and $Y$, stratified homotopy defines an equivalence relation on the set $\strat(X,Y)$ of conically smooth maps.  
Moreover, should two of the arrows in a commutative diagram in $\strat$
\[
\xymatrix{
X \ar[rr]^-{h}  \ar[dr]_-{g}
&&
Z
\\
&
Y  \ar[ur]_-{f}
&
}
\]
be stratified homotopy equivalences, then the third, too, is a stratified homotopy equivalence.  

\end{observation}

\begin{observation}\label{htpy-eq}
A conically smooth map $X\xra{f} Y$ is a stratified homotopy equivalence if and only if it fits into a diagram in $\strat$
\[
\xymatrix{
&
X  \ar[d]_{\{0\}}  \ar[rr]^-{1_X}
&&
X  \ar[d]^{\{0\}}
\\
&
X\times \RR  \ar[rr]^-{H}
&&
X\times \RR
\\
Y  \ar[r]^-g  \ar[d]_{\{0\}}
&
X  \ar[r]^-f  \ar[u]^{\{1\}}  
&
Y  \ar[d]^{\{0\}}  \ar[r]^-g
&
X \ar[u]_{\{1\}}
\\
Y\times \RR \ar[rr]^-{H'}
&&
Y\times \RR
&
\\
Y  \ar[rr]^-{1_Y}  \ar[u]^{\{1\}}
&&
Y  \ar[u]_{\{1\}}
&
}
\]
in which both $H$ and $H'$ lie over $\RR$.
\end{observation}

\begin{notation}\label{def.J-and-R}
We use the notation $\cJ$ for the collection of those morphisms in $\strat$ which are stratified homotopy equivalences. We use the notation $\cR$ for the collection of those morphisms in $\strat$ of the form $X\times \RR\xra{\sf pr} X$.

\end{notation}

Observation~\ref{htpy-closure} grants that $\cJ$ forms a subcategory of $\strat$, and Observation~\ref{R-in-J} gives an inclusion $\cR\subset \cJ$.

\begin{lemma}\label{lem.R-is-homotopy}
The canonical map between localizations
\[\strat[\cR^{-1}] \longrightarrow \strat[\cJ^{-1}]\]
is an equivalence of $\oo$-categories.
\end{lemma}
\begin{proof}
Consider a functor $\strat \to \cC$ to an $\infty$-category that carries each morphism in $\cR$ to an equivalence in $\cC$.
Consider a morphism $X\ra Y$ for which there is a stratified homotopy inverse.
Examine the diagram of Observation~\ref{htpy-eq}.
By assumption, the vertical arrows are carried to isomorphisms in $\cC$, and necessarily the upper and the lower horizontal arrows are as well.  
It follows that $X\ra Y$ is carried to an equivalence in $\cC$.  
\end{proof}

\subsection{Constructible sheaves}
We define \emph{constructibility} for sheaves. In the following, $\man$ is the category of smooth manifolds and smooth maps among them. We will make use of the natural inclusion $\man \hookrightarrow \strat$, whose image is those conically smooth stratified spaces consisting of a single stratum.

Let $Z$ be a space.
Cotensoring with $Z$ gives the presheaf
\[
Z^-\colon \man^{\op} \longrightarrow \spaces^{\op} \xra{~\Spaces(-,Z)~} \spaces
\]
where the first functor is the underlying space functor, and the second is Yoneda.

\begin{lemma}\label{R-invt=cbl}
Let $\cF$ be a sheaf on $\strat$.  The following conditions on $\cF$ are equivalent.  
\begin{enumerate}

\item Each stratified homotopy equivalence $f\colon X\to Y$ induces an equivalence between spaces
\[
f^\ast \colon \cF(Y) \xra{~\simeq~} \cF(X)~.
\]

\item For each conically smooth stratified space $K$, the projection $K\times \RR\to K$ induces an equivalence between spaces
\[
\cF(K) \xra{~\simeq~} \cF(K\times \RR)~.
\]  

\item For each conically smooth stratified space $K$, the restriction $\cF_{|K\times -}$ of $\cF$ along $\man \xra{K\times -} \strat$ is a locally constant sheaf.

\item For each conically smooth stratified space $K$, there is a canonical equivalence between presheaves on $\man$,
\[
\cF_{|K\times -}~\simeq ~\cF(K)^- ~,
\]
to the cotensor.

\item The functor $\cF\colon \strat^{\op}\to \Spaces$ factors through $\strat[\cJ^{-1}]^{\op}$.

\item The functor $\cF\colon \strat^{\op}\to \Spaces$ factors through $\strat[\cR^{-1}]^{\op}$.

\end{enumerate}

\end{lemma}

\begin{proof}
Definitionally, (1) is equivalent to (5), and (2) is equivalent to (6).
The equivalence of (1) and (2) follows immediately from Lemma~\ref{lem.R-is-homotopy}.

We now prove that, for each space $Z$, the presheaf $Z^-:=\Spaces(-,Z)$ on $\man$ is a sheaf.
Let $\cU$ be a covering sieve of a smooth manifold $S$.
Consider the resulting composite functor $\cU^{\tr} \to \man\to \spaces$.  
Because $\cU$ is in particular a hypercover of $S$, then by~\cite{dugger-isaksen}, this composite functor is a colimit diagram.
Being in terms of a cotensor, it follows that the composite functor ${Z^-}:(\cU^{\tr})^{\op} \to \man^{\op} \to \spaces^{\op} \ra \spaces$ is a limit diagram.  
This shows that $Z^-$ is a sheaf on $\man$.  

Let us now argue that $Z^-$ is locally constant.  
Let $S$ be a smooth manifold and choose a hypercover $\cU$ of $S$ comprised of Euclidean spaces.  
From the previous paragraph, we know that the canonical map $Z^S \xra{\simeq} \underset{U\in \cU} \limit Z^U$ is an equivalence between spaces.  
Because each $U$ is contractible, we see that $Z^U \simeq  Z$, and we conclude that $Z^-$ is locally constant.  

The full subcategory $\man^{\sf Euc}\subset \man$ consisting of Euclidean spaces forms a hyperbasis for the Grothendieck topology of $\man$ whose associated $\infty$-topos of sheaves is hypercomplete.
There is then a natural equivalence $\cG\xra{\simeq} \cG'$ of sheaves on $\man$ if and only if there is a natural equivalence of restrictions $\cG_{|\sf Euc} \xra{\simeq} \cG'_{|\sf Euc}$ to $\man^{\sf Euc}$.  
Because $\RR^n$ is contractible, then the restriction of $\cF(K)^-$ to $\man^{\sf Euc}$ is canonically identified as the constant sheaf at $\cF(K)$.  
And so~(4) is true if and only if the restriction of $\cF_{|K\times -}$ to $\man^{\sf Euc}$ is constant.  

Likewise,~(3) is true if and only if the restriction of $\cF_{|K\times-}$ to $\man^{\sf Euc}$ is locally constant.  Because each object of $\man^{\sf Euc}$ is contractible, this is the case if and only if $\cF_{|K\times -}$ is constant.  
We conclude that~(3) is equivalent to~(4).  

If~(2) is true, a quick induction argument on $n\geq 0$ implies that the map of spaces $\cF(K) \xra{\simeq} \cF(K\times \RR^n)$, induced by the projection, is an equivalence; so~(2) implies~(4).  
Conversely, if~(4) is true, then the map induced from the projection $\cF(K) \to \cF(K\times \RR)$ is necessarily an equivalence, and we see that~(4) implies~(2).  

\end{proof}

\begin{remark}
Lemma~\ref{R-invt=cbl} fundamentally relies on the existence of conically smooth bump functions, and thereafter on the existence of conically smooth partitions of unity.  
This contrasts with the situation of $\AA^1$-homotopy theory.

\end{remark}

\begin{definition}\label{def.constructible-sheaves}
The $\infty$-category of \emph{constructible sheaves on $\strat$} is the full $\infty$-subcategory 
\[
\Shv^{\sf cbl}(\strat) ~\subset~\Shv(\strat)
\]
consisting of those sheaves that satisfy the equivalent conditions of Lemma~\ref{R-invt=cbl}.

\end{definition}

\begin{observation}\label{cbl-R-local}
The $\infty$-category of constructible sheaves fits into a pullback diagram among $\infty$-categories
\[
\xymatrix{
\shv^{\sf cbl}(\strat)  \ar[r]  \ar[d]
&
\Psh(\strat[\cR^{-1}])  \ar[d]
\\
\shv(\strat)  \ar[r]  
&
\Psh(\strat)
}
\]
which is comprised of fully faithful functors.
\end{observation}

\subsection{Isotopy sheaves}
Here we give a useful method of obtaining constructible sheaves from point-set data.

The adjoint functor theorem grants the existence of a functor
\[
L\colon \Shv(\strat)\longrightarrow  \Shv^{\sf cbl}(\strat)
\]  
which is left adjoint to the inclusion of constructible sheaves on $\strat$ into all sheaves.
The next section makes explicit the values of $L$ on a certain class of sheaves: \emph{isotopy sheaves}.
In the present section, we define isotopy sheaves and give examples. For the next definition, ${\sf Gpd}$ is the category of small ordinary groupoids, and ${\sf Cat}$ is the category of small ordinary categories. 
Recall the notion of a right fibration between $\oo$-categories (and, in particular, between ordinary categories), which is a functor $\cC\ra \cD$ with the right lifting property with respect to all final functors $\cK_0\ra \cK$ between $\infty$-categories: that is, each solid square
\[
\xymatrix{
\cK_0\ar[r]\ar[d]&\cC\ar[d]\\
\cK\ar@{-->}[ur]\ar[r]&\cD}
\]
admits a unique filler, as indicated.
Each presheaf $\cF:\cD^{\op} \ra\Spaces$ determines a right fibration
\[
\cD_{/\cF}:=  \cD \underset{\psh(\cD)}\times \psh(\cD)_{/\cF}  \xra{~\sf pr~} \cD~,
\]
also known as the {\it unstraightening} of $\cF$. 
This association of unstraightening has an inverse, the {\it straightening} construction of Lurie:
\[
{\sf St}\colon \Psh(\cD)~\simeq~ {\sf RFib}_\cD \colon {\sf Un}~.
\]
The straightening of a right fibration $\cC \ra \cD$ is in particular a presheaf $\cD^{\op}\ra \Spaces$ whose value on an object $d\in \cD$ is canonically identified as the fiber $\cC_{|d}$ over $d$.
For more details see \S2.0.0.3 of \cite{HTT} or \cite{fibrations}.

In what follows, we introduce an explicit expression for the values of the functor $L$ on sheaves on $\strat$ with checkable properties.
We make use of the \emph{extended simplices}, which we now define.

\begin{definition}[$\Delta^\bullet_e$]\label{def.extended-delta}
The \emph{extended} cosimplicial smooth manifold $\Delta^\bullet_e\colon \bdelta \ra \man$ takes values
\[
\small
[p]\mapsto \Delta^p_e:= \bigl\{\{0,\dots,p\} \xra{t} \RR\mid \sum_{i=0}^p t_i =1\bigr\}
\text{ on objects, and }\]
\[\small
([p]\xra{\rho}[q])\mapsto  \bigl(t\mapsto (j\mapsto \sum_{\rho(i)=j} t_i)\bigr)
\text{ on morphisms}
\]
The functor $\Delta^\bullet_e$ defines a cosimplicial stratified space by composing with the inclusion $\man \hookrightarrow \strat$. 
The extended boundary $(\partial \Delta^p)_e$ and the extended horns $(\Lambda^p_k)_e$ are defined likewise: $\partial \Delta^p_e \subset \Delta^p_e$ is the subspace consisting of those $t$ for which $t_i=0$ for some $i$; 
$(\Lambda^p_k)_e\subset \partial \Delta^p_e$ is the subspace consisting of those $t$ for which, should $t_k=0$, then $t_j=0$ for some $j\neq k$.
\end{definition}

\begin{notation}[Delta conventions]\label{not.delta}
For the reader's convenience, we here enumerate our uses of the capitalized Greek letter delta.
\begin{itemize}
\item $\bdelta$ is the simplex category; its objects are finite non-empty linearly ordered sets, which we will often write $[p] = \{0<1<\ldots <p\}$ or $\{i< \ldots<i+p\}$, and its morphisms are non-decreasing maps between underlying sets.

\item $\Delta^p$ is the topological $p$-simplex of Definition \ref{def.standard-simplices}; it is equipped with the standard stratification $\Delta^p\ra [p]$ coming from the identification $\Delta^p\cong \oC^p(\ast)$ with a $p$-fold closed cone on a point.
\item $\Delta^p_e$ is the extended $p$-simplex of Definition \ref{def.extended-delta}, an affine hyperplane in $\RR^{p+1}\cong\RR^{\{0,\dots,p\}}$, which is noncanonically isomorphic to $\RR^p$.
\item $\Delta[p]\in \psh(\bdelta)$ is the combinatorial $p$-simplex, which is the presheaf represented by $[p]$.
\end{itemize}
When dealing with diagrams, we will also write $\Delta^{\{i<\ldots <i+p\}}$, $\Delta_e^{\{i<\ldots <i+p\}}$, or $\Delta[\{i<\ldots <i+p\}]$ for the various $p$-simplices in order to indicate that the morphisms in the diagram are those induced by the standard inclusion $\{i<\dots<i+p\}\subset [p+i]$.
\end{notation}

\begin{notation}\label{simp.not}
Let $\fI \colon \strat^{\op}\to\cS$ be a functor to an $\infty$-category with limits.  
For each stratified space $K$, and each simplicial set $D\colon \bDelta^{\op} \to \Set$, we define the value
\[
\fI(K\times D_e)~:=~\limit\bigl( {\bDelta_{/D}}^{\op} \to \bDelta^{\op} \xra{K\times \Delta^\bullet_e} \strat\bigr)~\in \cS~,
\]
as the limit over simplices of $D$.  

\end{notation}

\begin{remark}
The object $\fI(K\times D_e)$ is evidently contravariantly functorial in the simplicial set $D$.  
In particular, each map $D\to D'$ between simplicial sets determines a map $\fI(K\times D'_e) \to \fI(K\times D_e)$.  

\end{remark}

\begin{remark}
We will only implement the above Notation~\ref{simp.not} in the cases that $\cS = \Set$, ${\sf Gpd}$, or $\Spaces$, and that $D=\Lambda_i[p]$ for some $0\leq i \leq p$ or $D=\partial \Delta[p]$ for some $p\geq 0$.  

\end{remark}

\begin{remark}
Because of the standard fact that, for each simplicial set $D$, the full subcategory of $ {\bDelta_{/D}}^{\op}$ consisting of the non-degenerate simplices of $D$ is initial (as an $\infty$-category), the value 
\[
\fI(K\times D_e)~\simeq~ \underset{[q]\xra{\text{non-deg}} D}\limit \fI(K\times \Delta^q_e)
\]
is identical to the limit indexed by the category of \emph{non-degenerate} simplices of $D$.

\end{remark}

\begin{example}
With the preamble of Notation~\ref{simp.not}, the canonical morphism in $\cS$ to the fiber product
\[
\fI(K\times (\Lambda^2_1)_e) 
\xra{~\simeq~}  
\fI(K\times \Delta^{\{0<1\}}_e)\underset{\fI(K\times \Delta^{\{1\}}_e)}\times \fI(K\times \Delta^{\{1<2\}}_e)
\]
is an equivalence.

\end{example}

\begin{definition}\label{def.iso-sheaves}
The category of \emph{isotopy sheaves} is the full subcategory
\[
\isot~\subset~\Cat_{/\strat}
\]
consisting of right fibrations $\fI \to \strat$ for which there is a straightening $\fI \colon \strat^{\op} \to {\sf Gpd}$ that satisfies the following conditions.
\begin{itemize}
\item {\bf Sheaf:} For each covering sieve $\cU \subset \strat_{/K}$, the restriction of $\fI$ along the adjoint diagram $\cU^{\tr} \to \strat$ is a limit diagram of groupoids, which is to say that the canonical map between groupoids
\[
\fI(K)\xra{~\simeq~} \underset{U\in \cU} \limit \fI(U)
\]
is an equivalence.

\item {\bf Isotopy extension:} For each stratified space $K$, the following two conditions are satisfied.
\begin{itemize}
\item
For each $p\geq 0$, the functor between groupoids
\[
\fI(K\times \Delta^q_e) \longrightarrow \fI(K\times \partial \Delta^q_e)
\]
is an isofibration.  

\item
For each $0\leq i \leq p$, the functor between groupoids
\[
\fI(K\times \Delta^q_e) \longrightarrow \fI(K\times (\Lambda^q_i)_e)
\]
is surjective on objects and on Hom-sets.

\end{itemize}

\end{itemize}

\end{definition}

\begin{remark}\label{rem.isofib}
Isofibrations will be relevant in this work for the following simple reason: if $\cE \ra \cG$ is an isofibration of groupoids, then the map on their nerves $\cE_\star \ra \cG_\star$ is a Kan fibration. Likewise, if one has a sheaf of groupoids $\cF: {\sf Opens}(K)^{\sf op}\ra \gpd$ where the pullbacks $\cF(V)\ra \cF(U)$ are isofibrations for every open $U\subset V$, then the associated sheaf of Kan complexes will be a homotopy sheaf of Kan complexes, and so the composite $\cF: {\sf Opens}(K)^{\sf op}\ra \gpd\ra \spaces$ will be sheaf in the $\oo$-categorical sense.
\end{remark}

The terminology of Definition~\ref{def.iso-sheaves} as \emph{sheaves} is justified through the following observation.
\begin{observation}\label{isot-are-shvs}
By definition isotopy sheaves are a full subcategory of the ordinary category of right fibrations over $\strat$.
This category of all right fibrations of $\oo$-categories has a natural simplicial enrichment, and thus yields an $\infty$-category $\rfib_{\strat}$.  
The straightening functor of~\cite{HTT} gives an equivalence of $\infty$-categories $\rfib_{\strat}\xra{\simeq} \Psh(\strat)$ to that of (space-valued) presheaves on $\strat$.
The defining properties of an isotopy sheaf grant that the composite functor $\isot \to \rfib_{\strat} \xra{\simeq} \Psh(\strat)$
factors through $\Shv(\strat)$, (space-valued) sheaves on the site $\strat$.

\end{observation}

\begin{definition}[Topologizing diagram]\label{def.TD}
The \emph{topologizing diagram} functor is the composite
\[
\psh(\strat)\longrightarrow \Fun(\bdelta^{\op}, \psh(\strat)) \longrightarrow \psh(\strat)
\]
of restriction along $\strat \times \bdelta\xra{-\times\Delta^\bullet_e} \strat$ with left Kan extension along the projection $\strat \times \bdelta \to \strat$.  
\end{definition}
Because the projection $\strat\times\bdelta\ra \strat$ is Cartesian, the values of the topologizing diagram can be expressed as the assignment
\[
\fF~\mapsto \Bigl(K \mapsto |\fF(K\times\Delta^\bullet_e)|\Bigr)
\]
where $|-|$ denotes geometric realization.

\begin{remark}
The terminology is intended to evoke the \emph{classifying diagram} functor $\Psh(\bdelta) \to \Psh(\bdelta)$ given by $\cC\mapsto |\cC(\Delta[-]\times E[\bullet])|$ where $E[p]$ is the nerve of the contractible groupoid whose set of objects is $\{0,\dots,p\}$. See ~\cite{rezk}, ~\cite{joyaltierney}, and Remark~\ref{class-vs-top}.
\end{remark}

\begin{lemma}\label{TD-is-cbl}
The topologizing diagram functor factors through $\Psh(\strat[\cJ^{-1}])$.
That is, for each presheaf $\fF$ on $\strat$, and each conically smooth stratified space $K$, projection induces an equivalence between spaces
\[
|\fF(K\times \Delta^\bullet_e)| \xra{~\simeq~}|\fF(K\times\RR\times \Delta^\bullet_e)|~.
\]

\end{lemma}

\begin{proof}
Choose an identification $\Delta^1_e\cong \RR$, after which there results an identification of simplicial objects $\fF(K\times \RR\times \Delta^\bullet_e)\simeq\fF(K\times\Delta^\bullet_e)^{\Delta[1]}$.  
Thereafter, stratified homotopies induce simplicial homotopies.
Therefore stratified homotopy equivalences induce simplicial homotopy equivalences.  

\end{proof}

\begin{theorem}\label{isot.cbl}
The restriction of the topologizing diagram functor to isotopy sheaves factors through constructible sheaves:
\[
\isot\longrightarrow \Shv^{\cbl}(\strat)~,\qquad \fI \mapsto |\fI(-\times\Delta^\bullet_e)|~.
\]
\end{theorem}

\begin{proof}
The topologizing diagram factors
\[
|\fI(-\times \Delta^\bullet_e)|~\simeq~ {\sf loc}\circ \fI^T
\]
as a composition of the following two functors.
The functor ${\sf loc}\colon \Fun(\bDelta^{\op},\Set) \to \Spaces$ is the localization on the weak equivalences with respect to the Quillen model structure on simplicial sets.  
The functor $\fI^T$ is the composite functor
\begin{eqnarray}
\nonumber
\fI^T\colon \strat^{\op} 
&
\xra{-\times \Delta^\bullet_e} 
&
\Fun(\bDelta^{\op},\strat ) 
\\
\nonumber
&
\xra{\fI} 
&
\Fun(\bDelta^{\op},{\sf Gpd}) 
\\
\nonumber
&
\xra{\sf Nerve} 
&
\Fun\bigl(\bDelta^{\op},\Fun(\bDelta^{\op},\Set)\bigr) \cong \Fun(\bDelta^{\op}\times \bDelta^{\op},\Set) 
\\
\nonumber
&
\xra{\sf diagonal^\ast}  
&
\Fun(\bDelta^{\op},\Set)~,
\\
\nonumber
K
&
\mapsto
&
\delta^\ast \fI(K\times \Delta^\bullet_e)_\star~,
\end{eqnarray}
where, for each stratified space $Z$, the simplicial set $\fI(Z)_\star$ is the nerve of the groupoid $\fI(Z)$.

Now let $\cU\subset \strat_{/K}$ be a covering sieve of a stratified space $K$.  
We must show that the canonical morphism in the $\infty$-category of spaces
\begin{equation}\label{50}
|\fI(K\times \Delta^\bullet_e)| \longrightarrow \underset{U\in \cU}\limit |\fI(U\times \Delta^\bullet_e)|
\end{equation}
is an equivalence.
Since $K$ is paracompact by definition, we can choose a locally finite open cover $\cO$ for which the canonical inclusion $\cO\subset \strat_{/K}$ factors through $\cU$.  
Choose a linear ordering on $\cO$, which will hereafter be present though implicit.  
For each order preserving map $[r]\xra{\sigma} \cO$, denote the intersection
\[
U_\sigma~:=~ \bigcap_{i=1}^r \sigma(i)~.
\]
Consider the functor 
\[
\cO^\vee\colon \bDelta^{\op} \longrightarrow \cU
~,\qquad
[r]\mapsto \underset{[r]\xra{\sigma}\cO}  \coprod  U_\sigma ~,
\]
in which, for each $[r]\in \bDelta$, the coproduct is indexed by the set of order preserving maps $[r]\to \cO$.
Note that this functor that carries each morphism to an open inclusion.  
By definition of covering sieves for $\strat$, this functor is initial (as a functor between $\infty$-categories).
Consequently, the canonical projection between limits
\[
\underset{U\in \cU}\limit |\fI(U\times \Delta^\bullet_e)|
\xra{~\simeq~}
\underset{[r]\in \bDelta}\limit \Bigl|\fI\Bigl( \underset{[r]\xra{\sigma}\cO}  \coprod U_\sigma \times \Delta^\bullet_e\Bigr)\Bigr|
\]
is an equivalence between spaces.  
We are therefore reduced to showing the canonical morphism in the $\infty$-category of spaces
\begin{equation}\label{51}
|\fI(K\times \Delta^\bullet_e)| 
\longrightarrow 
\underset{[r]\in \bDelta}\limit \Bigl|\fI\Bigl( \underset{[r]\xra{\sigma}\cO}  \coprod U_\sigma \times \Delta^\bullet_e\Bigr)\Bigr|
\end{equation}
is an equivalence.
Theorem~4.2.4.1 of~\cite{HTT} gives that the functor ${\sf loc}\colon \Fun(\bDelta^{\op},\Set) \to \Spaces$ carries homotopy limit diagrams among Kan complexes to limit diagrams.
We are thereby further reduced to showing that $\fI^T$ takes values in Kan complexes, and that the canonical map between simplicial sets
\begin{equation}\label{52}
\fI^T(K) 
\longrightarrow
\underset{[r]\in \bDelta}{\sf holim}~ \fI^T\Bigl( \underset{[r]\xra{\sigma}\cO}  \coprod U_\sigma \times \Delta^\bullet_e\Bigr)
\end{equation}
is a weak homotopy equivalence with respect to Quillen's model structure.  
We first address the former,  that $\fI^T$ takes values in Kan complexes; we then address the latter, using the Bousfield--Kan homotopy limit.

We argue that the functor $\fI^T\colon \strat^{\op} \to \Fun(\bDelta^{\op},\Set)$ takes values in Kan complexes:
\begin{equation}\label{56}
\fI^T\colon \strat^{\op}\longrightarrow \kan \hookrightarrow~\Fun(\bDelta^{\op},\Set)~.
\end{equation}
Kan complexes are the fibrant objects in Quillen's model structure on simplicial sets \cite{quillen.homotopical}.
Moerdijk \cite{moerdijk} establishes a model structure on bisimplicial sets, $\Fun(\bDelta^{\op}\times \bDelta^{\op},\Set)$, in which the diagonal functor $\delta^\ast\colon \Fun(\bDelta^{\op}\times \bDelta^{\op},\Set)  \to \Fun(\bDelta^{\op},\Set)$ is a right adjoint in a Quillen adjunction between Quillen's model structure on simplicial sets and the induced model structure on bisimplicial sets.
Therefore, the factorization~(\ref{56}) is implied upon verifying that, for each stratified space $K$, the bisimplicial set $\fI(K\times \Delta^\bullet_e)_\star$ is fibrant in this model structure.  
According to~\cite{jardine}, fibrancy of a bisimplicial set $Z$ is the following pair of conditions.
We state these conditions in terms of the external product functor $\boxtimes:\Fun(\bdelta^{\op},\set)\times \Fun(\bdelta^{\op},\set)\xra{\times} \Fun(\bdelta^{\op}\times\bdelta^{\op},\set)$.
\begin{itemize}
\item
For each $p\geq 0$, and each $0\leq l \leq q\neq 0$, the canonical map involving sets of morphisms among bisimplicial sets
\begin{equation}\label{lift1}
\Map\bigl(\Delta[p] \boxtimes \Delta[q] , Z\bigr) 
\longrightarrow
\Map\bigl(\partial \Delta[p] \boxtimes \Delta[q] , Z\bigr)\underset{\Map\bigl(\partial \Delta[p] \boxtimes \Lambda_l[q] , Z\bigr)} \times  \Map\bigl(\Delta[p] \boxtimes \Lambda_l[q] , Z\bigr)
\end{equation}
is surjective.

\item
For each $0\leq k \leq p\neq 0$, and each $q\geq 0$, the canonical map involving sets of morphisms among bisimplicial sets
\begin{equation}\label{lift2}
\Map\bigl(\Delta[p] \boxtimes \Delta[q] , Z\bigr) 
\longrightarrow
\Map\bigl(\Lambda_k[p] \boxtimes \Delta[q] , Z\bigr)\underset{\Map\bigl(\Lambda_k[p] \boxtimes \partial \Delta[q] , Z\bigr)} \times  \Map\bigl(\Delta[p] \boxtimes \partial \Delta[q] , Z\bigr)
\end{equation}
is surjective.

\end{itemize}
We examine these conditions in the case of the bisimplicial set $Z=\fI(K\times \Delta^\bullet_e)_\star$.
In this case, for each $p\geq 0$, the simplicial set $\fI(K\times \Delta^q_e)_\star$ is, by definition, the nerve of a groupoid; so it is a 1-coskeletal Kan complex.
As so, the conditions~(\ref{lift1}) and~(\ref{lift2}) are implied for $Z=\fI(K\times \Delta^\bullet_e)_\star$ by their special cases:
\begin{itemize}
\item
For each $p\geq 0$, the canonical map
\begin{equation}\label{lift1'}
\Map\bigl(\Delta[p] \boxtimes \Delta[1] , Z\bigr) 
\longrightarrow
\Map\bigl(\partial \Delta[p] \boxtimes \Delta[1] , Z\bigr)\underset{\Map\bigl(\partial \Delta[p] \boxtimes \Delta\{0\} , Z\bigr)} \times  \Map\bigl(\Delta[p] \boxtimes \Delta\{0\} , Z\bigr)
\end{equation}
is surjective.

\item
For each $0\leq k \leq p\neq 0$, the canonical maps
\begin{equation}\label{lift2'}
\Map\bigl(\Delta[p] \boxtimes \Delta[0] , Z\bigr) 
\longrightarrow
\Map\bigl(\Lambda_k[p] \boxtimes \Delta[0] , Z\bigr)
\end{equation}
and
\begin{equation}\label{lift2''}
\Map\bigl(\Delta[p] \boxtimes \Delta[1] , Z\bigr) 
\longrightarrow
\Map\bigl(\Lambda_k[p] \boxtimes  \Delta[1] , Z\bigr)\underset{\Map\bigl(\Lambda_k[p] \boxtimes \partial \Delta[1] , Z\bigr)} \times  \Map\bigl(\Delta[p] \boxtimes \partial \Delta[1] , Z\bigr)
\end{equation}
are each surjective.

\end{itemize}
The condition~(\ref{lift1'}) is the condition that the functor between groupoids, $\fI(K\times \Delta^q_e) \to \fI(K\times \partial \Delta^q_e)$ is an isofibration, which is precisely assumed.  
The conditions~(\ref{lift2'}) and~(\ref{lift2''}) are the respective condition that the functor $\fI(K\times \Delta^q_e) \to \fI(K\times  (\Lambda^q_i)_e)$ is surjective on objects and on Hom-sets.  
We conclude that each value of $\fI^T$ is a Kan complex, as desired.

With the conclusion of the previous paragraph, the homotopy limit in~(\ref{52}) can be identified with Bousfield--Kan's formula (Chapter 11 of~\cite{holims.bousfield.kan}) as the totalization of the cosimplicial Kan complex
\[
\bDelta \xra{~\cO^\vee~} \cU^{\op} \longrightarrow \strat^{\op} \xra{~\fI^T~} \kan~:
\]
\begin{equation}\label{57}
\underset{[r]\in \bDelta}{\sf holim}~ \fI^T( \underset{[r]\xra{\sigma}\cO}  \coprod U_\sigma \times \Delta^\bullet_e)
~:=~
{\sf Eq}\Bigl( 
\prod_{([r]\xra{\sigma}\cO)\in \bDelta_{/\cO}} \fI^T(U_\sigma)^{\Delta[r]}  
\rightrightarrows
\prod_{([r]\to[s]\xra{\sigma}\cO)\in \Ar(\bDelta_{/\cO})} \fI^T(U_\sigma)^{\Delta[r]}  
\Bigr)
\end{equation}
involving products respectively indexed by the set of objects, and of morphisms, in the category of simplices $\bDelta_{/\cO}$ of the linearly ordered set $\cO$.
By definition of $\fI^T$, we recognize this homotopy limit, as it is equipped with its map from $\fI^T(K)$, as the equalizer of the functor $\fI^T$ applied to the diagram
\begin{equation}\label{61}
\coprod_{[r]\to[s]\xra{\sigma}\cO}  U_\sigma \times \Delta^r_e
~{}~  \rightrightarrows  ~{}~
\coprod_{[r]\xra{\sigma}\cO}   U_\sigma  \times \Delta^r_e
\end{equation}
in the category $\strat_{/K}$ of stratified spaces over $K$.

Our strategy in what follows is to select an open cover of $K$, the diagram associated to which is equipped with a map to the diagram~(\ref{61}).
Using that $\fI$ is a sheaf, this will determine a section to the map~(\ref{52}).
We repeate this technique to show this section is in fact a homotopy inverse.  
Note that, because $\fI$ is a sheaf, then, for each covering sieve $\cV'$ of $K$, the canonical map between Kan complexes
\[
\fI^T(K) \longrightarrow \underset{V\in \cV}\limit~ \fI^T(V')
\]
is an equivalence.

Consider the coequalizer of the diagram~(\ref{61}) in the ordinary category of topological spaces:
\[
\underset{[r]\in \bDelta}{\sf hocolim}~\underset{[r]\xra{\sigma}\cO}\coprod U_\sigma ~\simeq~
{\sf coEq}\Bigl(\coprod_{[r]\to[s]\xra{\sigma}\cO}  U_\sigma  \times \Delta^r_e
\rightrightarrows 
\coprod_{[r]\xra{\sigma}\cO}   U_\sigma \times \Delta^r_e \Bigr)~\in {\sf Top}_{/K}~.
\]
(Note that the finite intersection property of $\cO$ ensures this coequalizer has finite topological dimension.)
Denote the canonical continuous map
\[
q\colon \underset{[r]\in \bDelta}{\sf hocolim}~\underset{[r]\xra{\sigma}\cO}\coprod U_\sigma \longrightarrow K~,
\]
so the canonical map~(\ref{52}) can justifiably be denoted as $\fI^T(q)$.
Choose a conically smooth partition of unity $(\phi_U)_{U\in \cO}$ subordinate to the open cover $\cO$ of $K$ (Lemma~7.1.1 of~\cite{aft1} grants the existence of such).
For each $x\in K$, denote by $\cO_x\subset \cO$ the full linearly ordered subset consisting of those $U\in \cO$ for which $x\in U$; by choice of $\cO$, this subset $\cO_x$ is non-empty and finite, and is therefore an object of the category $\bDelta$.
This partition of unity defines the continuous map
\begin{equation}\label{54}
\Phi\colon K
\longrightarrow
\underset{[r]\in \bDelta}{\sf hocolim}~\underset{[r]\xra{\sigma}\cO}\coprod U_\sigma
~,\qquad
x\mapsto [x,(\phi_U(x))_{U\in \cO_x}]~.
\end{equation}
By direct inspection, $\Phi$ is a section of $q$: namely, $q\circ \Phi = {\sf id}_K$.

Using that $K$ is paracompact and Hausdorff, choose an open cover 
\begin{equation}\label{38}
\{V_\sigma\} \ \text{ of }K~,
\end{equation}
indexed by the set of objects of the category $\bDelta_{/\cO}$, with these properties:
\begin{itemize}
\item
for each $[r]\xra{\sigma}\cO$, there is an inclusion $V_\sigma \subset U_\sigma$;

\item
for each $x\in V_\sigma$, the sum $\sum_{i=1}^r \phi_{\sigma(i)}(x)=1$.
\end{itemize}
Being an open cover, we have a colimit diagram among topological spaces:
\begin{equation}\label{65}
\coprod_{\sigma'\colon [r]\to[s]\xra{\sigma}\cO}  V_\sigma \cap V_{\sigma'}
~{}~\rightrightarrows ~{}~
\coprod_{[r]\xra{\sigma}\cO}   V_\sigma
\longrightarrow
K~.
\end{equation}
Because this diagram is comprised of open inclusions among stratified spaces, this is a colimit diagram in the category $\strat$ of stratified spaces.
Restricting $\Phi$ to each $V_\sigma$ determines a map between diagrams of topological spaces,
\begin{equation}\label{64}
\Bigl(\coprod_{\sigma'\colon [r]\to[s]\xra{\sigma}\cO}  V_{\sigma} \cap V_{\sigma'}
\rightrightarrows
\coprod_{[r]\xra{\sigma}\cO}   V_{\sigma} \Bigr)
\xra{~\Phi_|~}
\Bigl(\coprod_{[r]\to[s]\xra{\sigma}\cO}  U_\sigma  \times \Delta^r_e
\rightrightarrows 
\coprod_{[r]\xra{\sigma}\cO}  U_\sigma\times \Delta^r_e \Bigr)~,
\end{equation}
over the continuous map~(\ref{54}).
Because the partition of unity $(\phi_U)_{U\in \cO}$ is conically smooth,~(\ref{64}) is in fact a map between diagrams in the category $\strat$ of stratified spaces.
We now explain the sequence of maps among Kan complexes:
{\Small
\begin{eqnarray}\label{63}
\nonumber
\fI^T(\Phi)\colon 
\underset{[r]\in \bDelta}{\sf holim}~ \fI^T( \underset{[r]\xra{\sigma}\cO}  \coprod U_\sigma \times \Delta^\bullet_e)
&
\underset{(\ref{57})}{~:=~}
&
{\sf Eq}\Bigl( 
\prod_{([r]\xra{\sigma}\cO)\in \bDelta_{/\cO}} \fI^T(U_\sigma)^{\Delta[r]}  
\rightrightarrows
\prod_{([r]\to[s]\xra{\sigma}\cO)\in \Ar(\bDelta_{/\cO})} \fI^T(U_\sigma)^{\Delta[r]}  
\Bigr)
\\
\nonumber
&
\xra{~\cJ^T(\Phi_{|})~}
&
{\sf Eq}\Bigl( 
\underset{[r]\xra{\sigma}\cO} \prod \fI^T(V_\sigma) 
\rightrightarrows
\underset{\sigma'\colon [r]\to[s]\xra{\sigma}\cO} \prod \fI^T(V_\sigma\cap V_{\sigma'}) 
\Bigr)
\\
\nonumber
&
\xla{~\simeq~}
&
\fI^T(K)~.
\end{eqnarray}}
\noindent
The top identification is the indicated definition~(\ref{57}) of the homotopy colimit.
The middle map is obtained by applying $\fI^T$ to~(\ref{64}).
The bottom map is $\fI^T$ applied to each inclusion $V_\sigma \hookrightarrow K$;
this map is an equivalence between Kan complexes, as indicated, because $\fI^T$ is a sheaf and $\{V_\sigma\}$ is an open cover of $K$.

We now show that this map $\fI^T(\Phi)$ is a homotopy inverse to the map $\fI^T(q)$.
In light of the identification $q\circ \Phi = {\sf id}_K$, it remains to show that the composition $\fI^T(q)\circ \fI^T(\Phi)$ is homotopic to the identity map on $\underset{S\in \cO}{\sf holim}~ \fI^T(\underset{U\in S}\bigcap U)$.

For each morphism $[r]\xra{\sigma} \cO$ between linearly ordered sets, consider the conically smooth map
\[
H_\sigma \colon 
V_\sigma \times \Delta^r_e\times \Delta^1_e 
\longrightarrow 
U_\sigma \times \Delta^r_e
~,\qquad
\bigl(x,(t_i)_{i=1}^r,(s_0,s_1)\bigr)
\mapsto
\bigl(x~,~(s_0\phi_{\sigma(i)}(x)+s_1t_i)_{i=1}^{r}\bigr)~;
\]
the characteristic properties of $V_\sigma$ ensure that this map indeed takes values in the named codomain.  
There results a span between diagrams in $\strat$:
{\Small
\[
\xymatrix{
\Bigl(\coprod_{\sigma'\colon [r]\to[s]\xra{\sigma}\cO}  V_{\sigma} \cap V_{\sigma'}\times \Delta^r_e
\rightrightarrows
\coprod_{[r]\xra{\sigma}\cO}   V_{\sigma}\times \Delta^r_e \Bigr)  \times \Delta^1_e
\ar[r]^-{H}  \ar[d]
&
\Bigl(\coprod_{[r]\to[s]\xra{\sigma}\cO}  U_\sigma  \times \Delta^r_e
\rightrightarrows 
\coprod_{[r]\xra{\sigma}\cO}  U_\sigma\times \Delta^r_e \Bigr)
\\
\Bigl(\coprod_{[r]\to[s]\xra{\sigma}\cO}  U_\sigma  \times \Delta^r_e
\rightrightarrows 
\coprod_{[r]\xra{\sigma}\cO}  U_\sigma\times \Delta^r_e \Bigr)\times \Delta^1_e
&
}
\]
}
in which the downward map is induced by the open inclusions $V_\sigma \subset U_\sigma$. 
Applying $\fI^T$ to this span between diagrams determines a cospan among Kan complexes
{\Small
\begin{eqnarray}\label{36}
\nonumber
\underset{[r]\in \bDelta}{\sf holim}~ \fI^T( \underset{[r]\xra{\sigma}\cO}  \coprod U_\sigma)
&
~:=~
&
{\sf Eq}\Bigl( 
\prod_{([r]\xra{\sigma}\cO)\in \bDelta_{/\cO}} \fI^T(U_\sigma)^{\Delta[r]}  
\rightrightarrows
\prod_{([r]\to[s]\xra{\sigma}\cO)\in \Ar(\bDelta_{/\cO})} \fI^T(U_\sigma)^{\Delta[r]}  
\Bigr)
\\
\nonumber
&
\xra{~\fI^T(H_|)~}
&
{\sf Eq}\Bigl( 
\prod_{([r]\xra{\sigma}\cO)\in \bDelta_{/\cO}} \fI^T(V_\sigma\cap V_{\sigma'})^{\Delta[r]\times\Delta[1]}  
\rightrightarrows
\prod_{([r]\to[s]\xra{\sigma}\cO)\in \Ar(\bDelta_{/\cO})} \fI^T(V_\sigma)^{\Delta[r]\times \Delta[1]}  
\Bigr)
\\
\nonumber
&
\xla{~\simeq~}
&
{\sf Eq}\Bigl( 
\prod_{([r]\xra{\sigma}\cO)\in \bDelta_{/\cO}} \fI^T(U_\sigma)^{\Delta[r]\times\Delta[1]}  
\rightrightarrows
\prod_{([r]\to[s]\xra{\sigma}\cO)\in \Ar(\bDelta_{/\cO})} \fI^T(U_\sigma)^{\Delta[r]\times \Delta[1]}  
\Bigr)
\\
\nonumber
&
~\cong~
&
\Bigl(\underset{[r]\in \bDelta}{\sf holim}~ \fI^T( \underset{[r]\xra{\sigma}\cO}  \coprod U_\sigma)\Bigr)^{\Delta[1]}
\end{eqnarray}}
The first line is definitional.
The second and third maps are obtained by first applying $\fI^T$ to the above span of diagrams, then taking limits. 
This third map is an equivalence between Kan complexes, as indicated, because $\fI^T$ is a sheaf and $\{V_\sigma\}$ is an open cover of $K$.   
The fourth map is an equivalence because equalizers commute with cotensoring.
Denote the composite map from top left to bottom right as
\[
\fI^T(H)\colon \underset{[r]\in \bDelta}{\sf holim}~ \fI^T( \underset{[r]\xra{\sigma}\cO}  \coprod U_\sigma)
\longrightarrow
\bigl(\underset{[r]\in \bDelta}{\sf holim}~ \fI^T( \underset{[r]\xra{\sigma}\cO}  \coprod U_\sigma)\bigr)^{\Delta[1]}~.
\]
By direct inspection, the composition $\ev_0\circ \fI^T(H) = \fI^T(q)\circ \fI^T(\Phi)$ and the composition $\ev_1\circ \fI^T(H)={\sf id}$ is the identity map on the domain.  
This concludes this proof.

\end{proof}

\begin{remark}\label{model-on-iso-shvs}
Following up on Observation~\ref{isot-are-shvs}, we invite a development of a model structure on categories over $\strat$ for which isotopy sheaves, or some minor variation thereof, form the fibrant objects and for which the topologizing diagram functor to constructible sheaves is fully faithful.
\end{remark}

From the defining property of $L:\shv(\strat)\ra \shv^{\cbl}(\strat)$ as a left adjoint in a localization, after Theorem~\ref{isot.cbl} there exists a canonical natural transformation from the restriction of $L$ to $\isot$ to the topologizing diagram functor. The next result verifies that this is an equivalence.  

\begin{lemma}\label{another.lemma}
For each isotopy sheaf $\fI$, the canonical natural transformation
\[
L\fI\longrightarrow |\fI(-\times\Delta^\bullet_e)|
\]
is an equivalence.
\end{lemma}
\begin{proof}
There is a unit in the adjunction $\fI \to L\fI$.
The topologizing diagram applied to the unit gives a natural transformation 
\[
|\fI(-\times\Delta^\bullet_e)| \longrightarrow |L\fI(-\times \Delta^\bullet_e)|\xra{~\simeq~}L\fI
\]
where the right arrow is induced from the zero-section. This is an equivalence because $L$ takes values in \emph{constructible} sheaves. 
By inspection, this natural transformation factors the unit $\fI \to L\fI$.  
From the universal property of $L$ as a left adjoint in a localization, this natural transformation is a left inverse to the one in the statement of the lemma.
We now verify that this natural transformation is also a right inverse.
By inspection, the resulting endo-transformation $|\fI(-\times\Delta^\bullet_e)|\to L\fI\to |\fI(-\times \Delta^\bullet_e)|$ factors as a composition 
\[
|\fI(-\times \Delta^\bullet_e)| \longrightarrow |\fI\bigl((-\times\Delta^\bullet_e)\times\Delta^{\bullet'}_e\bigr)|\longrightarrow |\fI(-\times\Delta^{\bullet'}_e)|
\]
of that induced by projection off the second cosimplicial factor, followed by that induced by the zero-section of the first factor.
Because $|\fI(-\times\Delta^\bullet_e)|$ is constructible, for each fixed $\bullet$ the left arrow is an equivalence; for the same reason the second arrow is an equivalence for each fixed $\bullet$.

\end{proof}

We complete this subsection by recording the following technical observations for later use.

\begin{lemma}\label{two.formals}
Let $\fI\colon \strat^{\op} \to {\sf Gpd}$ be a sheaf.
Let $K$ be a stratified space, and let $0\leq i \leq q$.
\begin{enumerate}
\item
The functor $\fI(K\times \Delta^q_e) \to  \fI(K\times (\Lambda^q_i)_e)$ is an isofibration provided the functor $\fI(K\times \Delta^{r}_e) \to \fI(K\times \partial \Delta^{r}_e)$ is an isofibration for $r=q,q-1$.

\item
Both of the canonical functors 
\[
\fI(K\times \Delta^1_e\times \Delta^q_e)
\longrightarrow
\fI(K\times \partial \Delta^1_e \times \Delta^q_e) \underset{\fI(K\times \partial \Delta^1_e\times (\Lambda^q_i)_e)}\times
\fI(K\times \Delta^1_e\times (\Lambda^q_i)_e)
\]
and
\[
\fI(K\times \RR \times \Delta^q_e)
\longrightarrow
\fI(K \times \Delta^q_e) \underset{\fI(K  \times (\Lambda^q_i)_e)}\times
\fI(K\times \RR\times (\Lambda^q_i)_e)
\]
are surjective on objects provided, for each $0\leq j \leq r$, the canonical functor
\[
\fI(K\times \Delta^r_e) 
\longrightarrow
\fI(K\times (\Lambda^r_j)_e)
\]
is surjective on objects.

\item
The canonical functor
\[
\fI(K\times \Delta^{q}_e \times \RR) 
\longrightarrow
\fI(K\times \Delta^{q}_e)  \underset{\fI(K\times (\Lambda^{q}_i)_e) }\times  \fI(K\times (\Lambda^{q}_i)_e \times \RR)
\]
is an isofibration provided, for each $0\leq j\leq r$, the canonical functor
\[
\fI(K\times \Delta^r_e) 
\longrightarrow
\fI(K\times (\Lambda^r_j)_e)
\]
is an isofibration.

\end{enumerate}

\end{lemma}

\begin{proof}
Consider the diagram among groupoids
\[
\xymatrix{
\fI(K\times \Delta^q_e)  \ar[r]
&
\fI(K\times \partial \Delta^q_e)   \ar[r]  \ar[d]
&
\fI(K\times (\Lambda^q_i)_e)  \ar[d]
\\
&
\fI(K\times \Delta^{q-1}_e)   \ar[r]
&
\fI(K\times \partial \Delta^{q-1}_e) 
}
\]
induced from the inclusion of pairs of simplicial sets 
\[
(\partial \Delta[q-1] \subset \Delta[q-1]) \cong (\partial \Delta[q\smallsetminus \{i\}] \subset \Delta[q\smallsetminus \{i\}]) \hookrightarrow(\Lambda_i[q] \subset \Delta[q])~.
\]
The square in this diagram is a pullback.  
It is assumed that the upper left and bottom horizontal functors are isofibrations.
Statement~(1) follows because isofibrations are closed under base change and composition.

Consider the composite functor ${\sf Obj}~\fI(K\times (-)_e) \colon\Fun(\bDelta^{\op},\Set)\to \Set$ whose value on a simplicial set $D$ is the set of objects of the groupoid $\fI(K\times D_e):=\underset{[q]\xra{\rm non\text{-}deg} D} \limit \fI(K\times \Delta^q_e)$ from Notation~\ref{simp.not}.
The assumption is that this simplicial set is a Kan complex.
The set of inclusions $\bigl\{\Lambda_j[r]\hookrightarrow \Delta[r]  \bigr\}_{0\leq j\leq r>0}$ generates the acyclic cofibrations in Quillen's model structure on simplicial sets.  
The result follows because both of the inclusions between simplicial sets
\[
\partial\Delta[1]\times \Delta[q] \underset{\partial \Delta[1]\times \Lambda_i[q]}\coprod \Delta[1]\times \Lambda_i[q]
~{}~\hookrightarrow~{}~
\Delta[1]\times \Delta[q]
~{}~\hookleftarrow~{}~
\Lambda_1[1]\times \Delta[q] \underset{\Lambda_1[1]\times \Lambda_i[q]}\coprod \Delta[1]\times \Lambda_i[q]
\]
are acyclic cofibrations.

The inclusion between finite simplicial sets 
\[
\Delta[q]\times \Delta\{0\} \underset{\Lambda_i[q]\times\Delta\{0\}}\coprod \Lambda_i[q]\times \Delta[1]
~{}~\hookrightarrow~{}~
\Delta[q]\times \Delta[1]
\]
is a monomorphism, which is a cofibration in Quillen's model structure on simplicial sets. 
Therefore, the functor 
\[
\fI(K\times \Delta^{q}_e \times \RR) 
\longrightarrow
\fI(K\times \Delta^{q}_e)  \underset{\fI(K\times (\Lambda^{q}_i)_e) }\times  \fI(K\times (\Lambda^{q}_i)_e \times \RR)
\]
is a finite limit of functors of the form
\[
\fI(K\times \Delta^r_e) 
\longrightarrow
\fI(K\times (\Lambda^r_j)_e)~.
\]
Using that isofibrations among groupoids are closed under finite limits, statement~(3) follows from statement~(1).

\end{proof}

\subsection{The $\infty$-category $\Strat$}
The identification $L\fI \simeq |\fI(-\times\Delta^\bullet_e)|$ of Lemma~\ref{another.lemma} gives an explicit $\kan$-enrichment of $\strat$ whose associated $\infty$-category agrees with the localization $\strat[\cJ^{-1}]$.  

\begin{lemma}\label{lem.strat-isot}\label{strat-to-isot}
For each conically smooth stratified space $K$, the projection from the over category $\strat_{/K}\to \strat$ is an isotopy sheaf. 
\end{lemma}

\begin{proof}
The functor $\strat_{/K}\ra \strat$ is manifestly a right fibration, being the unstraightening of the representable presheaf $\strat(-,K)\colon\strat^{\op}\ra \set$. 
This functor $\strat(-,K)$ carries colimit diagrams in $\strat$ to limit diagrams in $\set$.
Because covering sieves in $\strat$ give colimit diagrams, the sheaf property follows.  
The isotopy condition follows directly after the existence of regular neighborhoods of constructible closed subspaces.

\end{proof}

Lemma~\ref{lem.strat-isot} informs us that the Yoneda functor $K\mapsto (\strat_{/K} \to \strat)$ factors through isotopy sheaves:
$
\strat \longrightarrow \isot.
$ 
Through Theorem~\ref{isot.cbl}, this leaves us with a functor 
\begin{equation}\label{strat-to-cbls}
\strat \longrightarrow \Shv^{\cbl}(\strat)~,\qquad K\mapsto L\strat(-,K)~.
\end{equation}

\begin{definition}\label{def.Strat}
The $\infty$-category $\Strat$ of \emph{conically smooth stratified spaces} is the essential image
\[
\strat\longrightarrow \Strat~\subset~\Shv^{\cbl}(\strat)
\]
of the functor~(\ref{strat-to-cbls}).  
We denote the defining functor as $c\colon \strat \to \Strat$.  

\end{definition}

There is an immediate consequence of Lemma~\ref{another.lemma}.  
\begin{cor}\label{Strat-Kan}
The $\infty$-category $\Strat$ has the following model as a $\kan$-enriched category.
An object of $\Strat$ is a conically smooth stratified space.
The simplicial set of morphisms from $X$ to $Y$ is $\strat(X\times\Delta^\bullet_e,Y)$.  
Composition is given by the assignment 
\[
(X\times\Delta^q_e\xra{f} Y)\circ (Y\times \Delta^q_e\xra{g} Z) = (X\times \Delta^q_e \xra{(x,t)\mapsto g(f(x,t),t)} Z)~.
\]  
As so, the given functor $\strat \to \Strat$ is Hom-wise inclusion of the 0-simplices.  

\end{cor}

\begin{observation}\label{internal-R-invariance}
For each conically smooth stratified space $X$, the projection $X\times \RR\to X$ is an equivalence in the $\infty$-category $\Strat$.
\end{observation}

\begin{theorem}\label{R-local}
The natural functor
$\strat \longrightarrow \Strat$
induces an equivalence between $\infty$-categories
\[
\strat[\cJ^{-1}]\xra{~\simeq~}\Strat
\]
from the localization of the ordinary category $\strat$ with respect to stratified homotopy equivalences and the $\oo$-category associated to the $\kan$-enriched category $\Strat$. 

\end{theorem}

\begin{proof}
We have a commutative diagram
\[\xymatrix{
\strat\ar@{_{(}->}[d]\ar[r]&\strat[\cJ^{-1}]\ar@{_{(}->}[d]\\
\shv(\strat)\ar[r]^-L&\shv^{\cbl}(\strat)\\}
\]
where the vertical functors are Yoneda functors; these are fully faithful.

\end{proof}

The defining functor $c\colon \strat \to \Strat$ yields an adjunction
\begin{equation}\label{discrete-continuous}
c^\ast \colon \Psh(\Strat) \rightleftarrows \Psh(\strat)\colon c_\ast
\end{equation}
given by restriction and right Kan extension.  
\begin{definition}\label{def.cont-sheaves}
The $\infty$-category of \emph{sheaves on $\strat$} is the pullback
\[
\xymatrix{
\Shv(\Strat) \ar[r]  \ar[d]
&
\Psh(\Strat)  \ar[d]^-{c^\ast}
\\
\Shv(\strat)  \ar[r]
&
\Psh(\strat).  
}
\]\end{definition}
In other words, a presheaf $\cF$ on $\Strat$ is a sheaf if, for each covering sieve $\cU \subset \strat_{/K}$, the diagram $(\cU^{\op})^{\tl} \to \strat^{\op} \to \Strat^{\op} \xra{\cF} \Spaces$ is a limit diagram.

\begin{theorem}\label{cbl-sheaves}
The adjunction~(\ref{discrete-continuous}) restricts as an equivalence of $\infty$-categories
\[
\shv(\Strat)~\simeq~ \shv^{\sf cbl}(\strat)
\]
between constructible sheaves on $\strat$ and sheaves on $\Strat$.
\end{theorem}

\begin{proof}

This is immediate after Theorem~\ref{R-local}, in light of Observation~\ref{cbl-R-local}.

\end{proof}

\subsection{Distinguished colimits in $\Strat$}\label{sec.colims-in-Strat}
We name some colimits in the $\infty$-category $\Strat$, most of which are colimits in $\strat$ as well.
The \emph{new} colimits that we examine arise from a classification of maps among basic singularity types, up to $\RR$-invariance. This clasification is not available in the non-$\RR$-invariant situation.

\subsubsection{\bf Open covers and blow-ups}
We show that an open cover gives a colimit diagram in $\Strat$, and also that a blow-up diagram along a deepest stratum is a colimit diagram in $\Strat$.

\begin{lemma}\label{lem.opens-cover}
For each covering sieve $\cU^{\tr} \to \strat$ of a conically smooth stratified space $X$, the composite functor
\[
\cU^{\tr}\longrightarrow \strat \longrightarrow\Strat
\]
is a colimit diagram. 

\end{lemma}

\begin{proof}
By definition of the $\infty$-category $\Strat$, for each conically smooth stratified space $Z$, the presheaf 
\[
\strat^{\op}\xra{\Strat(-,Z)}\Spaces
\]
is a sheaf. Since this is in fact constructible, the result follows.

\end{proof}

\begin{lemma}\label{lem.blow-up-colims}
For each conically smooth stratified space $X$ with a deepest stratum $X_d \subset X$, the blow-up square
\[
\xymatrix{
\Link_{X_d}(X) \ar[r]  \ar[d]
&
\Unzip_{X_d}(X)  \ar[d]
\\
X_d  \ar[r]
&
X
}
\]
is a pushout diagram in $\Strat$.  

\end{lemma}

\begin{proof}
Fix a conically smooth stratified space $Z$.
We must show that the diagram of spaces and restriction maps among them
\[
\xymatrix{
\Strat\bigl(\Link_{X_d}(X),Z\bigr)
&
\Strat\bigl(\Unzip_{X_d}(X) ,Z\bigr)  \ar[l]
\\
\Strat(X_d,Z)  \ar[u]
&
\Strat(X,Z) \ar[l]  \ar[u]
}
\]
is a pullback. 
Proposition 8.2.5 of~\cite{aft1} states exactly that the inclusions $X_d \hookrightarrow X$ and $\Link_{X_d}(X)\hookrightarrow \Unzip_{X_d}(X)$ each have stratified regular neighborhoods.
It follows that the horizontal maps are Kan fibrations.
So it is enough to verify that each map of fiber Kan complexes is an isomorphism.
This follows because the blow-up diagram in question is a pushout in $\strat$, and $-\times \Delta^q_e$ preserves colimits in $\strat$.

\end{proof}

\begin{cor}\label{cones-colimits}
For each compact conically smooth stratified space $L$, the diagram
\[
\xymatrix{
L \ar[r]^-{\{0\}}  \ar[d]
&
L\times[0,1)  \ar[d]
\\
\ast  \ar[r]
&
\sC(L)
}
\]
is a pushout in $\Strat$.

\end{cor}

\subsubsection{\bf Double-cones}
We show that double closed cones are pushouts in $\Strat$ in terms of single closed cones.

\begin{lemma}\label{lem.to-cone}
For any conically smooth stratified spaces $X$ and $Z$, where $Z$ is compact, there is a canonical identification of the space of maps to the cone
\[
\Strat\bigl(X, \sC(Z)\bigr)~\simeq~ \underset{X_0\underset{\sf cbl,cls}\subset X}\coprod \Strat\bigl(X\smallsetminus X_0, Z\bigr)
\]
with the coproduct indexed by sub-stratified spaces $X_0 \hookrightarrow X$ whose inclusion is constructible and closed.

\end{lemma}

\begin{proof}
Let $X\times\Delta^q_e\xra{f} \sC(Z)$ be a conically smooth map.  Consider the preimage $f^{-1}(\ast)\subset X\times \Delta^q_e$.  Because $\ast\hookrightarrow \sC(Z)$ is constructible and closed, then so is the inclusion of this preimage.  
Because $\Delta^q_e$ is trivially stratified, then the sub-stratified space $f^{-1}(\ast)\subset X\times \Delta^q_e$ is of the form $X_0\times\Delta^q_e\subset X\times \Delta^q_e$ for a unique sub-stratified space $X_0\subset X$ whose inclusion is constructible and closed. 
It is straightforward to notice that the assignment $f\mapsto X_0$ defines a map from the Kan complex $\Strat\bigl(X,\sC(Z)\bigr)$ to the indexing set of the coproduct.  
In particular, we recognize
\[
\Strat\bigl(X,\sC(Z)\bigr)~\simeq~\underset{X_0\underset{\sf cbl,cls} \subset X}\coprod \Strat_{X_0}\bigl(X,\sC(Z)\bigr)
\]
as a coproduct, where the $X_0$-cofactor is the space of those maps $X\times \Delta^q_e\xra{f} \sC(Z)$ for which $X_0\times\Delta^q_e = f^{-1}(\ast)$.  

It remains to show that the composite map
\[
\Strat_{X_0}\bigl(X,\sC(Z)\bigr)\longrightarrow \Strat\bigl(X\smallsetminus X_0,Z\times(0,1)\bigr)\underset{\rm Lem~\ref{internal-R-invariance}}{\xra{~\simeq~}}\Strat(X\smallsetminus X_0,Z\bigr)
\]
is an equivalence between spaces.
Fix a compact smooth manifold $S$ with boundary, and a union of components of its boundary $S_0\subset \partial S$.
Fix a conically smooth map $X\times S_0 \xra{f_0} \sC(Z)$ for which $X_0\times S_0 = f_0^{-1}(\ast)$.
Use the notation
$\ov{f_0}_{|}\colon (X\smallsetminus X_0)\times S_0\to Z$ for the projection of the restriction, and the notation $\un{f_0}\colon X\times S_0 \xra{f_0} \sC(Z) \to [0,1)$ for the projection. 
We must show that the map of path components of spaces of maps relative to $f_0\mapsto \ov{f_0}_|$
\begin{equation}\label{surjective?}
\pi_0\Bigl(\Strat_{X_0}^{{\sf rel}f_0}\bigl(X\times S,\sC(Z)\bigr)\Bigr)\longrightarrow \pi_0\Bigl(\Strat^{{\sf rel}\ov{f_0}_{|}}\bigl((X\smallsetminus X_0)\times S,Z\bigr)\Bigr)
\end{equation}
is surjective. 
So fix a conically smooth map $(X\smallsetminus X_0)\times S \xra{\ov{f}} Z$ extending $\ov{f_0}_{|}$.

Let us first prove the desired surjectivity of~(\ref{surjective?}) for the case $Z=\ast$.
Because the codomain of~(\ref{surjective?}) is terminal in this case, the problem is to show that the domain is not empty.
This is to say that each conically smooth map $X\times S_0 \xra{f_0} [0,1)$ for which $X_0\times S_0 = f^{-1}_0\{0\}$ can be extended to a conically smooth map $X\times S\xra{f}[0,1)$ for which $X_0\times S_0 = f^{-1}\{0\}$.
Choose a collar-neighborhood $S_0\times[0,1)\subset S$, and choose a smooth partition of unity $\{\phi_0,\phi_1\}$ subordinate to the open cover $S_0\times[0,1)\cup S\smallsetminus S_0$ of the smooth manifold $S$.  
Choose a conically smooth map $X \xra{\alpha} [0,1)$ for which $X_0=\alpha^{-1}\{0\}$ -- see \S\ref{sec-reg}.
By design, the conically smooth map
\[
f\colon X\times S\longrightarrow[0,1)~,\qquad f(x,s) = \phi_0(s) \alpha(x)+\phi_1(s)f_0(x,s)
\]
is defined, and has the property that $X_0\times S = f^{-1}\{0\}$.
This concludes the verification that~(\ref{surjective?}) is surjective for the case $Z=\ast$.  

Knowing the case of $Z=\ast$, we can choose a conically smooth extension $\un{f}\colon X\times S\to [0,1)$ of $\un{f_0}$.  
The resulting conically smooth map
\[
f\colon X\times S\longrightarrow \sC(Z)~,\qquad (x,s)\mapsto \bigl(\ov{f}(x,s),\un{f}(x,s)\bigr)~\in~\ast\underset{Z\times\{0\}}\amalg Z\times[0,1)
\]
is a lift of $f$, thereby implying the desired surjectivity for the case of a general compact $Z$.

\end{proof}

\begin{cor}\label{cone-maps}
For $L$ and $Z$ compact conically smooth stratified spaces, there is a canonical identification of the space of based morphisms
\[
\Strat_\ast\bigl(\RR^i\times \sC(L), \RR^j\times \sC(Z)\bigr)~\simeq~ \underset{L'\underset{\sf cbl,open}\subset L}\coprod \Strat\bigl(L', Z\bigr)
\]
with the coproduct indexed by sub-stratified spaces $L' \hookrightarrow L$ whose inclusion is constructible and open.

\end{cor}
\begin{proof}
After Observation~\ref{internal-R-invariance}, we can assume $i=0=j$.  Note that constructible closed sub-stratified spaces are exactly the complements of constructible open sub-stratified spaces.
The result now follows from Lemma~\ref{lem.to-cone}, upon noticing that constructible closed sub-stratified spaces of $\sC(L)$ that include the cone-point are in bijection with those of $L$.

\end{proof}

\begin{cor}\label{cor.from-cones}
For each compact conically smooth stratified space $L$, and each conically smooth stratified space $X$, the fiber over $x\in X$ of the map of spaces
\[
\ev_{0} \colon \Strat\bigl(\RR^i\times \sC(L),X\bigr)\longrightarrow X~,
\]
given by restriction along the origin is canonically identified as the space
\[
\underset{L'\underset{\sf cbl,open}\subset L}\coprod \Strat\bigl(L', Z\bigr)
\]
where $\bigl(\RR^j\times \sC(Z),0\bigr) \hookrightarrow (X,x)$ is a basic neighborhood.

\end{cor}

\begin{proof}
After Observation~\ref{internal-R-invariance}, we can take $i=0$. 
We will prove that, for each neighborhood $x\in O\subset X$, the canonical map of spaces of based maps
\[
\Strat_\ast\bigl(\sC(L),O\bigr) \longrightarrow
\Strat_\ast\bigl(\sC(L),X\bigr) 
\]
is an equivalence. 
Taking the neighborhood $O$ to be the image of the open embedding $\RR^i\times \sC(Z) \hookrightarrow X$, then the corollary follows from the identification of Corollary~\ref{cone-maps}.

Fix a pair $(S, S_0\subset \partial S)$ consisting of a compact smooth manifold with boundary and a union of components of its boundary, equipped with a conically smooth map $\sC(L)\times S_0 \xra{f_0} O$ whose restriction to $\{0\}\times S_0$ is constantly $x$.  
We must show that the map of sets of path components of spaces of based maps relative to $f_0$
\[
\pi_0\Bigl(\Strat_\ast^{{\sf rel} f_0}\bigl(\sC(L)\times S,O\bigr)\Bigr)\longrightarrow \pi_0\Bigl(\Strat^{{\sf rel} f_0}_\ast\bigl(\sC(L)\times S,X\bigr)\Bigr)
\]
is surjective.
Let $\sC(L)\times S\xra{f} X$ be a conically smooth map extending $f_0$ whose restriction to $\{0\}\times S_0$ is constantly $x$.
In~\cite{aft1}, it is explained that there exists a conically smooth map $\sC(L) \times \RR \xra{\phi} \sC(L)$ for which 
\begin{itemize}
\item $\phi_t = 1_{\sC(L)}$ for each $t\leq 0$, 
\item $\phi_t(0) = 0$ for all $t\in \RR$,
\item $\phi_s(\sC(L)) \subset \phi_t(\sC(L))$ whenever $s>t\geq 0$,
\item the collection $\bigl\{\phi_t\bigr(\sC(L)\bigr) \subset \sC(L)\mid t\in \RR\bigr\}$ is a basis for the topology about $0\in \sC(L)$.  
\end{itemize}
By design, the preimage $f^{-1} O\subset \sC(L)\times S$ contains a neighborhood of $\sC(L)\times S_0 \bigcup \{0\}\times S$.  
So there is a conically smooth map $S\xra{\epsilon} \RR$ taking values in the non-negatives, whose restriction to $S_0$ is constantly $0$, for which the composition $\sC(L)\times S\xra{\phi_\epsilon} \sC(L)\times S \xra{f} X$ factors through $O$.
The family $[0,1]\ni t\mapsto f\circ \phi_{t\epsilon}$ witnesses a stratified homotopy from $f$ to $f\circ \phi_\epsilon$, which verifies the surjectivity.

\end{proof}

\begin{lemma}\label{lem.pre-consec}
For each compact conically smooth stratified space $L$, the diagram in $\Strat$
\[
\xymatrix{
\oC(\emptyset) \ar[rr]^-{\oC(\emptyset \hookrightarrow L)}  \ar[d]_-{\{1\}}
&&
\oC(L) \ar[d]^-{\{1\}}
\\
\oC^2(\emptyset) \ar[rr]^-{\oC^{2}(\emptyset\hookrightarrow L)}
&&
\oC^2(L)
}
\]
is a pushout.  

\end{lemma}

\begin{proof}
Fix a conically smooth stratified space $Z$.  
We must show that the diagram of spaces
\[
\xymatrix{
\Strat\bigl(\oC(\emptyset),Z\bigr)
&
\Strat\bigl(\oC(L),Z\bigr)    \ar[l]
\\
\Strat\bigl(\oC^2(\emptyset),Z\bigr)   \ar[u]
&
\Strat\bigl(\oC^2(L), Z\bigr)      \ar[l]\ar[u]
}
\]
is a pullback.  
In the diagram of the statement of the lemma, the horizontal maps have conically smooth regular neighborhoods, manifestly.  
It follows that the horizontal maps in the above diagram are Kan fibrations, and therefore the underlying spaces of the point-set fibers of the horizontal maps agree with the fibers.  
Fix a conically smooth map $\gamma \colon \Delta^1 \cong \oC^2(\emptyset) \to Z$, and denote the restriction $z_1\colon \ast = \oC(\emptyset) \xra{\{1\}} \oC^2(\emptyset) \xra{\gamma} Z$.  
We will argue that the map of relative mapping spaces
\begin{equation}\label{double-rel}
\Strat^{{\sf rel} \gamma}\bigl(\oC^2(L), Z\bigr)\longrightarrow \Strat^{{\sf rel}z_1}\bigl(\oC(L),Z\bigr)
\end{equation}
is an equivalence.
We will do this by repeatedly applying Corollary~\ref{cor.from-cones} to identify various spaces of relative maps from cones.

Denote the restriction $z_0\colon \ast = \oC(\emptyset) \xra{\{0\}} \oC^2(\emptyset) \xra{\gamma} Z$, and choose a basic neighborhood $z_0\in \RR^i\times \sC(K_0)\subset Z$.
Corollary~\ref{cor.from-cones} gives the canonical identification
\[
\Strat^{{\sf rel}z_0}\bigl(\oC^2(\emptyset), Z\bigr)~\simeq~\Strat\bigl(\oC(\emptyset),K_0\bigr)\amalg \Strat(\emptyset,K_0)~.
\]
Therefore, the conically smooth map $\gamma\colon \Delta^1=\oC^2(\emptyset) \to Z$ is classified either by a map $\w{z}_1\colon \ast = \oC(\emptyset)\to  K_0$, or by the unique map $\emptyset \to K_0$.  
We examine these two cases separately.

Suppose $\gamma$ is classified by $\emptyset\to K_0$.
In this case, the composite map 
\[
\Strat^{{\sf rel} \gamma}\bigl(\oC^2(L), Z\bigr) \longrightarrow \Strat^{{\sf rel}z_0}\bigl(\oC^2(L), Z\bigr)~\underset{\rm Cor~\ref{cor.from-cones}}\simeq~\underset{C'\underset{\sf cbl,open}\subset \oC(L)}\coprod \Strat(C',K_0)
\]
factors through those cofactors indexed by those $ C'\subsetneq \oC(L)$ that do not contain the cone-point.  
Each such constructible open is of the form $C'= L'\times (0,1]\subset \oC(L)$ for a unique constructible open $L'\subset L$.
Furthermore, this factorized map is an equivalence:
\[
\Strat^{{\sf rel} \gamma}\bigl(\oC^2(L), Z\bigr)\xra{~\simeq~}\underset{L'\underset{\sf cbl,open}\subset L}\coprod \Strat(L'\times(0,1),K_0)~\underset{\rm Lem~\ref{internal-R-invariance}}\simeq~\underset{L'\underset{\sf cbl,open}\subset L}\coprod \Strat(L',K_0)~.
\]
That~(\ref{double-rel}) is an equivalence in this case follows immediately from Corollary~\ref{cor.from-cones}.

Suppose $\gamma$ is classified by a map $\w{z}_1\colon \ast \to K_0$.  
In this case, the composite map 
\[
\Strat^{{\sf rel} \gamma}\bigl(\oC^2(L), Z\bigr) \longrightarrow \Strat^{{\sf rel}z_0}\bigl(\oC^2(L), Z\bigr)~\underset{\rm Cor~\ref{cor.from-cones}}\simeq~\underset{C'\underset{\sf cbl,open}\subset \oC(L)}\coprod \Strat(C',K_0)
\]
factors through the cofactor indexed by $\oC(L) \subset \oC(L)$, for this is the only constructible open sub-stratified space that contains the cone-point.  
Furthermore, this factorized map recognizes $\Strat^{{\sf rel}\gamma}\bigl(\oC^2(L),Z\bigr)$ as the fiber:
\[
\xymatrix{
\Strat^{{\sf rel} \gamma}\bigl(\oC^2(L), Z\bigr)   \ar[r]  \ar[d]
&
\Strat\bigl(\oC(L),K_0\bigr)  \ar[d]^-{{\sf ev}_\ast}
\\
\ast \ar[r]^-{\w{z}_1}  
&
\Strat(\ast,K_0);
}
\]
which is the assertion that the canonical map of spaces
\[
\Strat^{{\sf rel} \gamma}\bigl(\oC^2(L), Z\bigr) \xra{~\simeq~}\Strat^{{\sf rel}\w{z}_1}\bigl(\oC(L),K_0\bigr) 
\]
is an equivalence.
Choose a basic neighborhood $\w{z}_1\in \RR^j\times \sC(K_1)\hookrightarrow K_0$.
Such a choice determines a basic neighborhood $z_1\in \RR^{i+j+1}\times \sC(K_0)\subset Z$.
Corollary~\ref{cor.from-cones} then gives the two canonical identifications
\[
\Strat^{{\sf rel}\w{z}_1}\bigl(\oC(L),K_0\bigr)~\simeq~\underset{L'\underset{\sf cbl,open}\subset L}\coprod \Strat(L',K_1)~\simeq~\Strat^{{\sf rel}z_1}\bigl(\oC(L),Z\bigr)~.
\]
which respect the map~(\ref{double-rel}).
So~(\ref{double-rel}) is an equivalence in this case as well.

\end{proof}

\begin{remark}\label{rem.double-cones-unexpected}
While the diagram of Lemma~\ref{lem.pre-consec} exists in $\strat$, it is not a pushout therein. In fact, this homotopy colimit has the unusual feature that it is not a equivalent to a point-set colimit in any readily apparent model of spaces. We therefore see Lemma~\ref{lem.pre-consec} as emphasizing the role of the localization $\strat[\cJ^{-1}]\simeq \Strat$ in this ultimate characterization of $\infty$-categories in terms of stratified spaces.  

\end{remark}

\section{Exit-paths}
We define a functor $\exit\colon \Strat \to \Cat_\infty$.
For this, we make use of an important cosimplicial stratified space, and use the incarnation of $\infty$-categories as \emph{complete Segal spaces}.

\subsection{Complete Segal spaces}

The reader may wish to consult Notation \ref{not.delta} regarding our uses of the Greek letter delta.

\begin{definition}[After~\cite{rezk}]\label{def.complete-segal}
The $\infty$-category of \emph{complete Segal spaces} is the full $\infty$-subcategory
\[
\Psh^{\sf Segal,cplt}(\bdelta)~\subset~\Psh(\bdelta)
\]
consisting of those presheaves $\cC$ that satisfy the following two conditions.
\begin{enumerate}
\item {\bf Segal:} For each $0\leq k \leq p$, $\cC$ carries the diagram in $\bdelta$
\[
\xymatrix{
\{k\}  \ar[r]  \ar[d]
&
\{k<\dots<p\} \ar[d]
\\
\{0<\dots<k\}  \ar[r]
&
\{0<\dots<p\}
}
\]
to a pullback diagram of spaces; here $\{i< \ldots < i+j\}$ is the totally ordered sets with $j+1$ elements, and the indexing indicates the maps in the diagram.

\item {\bf Complete:} The functor $\Map(-,\cC)\colon \Psh(\bdelta)^{\op} \longrightarrow \Spaces$ carries the diagram
\[
\xymatrix{
\Delta[{\{0<2\}}]\amalg \Delta[{\{1<3\}}]  \ar[r] \ar[d]
&
\Delta[{\{0<1<2<3\}}] \ar[d]
\\
\Delta[0]\amalg \Delta[0]  \ar[r]
&
\Delta[0]
}
\]
to a pullback diagram of spaces.

\end{enumerate}

Regarding each finite non-empty linearly ordered set as a category in a standard manner, and thereafter as an $\infty$-category, gives a functor
\[
\bdelta \longrightarrow \Cat_\infty~.
\]
There results the restricted Yoneda functor
\begin{equation}\label{along-delta}
\Cat_\infty\longrightarrow \Psh(\bdelta)~.  
\end{equation}
\begin{theorem}[\cite{joyaltierney}]\label{thm.rezk}
The functor~(\ref{along-delta}) factors as an equivalence of $\infty$-categories
\[
\Cat_\infty\xra{~\simeq~} \Psh^{\sf Segal, cplt}(\bdelta)~.
\]

\end{theorem}

\end{definition}

\begin{remark}
While we attribute Theorem~\ref{thm.rezk} to~\cite{joyaltierney}, the result from that work was phrased as an equivalence of model categories.  
The present form of Theorem~\ref{thm.rezk} is along the lines of~\cite{toen} and~\cite{clark-chris}, which works internal to quasi-categories.  

\end{remark}

\subsection{Standard simplices}\label{sec.standard-simplices}

Recall the standard cosimplicial topological space $\bdelta \to {\sf Top}$ given on objects, and on morphisms, as 
\[
[p]\mapsto \Delta^p:=\Bigl\{\{0,\dots,p\}\xra{t} [0,1] \thinspace \Big| \thinspace \sum_i t_i = 1\Bigr\}~,\qquad ([p]\xra{\rho}[q])\mapsto   \bigl(t\mapsto (j\mapsto \sum_{\rho(i)=j} t_i)\bigr)
\]
where the rightmost sum is understood to take the value $0$ should the indexing set be empty.
\begin{definition}[$\Delta^\bullet$]\label{def.standard-simplices}
The \emph{standard} cosimplicial stratified space
\[
{\st}\colon \bdelta \longrightarrow \strat
\]
is given as
\[
[p]\mapsto \bigl(\Delta^p \to [p]~,~{}~t\mapsto {\sf Max}\{i\mid t_i\neq 0\}\bigr)~.
\]
\end{definition}

\begin{observation}\label{simplices-are-cones}
For each $p\geq 0$, there is an identification $\oC(\Delta^{p-1}) \cong \Delta^p$ between the closed cone and the standardly stratified simplex.  

\end{observation}

\begin{cor}\label{pre-segal}
For each $0\leq k \leq p$ the commutative diagram in $\Strat$
\[
\xymatrix{
\Delta^{\{k\}}  \ar[r]  \ar[d]
&
\Delta^{\{k<\dots<p-1\}} \ar[d]
\\
\Delta^{\{0<\dots<k\}}  \ar[r]
&
\Delta^{\{0<\dots<p\}}
}
\]
is a pushout.  

\end{cor}

\begin{proof}
Make use of Observation~\ref{simplices-are-cones}.
Use induction on $k$.
Apply Lemma~\ref{lem.pre-consec} for the base case $k=1$.  

\end{proof}

\begin{lemma}\label{standard-ff}
The composite functor $\bdelta \xra{\st} \strat \xra{c} \Strat$ is fully faithful.  

\end{lemma}

\begin{proof}
We must show that the map of spaces $\bdelta\bigl([p],[q]\bigr) \xra{\simeq}\Strat\bigl(\Delta^p, \Delta^q\bigr)$ is an equivalence.  
We do this by induction on $p$.
The assertion is clear for $p=0$, since the strata of $\Delta^q$ are contractible.  
For $p>0$, consider the diagram of spaces
\[
\xymatrix{
\bdelta\bigl([p],[q]\bigr) \ar[r] \ar[d]
&
\Strat\bigl(\Delta^p, \Delta^q\bigr)  \ar[d]
\\
\bdelta\bigl(\{0\},[q]\bigr)  \ar[r] 
&
\Strat\bigl(\Delta^{\{0\}}, \Delta^q\bigr)
}
\]
in which the vertical maps are the evident restrictions.
We have already argued that the bottom horizontal map is an equivalence, so it remains to argue as much for each fiber.
Fix a map $\Delta^{\{0\}} \xra{f} \Delta^q$.
Denote by $0\leq i \leq q$ the stratum containing the image of $f$.  
In the case $i=q$, the fiber of the lefthand vertical map over $i$ is terminal, while the fiber of the righthand vertical map over $f$ is also terminal, by inspection. 
So assume $i<q$.  
The fiber of the lefthand vertical map is 
\[
\bdelta\bigl(\{1<\dots<p\},\{i<\dots<q\}\bigr)~\simeq~ \underset{0\leq k\leq p}\coprod \bdelta\bigl(\{k+1<\dots<p\},\{i+1<\dots<q\}\bigr)~.
\]
Now, recognize $\Delta^j = \oC(\Delta^{j-1})$, so that Corollary~\ref{cor.from-cones} canonically identifies the fiber of the righthand vertical map as 
\[
\underset{D_0\underset{\sf cbl,open}\subset \Delta^{\{1<\dots<p\}}}\coprod\Strat\bigl(D_0,\Delta^{\{i+1<\dots<q\}}\bigr)~.
\]
Now recognize that each constructible open subspace $D_0\subset \Delta^{\{1<\dots<p\}}$ is of the form $(\Delta^{\{1<\dots<p\}})_{\geq k+1}$ for some $0\leq k\leq p$.
For $0\leq k <p$, the collapse map $\{1<\dots<p\}\to \{k+1<\dots<p\}$ induces a stratified map $\Delta^{\{1<\dots<p\}}\to \Delta^{\{k+1<\dots<p\}}$ that restricts to a stratified map $(\Delta^{\{1<\dots<p\}})_{\geq k+1}\to \Delta^{\{k+1<\dots<p\}}$ that is isomorphic to the projection $[0,1)^{k}\times \Delta^{\{k+1<\dots<p\}}\to \Delta^{\{k+1<\dots<p\}}$; and for $k=p$ there is a unique such projection.
Summarizing with Observation~\ref{internal-R-invariance}, we identify the map of fibers over $i\mapsto f$ as a map
\[
\underset{0\leq k\leq p}\coprod \bdelta\bigl(\{k+1<\dots<p\},\{i+1<\dots<q\}\bigr) \longrightarrow\underset{0\leq k\leq p}\coprod \Strat\bigl(\Delta^{\{k+1<\dots<p\}},\Delta^{\{i+1<\dots<q\}}\bigr)~.
\]
By inspection, this map respects the coproduct structure, and restricts on each cofactor as that induced by the functor $\bdelta \to \Strat$. 
That this map of fibers over $i\mapsto f$ is an equivalence between spaces follows from the inductive hypothesis.

\end{proof}

\subsection{Exit-paths}

Here we use the functor $\st: \bdelta\ra \Strat$ to define exit-path $\oo$-categories.

\begin{definition}[$\exit$]\label{exit}
The \emph{exit-path $\infty$-category} functor is the restricted Yoneda functor
\[
\exit\colon \Strat \xra{~{\sf y}~}\Psh(\Strat) \xra{~{\st}^\ast} \Psh(\bdelta)~.
\]
\end{definition}

Explicitly, the value $\exit(X)$ is the simplicial space $[p]\mapsto \Strat(\Delta^p, X)$, the values of which are incarnated as a Kan complex for which the set of $q$-simplices is $\strat(\Delta^p\times\Delta^q_e,X)$.

\begin{remark}
Lemma~\ref{exits-agree} shows that Definition~\ref{exit} is consistent with the exit-path $\infty$-category defined in Appendix \S A of~\cite{HA}. The difference is model specific.
Namely, in~\cite{HA} the exit-path $\infty$-category is given as a quasi-category, whereas here we present it as a complete Segal space (Corollary~\ref{exit-cplt-segal}).

\end{remark}

\begin{observation}\label{exit-products}
For each pair of stratified spaces $X$ and $X'$, the canonical map
\[
\exit(X\times X') \xra{~\simeq~} \exit(X)\times \exit(X')
\]
is an equivalence of simplicial spaces.  
This is direct from the equivalence $\exit(-) \simeq \Strat(\Delta^\bullet,-)$ and using that the stratified space $X\times X'$ is the product of $X$ and $X'$ in the $\infty$-category $\Strat$.  

\end{observation}

\begin{observation}\label{exit-smooth}
For each smooth manifold $M$, there is a canonical identification 
\[
\exit(M)~ \simeq ~M
\]
as the constant simplicial space at the underlying space of $M$.  
Indeed, the space of $p$-simplices of $\exit(M)$ is the Kan complex $\Strat(\Delta^p\times \Delta^\bullet_e, M)$.  
The degeneracy map from the space of $0$-simplices $\Sing(M) \simeq \Strat(\Delta^\bullet_e, M) \xra{\simeq} \Strat(\Delta^p\times \Delta^\bullet_e, M)$ is an equivalence of Kan complexes.  

\end{observation}

Next we identify the spaces of $0$- and $1$-simplices of $\exit(X)$.  
\begin{lemma}\label{exit-description}
Let $X = (X\to P)$ be a conically smooth stratified space.
The space of $0$-simplices of $\exit(X)$ is canonically identified
\[
\exit(X)_{|[0]}~\simeq~ \underset{p\in P}\coprod X_p
\]
as the coproduct of the underlying spaces of the strata of $X$.  
For each pair of strata $X_p, X_{p'}\subset X$, the space of $1$-simplices from $X_p$ to $X_{p'}$ is canonically identified
\[
\bigl(X_p\times X_{p'}\bigr) \underset{\exit(X)_{|\partial [1]}} \times \exit(X)_{|[1]}~\simeq~ \Link_{X_p}(X)_{p'}
\]
as the underlying space of the $p'$-stratum of the link of the $p$-stratum.  

\end{lemma}

\begin{proof}
The first statement is direct from the definition of $\Strat$.  Namely, $\Strat(\ast, X)$ is the Kan complex $[q]\mapsto \strat\bigl(\Delta^q_e, X\bigr)$.  Because $\Delta^q_e$ is a connected smooth manifold for each $q\geq 0$, then any map from it to $X$ factors through a stratum of $X$.  More precisely, there is an identification of Kan complexes
\[
\Strat(\ast, X) ~\cong~ \Sing\Bigl(\underset{p\in P} \coprod X_p\Bigr)~.
\]

Let $x_0\in X$ be a point, and choose a basic neighborhood $x_0\in \RR^i\times \sC(L) \subset X$.  
Using that $\Delta^q = \oC(\Delta^{q-1})$, Corollary~\ref{cor.from-cones} canonically identifies the fiber of the evaluation map $\exit(X)_{|[q]} \xra{{\sf ev}_{\{0\}}} \exit(X)_{|[0]}$ over $x_0\in X$ as the space
\[
\underset{0\leq k \leq q} \coprod \exit(L)_{|\{k+1<\dots<q\}}~.
\]
As the case $q=1$, 
the second statement then follows from the first.

\end{proof}

\begin{cor}\label{exit-cplt-segal}
The functor $\exit\colon \Strat \to \Psh(\bdelta)$ takes values in complete Segal spaces.
\end{cor}
\begin{proof}
Corollary~\ref{pre-segal} gives that, for each $0\leq k \leq p$, the diagram of spaces
\[
\xymatrix{
\Strat(\Delta^{\{0<\dots<p\}}, X)   \ar[r]    \ar[d]
&
\Strat(\Delta^{\{k<\ldots<p\}},X)    \ar[d]
\\
\Strat(\Delta^{\{0<\dots<k\}},X)   \ar[r]
&
\Strat(\Delta^{\{k\}},X)
}
\]
is a pullback.
This verifies the Segal condition for $\exit(X)$.

Notice that a conically smooth map $\Delta^2 \xra{\sigma} X$ factors through a single stratum of $X$ if and only if both $\sigma_{|\Delta^{\{0\}}}$ and $\sigma_{|\Delta^{\{2\}}}$ factor through the same stratum of $X$.
It follows that the only retracts in the Segal space $\exit(X)$ are equivalences.
The completeness of $\exit(X)$ follows.

\end{proof}

\begin{convention}
By way of Corollary~\ref{exit-cplt-segal}, we will henceforth regard $\exit$ as a functor to $\oo$-categories, in their incarnation as complete Segal spaces~(as established by~\cite{rezk}).
\end{convention}

The colimits examined in~\S\ref{sec.colims-in-Strat} are preserved by the exit-path functor.
\begin{prop}\label{exit-facts}
The exit-path functor $\exit:\Strat \ra \Cat_\infty$ preserves the following colimits.

\begin{enumerate}

\item
For each covering sieve $\cU$ of a conically smooth stratified space $X$, the composite functor
\[
\cU^{\tr}\to \Strat \xra{~\exit~} \Cat_\infty
\]
is a colimit diagram. 

\item
For each deepest stratum $X_0\subset X$ of a stratified space, the functor $\exit$ carries the pushout diagram in $\strat$
\[
\xymatrix{
\Link_{X_d}(X)  \ar@{^{(}->}[r]  \ar[d]
&
\Unzip_{X_d}(X)  \ar[d]
\\
X_d \ar@{^{(}->}[r]
&
X
}
\]
to a pushout diagram among $\oo$-categories.

\item
For each compact conically smooth stratified space $L$, the functor $\exit$ carries each diagram
\[
\xymatrix{
\oC(\emptyset) \ar[rr]^-{\oC(\emptyset \hookrightarrow L)}  \ar[d]_-{\{1\}}
&&
\oC(L) \ar[d]^-{\{1\}}
\\
\oC^2(\emptyset) \ar[rr]^-{\oC^{2}(\emptyset\hookrightarrow L)}
&&
\oC^2(L)
}
\]
to a colimit diagram of $\oo$-categories.

\end{enumerate}

\end{prop}

\begin{proof}
{\bf (1)}
We use the following sufficient condition for identifying certain colimits in $\Cat_\infty\simeq \Psh^{\sf Segal,cplt}(\bdelta)$:
\begin{itemize}
\item[~] 
A functor $\cU^{\tr} \to \Cat_{\oo}$ is a colimit diagram if the composite functor
\[
\cU^{\tr}\longrightarrow \Cat_{\oo} \longrightarrow \psh(\bdelta)
\]
is a colimit diagram.
\end{itemize}
Let $\cU\subset \strat_{/X}$ be a covering sieve.
It is enough to show that, for each $p\geq 0$, the functor 
\[
\cU^{\tr}\to \strat \xra{\exit} \Cat_\infty \xra{{\sf ev}_{[p]}}\Spaces
\]
is a colimit diagram.  
We do this by induction on $p$.

For each conically smooth stratified space $Y=(Y\to Q)$, Lemma~\ref{exit-description} offers the identification of spaces: $\exit(Y)_{|[0]} \simeq \Strat(\ast,Y)\simeq \underset{q\in Q} \coprod Y_q$.  
So the $p=0$ assertion is that 
\[
\cU^{\tr} \longrightarrow  \strat\xra{(Y\to Q) \mapsto \underset{q\in Q} \coprod Y_q} \man \longrightarrow \spaces
\]
is a colimit diagram, where the last arrow is the underlying space functor.  
After Lemma~\ref{lem.opens-cover}, it is enough to argue that the composite of the first two arrows $\cU \to \man$ generates a covering sieve.  
This is to say that, for each $q\in Q$, an open cover of $Y=(Y\to Q)$ restricts as an open cover of the stratum $Y_q$, which is directly the case.  
This establishes the base case of the induction.

Let $x_0\in X$ be a point, and choose a basic neighborhood $x_0\in \RR^i\times \sC(L) \subset X$.  
Using that $\Delta^p = \oC(\Delta^{p-1})$, Corollary~\ref{cor.from-cones} canonically identifies the fiber of the evaluation map $\exit(X)_{|[p]} \xra{{\sf ev}_{\{0\}}} \exit(X)_{|[0]}$ over $x_0\in X$ as the space
\[
\underset{0\leq k \leq p} \coprod \exit(L)_{|\{k+1<\dots<p\}}~.
\]
For each $q\in Q$, an open cover of $Y=(Y\to Q)$ restricts as an open cover of the link $\Link_{Y_q}(Y)$.  
This supplies the inductive step.

\smallskip

{\bf (2)}
Through the definition of a stratified space, choose an open hopercover of $X$ by basics.
Because the square diagram is comprised of, in particular, continuous maps, taking preimages of each term in this open cover determines open covers of each term in this square diagram.  
Through~{\bf (1)}, we recognize the diagram of simplicial spaces 
\begin{equation}\label{01}
\xymatrix{
\exit\bigl(\Link_{X_d}(X)\bigr)  \ar@{^{(}->}[r]  \ar[d]
&
\exit\bigl(\Unzip_{X_d}(X) \bigr)  \ar[d]
\\
\exit(X_d) \ar@{^{(}->}[r]
&
\exit(X)
}
\end{equation}
as a colimit of squares among simplicial spaces of the form
\begin{equation}\label{02}
\xymatrix{
\exit\bigl(\RR^i\times L\bigr)  \ar@{^{(}->}[r]  \ar[d]
&
\exit\bigl(\RR^i\times \RR_{\geq 0}\times L \bigr)  \ar[d]
\\
\exit(\RR) \ar@{^{(}->}[r]
&
\exit\bigl(\RR^i\times \sC(L)\bigr).
}
\end{equation}
We are therefore reduced to showing that the diagram~(\ref{02}) is a pushout.
Because $\exit$ factors through $\Strat$, Observation~\ref{internal-R-invariance} reduces to the case that $i=0$.
By calculation of the mapping spaces in $\exit(\sC(L))$ from Lemma~\ref{exit-description}, we have a natural equivalence $\exit\bigl(\sC(L)\bigr)\simeq \exit(L)^{\tl}$. By Observation~\ref{exit-products}, there is a further equivalence $\exit\bigl(L\times [0,1)\bigr) \simeq \exit(L)\times \exit\bigl([0,1)\bigr)\simeq \exit(L)\times [1]$.
The result now follows, by identifying the pushouts.

\smallskip

{\bf (3)} 
Point~{\bf(2)} gives an identification $\exit\bigl(\oC(K)\bigr)\simeq \exit(K)^{\tl}:= \ast\underset{\exit(K)\times\{0\}} \amalg \exit(K)\times[1]$.  
The result follows by calculating a double pushout in $\infty$-categories.

\end{proof}

For the next result we reference, for $X= (X\to P)$ a conically stratified space, the exit-path quasi-category $\Sing^P(X)$ defined in Appendix A of~\cite{HA}; it is the sub-simplicial set of the singular simplicial set $\Sing(X)$ consisting of those maps $\Delta^p\ra X$ which preserve the stratification. We reference also the enter-path $\oo$-category ${\sf Entr}(X)$ of \cite{aft1}; it is the $\oo$-category $\Bsc_{/X}$ of basic singularity types embedded into $X$.
\begin{lemma}\label{exits-agree}
For each conically smooth stratified space $X = (X\to P)$, there are equivalences of $\oo$-categories
\[
\exit(X)~\simeq~ \Sing^P(X)~\simeq~{\sf Entr}(X)^{\op}~.
\]
\end{lemma}

\begin{proof}

Direct from the definition of a stratified space, $X$ is \emph{conically stratified} in the sense defined in Definition A.5.5 of~\cite{HA}.
Therefore, the simplicial set $\Sing^P(X)$ is a quasi-category (Theorem~A.6.4 of~\cite{HA}). 
The equivalence $ \Sing^P(X)~\simeq~{\sf Entr}(X)^{\op}$ is Corollary 1.2.10 of \cite{aft1}. 
So it suffices to prove the first equivalence.
We first construct a functor $\exit(X)\ra \sing^P(X)$, which is essentially given by forgetting the condition of conical smoothness on the source.
That is, the complete Segal space associated to the quasi-category $\sing^P(X)$ is equivalent to the simplicial space 
\[
\exit^{C^0}(X)  \colon  \bdelta^{\op}\overset{\st}\longrightarrow (\Strat^{C^0})^{\op} \xra{\Strat^{C^0}(-,X)}\spaces~,
\]
where $\Strat^{C^0}$ is the coherent nerve of the simplicial category of $C^0$ stratified spaces, with simplicial enrichment given as $\Strat^{C^0}(K,X) := \strat^{C^0}(K\times \Delta^\bullet_e,X)$.
Through this identification, the forgetful functor $\strat \to \strat^{C^0}$ results in a map between simplicial spaces
\begin{equation}\label{03}
\exit(X)\longrightarrow \exit^{C^0}(X)~.
\end{equation}
We now prove this map is an equivalence between simplicial spaces.

Now, Proposition~\ref{exit-facts} gives that the domain of~(\ref{03}) satisfies descent with respect to open covers.
Theorem A.7.1 (the Seifert--van Kampen Theorem for exit-paths) of \cite{HA} gives that the codomain of~(\ref{03}) also satisfies descent with respect to open covers.  
The problem of showing~(\ref{03}) is an equivalence is therefore reduced to the case that $X$ is a basic: $X\cong \RR^k\times \sC(L)$.
Because both $\exit$ and $\exit^{C^0}$ are local with respect to projections off of Euclidean space, we are further reduced to the case that $X=\sC(L)$ is the open cone on a compact stratified space.

We now prove the desired equivalence by induction on depth.
Suppose $X$ has depth zero.  
In this case, Observation~\ref{exit-smooth} identifies $\exit(X)$ as the underlying $\infty$-groupoid associated to the underlying topological space of $X$.  
Direct from its definition, $\exit^{C^0}(X)$ is the complete Segal space associated to the simplicial set $\sing(X)$, the singular simplicial set on $X$.  
Through these identifications, the functor~(\ref{03}) is identified as an equivalence.  

We now establish the inductive step.
Above, we reduced to the case that $X\cong \sC(L)$ is the open cone on a compact stratified space.  
Necessarily, the depth of $L$ is strictly less than that of $X$.  
The inductive hypothesis gives that the functor~(\ref{03}) is an equivalence for the stratified space $L$.  
Proposition~\ref{exit-facts}(2) gives a canonical equivalence $\exit\bigl(\sC(L)\bigr)\simeq \exit(L)^{\tl}$ between $\infty$-categories.  
It therefore suffices to establish a likewise canonical equivalence $\exit^{C^0}\bigl(\sC(L)\bigr)\simeq \exit^{C^0}(L)^{\tl}$ between $\infty$-categories.

This is the problem of showing the cone-point $\ast\in \exit^{C^0}\bigl(\sC(L)\bigr)$ is initial.
So, for each point $z\in \sC(L)$, we must show that the space of morphisms $\Map_{ \exit^{C^0}(\sC(L))}(\ast,z)$ is contractible. 
It is sufficient to show that, for $K$ a compact smooth manifold, the space of $K$-points of this space is connected.
We will show this using a form of the Alexander trick (which is available since we are in the situation of $C^0$ exit-paths).
Such a $K$-point is the datum of a map between $C^0$ stratified spaces $f \colon \ov{\sC}(K) \to \sC(L)$ that preserves cone-points together with an identification of the restriction to the equator $f_{|K}\colon K \to \sC(L)$ as the constant map at the point $z\in \sC(L)$.  
Note the $K$-point $f_0\colon \ov{\sC}(K) \to \Delta^1 \cong [0,1] \xra{s\mapsto s\cdot z} \sC(L)$, where, here, $s\cdot z$ makes reference to the action of the (non-unital) monoid $[0,1]$ on the open cone $\sC(L)$.  
Choose an arbitrary such $K$-point $f$.  
Consider the $([0,1]\times K)$-point
\[ 
[0,1]\times \ov{\sC}(K) \ni (t,[s,x])\mapsto f_t(s,x)\in \sC(L)
\]
where $f_t([s,x])= \ast$ is constantly the cone-point for $s=0$ while
\[
f_t([s,x]) = \bigl[ f^1([s,x]), t\cdot f^2([s,x])\bigr]~,
    \quad \text{for $0<s< t$}~,
\]
and
\[
f_t([s,x]) = \bigl[ f^1([1,x]), s\cdot f^2([1,x])\bigr]~,
    \quad \text{for $t\leq s \leq 1$}   ~;
\]
here, we have written the values of $f_t$ via the coordinates of the open cone $\sC(L)$ as it is a quotient of $[0,1)\times L$.  
This $([0,1]\times K)$-point of $\sC(L)$ demonstrates $f$ and $f_0$ as in the same connected component of the space of such $K$-points.  
We conclude that the space of such $K$-points is connected, which completes this proof.

\end{proof}

\begin{remark}
The reader can compare the earlier proof that $\exit(X)$ is a complete Segal space with the subdivision-based proof from \S A.6 \cite{HA} that the exit-path simplicial set $\sing^P(X)$ is a quasi-category. 
This indicates our point of view that sheaves on conically smooth stratified spaces offer a navigable avenue for constructing $\oo$-categories by hand from geometry, where existence of regular neighborhoods often makes open covers more manageable than subdivisions.
\end{remark}

We lastly have the following connection to constructible sheaves.

\begin{cor}\label{exit-sheaves}For each conically smooth stratified space $X$ there is an equivalence
\[
\Fun(\exit(X),\spaces)~\simeq~ \shv^{\sf cbl}(X)
\]
between copresheaves on the exit-path $\infty$-category of $X$ and sheaves on $X$ which are constructible with respect to the given stratification.  

\end{cor}
\begin{proof}
This follows after Lemma~\ref{exits-agree}, since the corresponding result for $\Sing^P(X)$ is proved in \S A.9 from~\cite{HA} and for ${\sf Entr}(X)^{\op}$ in Theorem 1.2.5 of \cite{aft1}.
\end{proof}

We record the next very useful result, that exit-paths localize along refinements, and give a proof based on the relation with constructible sheaves. The statement could also be deduced from the localization statement for enter-paths, proved in~\cite{aft1} (Proposition 1.2.13).

\begin{theorem}\label{exit.paths.localize}
Let $X  \to Y$ be a conically smooth map between stratified spaces.
If this map is a refinement, then the associated functor between exit-path $\infty$-categories
$\exit(X) \to  \exit(Y)$ 
is a localization.
\end{theorem}

\begin{proof}
Because each refinement $X\to Y$ is, in particular, a homeomorphism, the induced restriction functor $\Shv(Y) \to \Shv(X)$ is an equivalence between $\infty$-categories.
By definition of the full $\infty$-subcategory $\Shv^{\sf cbl}(Z)\subset \Shv(Z)$, the composite fully faithful functor
\[
\Shv^{\sf cbl}(Y) \hookrightarrow \Shv(Y)~\simeq~\Shv(X)
\]
factors through the full $\infty$-subcategory $\Shv^{\sf cbl}(X)$.  
Corollary~\ref{exit-sheaves} states a canonical equivalence between $\infty$-categories $\PShv\bigl(\exit(Z)^{\op}\bigr) = \Shv^{\sf cbl}(Z)$, which is evidently contravariantly functorial in the stratified space $Z$.
We conclude that the restriction functor
\[
\PShv\bigl(\exit(Y)^{\op}\bigr) \longrightarrow  \PShv\bigl(\exit(X)^{\op}\bigr)
\]
is fully faithful.
Now, apply the Lemma~\ref{abstract.condition} to the case that  
$\cD = \exit(Y)^{\op}$,
$P$ is the opposite of the stratifying poset for $Y$, 
$\cC$ is the terminal localization through which $\exit(X)^{\op} \to \exit(Y)^{\op}$ factors.
This lemma then gives that the functor
$\cC \to \exit(Y)^{\op}$
is an equivalence.
We conclude that the functor $\exit(X) \to \exit(Y)$ is a localization, as desired.

\end{proof} 

The next two assertions support Theorem~\ref{exit.paths.localize}; the first of which is direct from the definition of finality.

\begin{observation}\label{easy.final.criterion}
Let $\cX\to \cY$ be a functor between $\infty$-categories.
If the restriction functor
$\PShv(\cY) \to \PShv(\cX)$
is fully faithful, then the functor $\cX \to \cY$ is final.
\end{observation}

\begin{lemma}\label{abstract.condition}
Let  $\cC \to \cD\to P$
be a sequence of conservative functors among $\infty$-categories.
Suppose $P$ is a poset such that, for each $p\in P$, the subposet $P_{/p}$ is finite.
If the restriction functor $\PShv(\cD) \to \PShv(\cC)$
is fully faithful, then the functor $\cC\to \cD$ is an equivalence.  

\end{lemma}

\begin{proof}
The canonical functor
\[
\colim\bigl(P \xra{p \mapsto \cC_{/p}} \Cat_{\oo/P}\bigr)\xra{~\simeq~} \cC
\]
is an equivalence, and likewise for $\cD$ in place of $\cC$.  
Consequently, the functor $\cC\to \cD$ is an equivalence if, for each $p\in P$, the functor between $\infty$-overcategories $\cC_{/p} \to \cD_{/p}$ is an equivalence.  
With the assumption on $P$, we are therefore reduced to the case that $P$ is finite.
We proceed by induction on the depth of $P$, which is necessarily finite.

In the case that the depth of $P$ is 0, both $\cC$ and $\cD$ are spaces.  
Observation~\ref{easy.final.criterion} reveals that $\cC \to \cD$ is final.
By Quillen's Theorem A, this means each fiber of this map between spaces is contractible.  
We conclude that the functor $\cC\to \cD$ is an equivalence, as desired.

Let $P_0 \subset P$ be the full subposet consisting of the minima of $P$.  
Denote $\cC_0 := \cC_{|P_0}$ and $\cD_0 := \cD_{|P_0}$.
Denote 
\[
\Link_{\cC_0}(\cC) := \Ar(\cC)^{|\cC_{|P_0}}_{|\cC_{P \smallsetminus P_0}}
\]
and likewise for $\Link_{\cD_0}(\cD)$.
By evaluation at targets, the functor $\cC\to \cD$ over $P$ determines a functor
\[
\Link_{\cC_0}(\cC) \longrightarrow \Link_{\cD_0}(\cD)
\]
over $P \smallsetminus P_0$.
By direct inspection, this functor is again conservative, and the functors to $P \smallsetminus P_0$ are again conservative.  
By definition, the depth of the poset $P\smallsetminus P_0$ is strictly less than that of $P$.
By induction on depth, each of the functors 
\[
\Link_{\cC_0}(\cC) \to \Link_{\cD_0}(\cD)
\qquad
\text{ and }
\qquad
\cC_{|P \smallsetminus P_0} \to  \cD_{|P \smallsetminus P_0}
\qquad
\text{ and }
\qquad
\cC_{|P_0} \to \cD_{|P_0}
\]
is an equivalence.
It follows that $\cC\to \cD$ is essentially surjective and fully faithful, as desired.

\end{proof}

\section{Striation sheaves}\label{sec.striation}

In this section we introduce \emph{striation sheaves}.  These are constructible sheaves on conically smooth stratified spaces that satisfy an additional locality with respect to blow-ups along closed substratified spaces and iterated cones.

\subsection{Localities for stratified spaces}
We define striation sheaves.  

\begin{definition}\label{def.striation}
The $\infty$-category of \emph{striation sheaves} is the full $\infty$-subcategory of space-valued presheaves on $\strat$
\[
\Stri~\subset~\Psh(\strat)
\]
consisting of those $\cF$ that satisfy the following properties.
\begin{enumerate}
\item {\bf Sheaf:} For each covering sieve $\cU \subset \strat_{/K}$, the restriction of $\cF$ along the adjoint diagram $\cU^{\tr} \to \strat$ is a limit diagram of spaces.

\item {\bf Constructible:} For each conically smooth stratified space $K$, the value of $\cF$ on the projection $K\times \RR \to K$ is an equivalence between spaces.

\item {\bf Cone-local:} For each compact conically smooth stratified space $L$, the value of $\cF$ on the diagram
\[
\xymatrix{
L  \ar[r]  \ar[d]
&
L\times \RR_{\geq 0}  \ar[d]
\\
\ast  \ar[r]
&
\sC(L)
}
\]
is a pullback diagram of spaces.  

\item {\bf Consecutive:} For each $p>0$, the value of $\cF$ on the diagram of stratified spaces
\[
\xymatrix{
\Delta^{\{1\}}  \ar[r] \ar[d]
&
\Delta^{\{1<\dots<p\}} \ar[d]
\\
\Delta^{\{0<1\}}  \ar[r]
&
\Delta^{\{0<\dots<p\}}
}
\]
is a pullback diagram of spaces. 

\item {\bf Univalent:}  The value of $\cF$ on the diagram of stratified spaces
\[
\xymatrix{
\Delta^{\{0<2\}}\amalg \Delta^{\{1<3\}}  \ar[r] \ar[d]
&
\Delta^{\{0<1<2<3\}} \ar[d]
\\
\ast\thinspace\amalg \thinspace\ast \ar[r]
&
\ast
}
\]
is a pullback diagram of spaces.

\end{enumerate}

\end{definition}

In~\S\ref{sec.trans-sheaves}, we describe a technique for constructing striation sheaves from point-set data.

\begin{remark}\label{strong-striation}
The latter four of our conditions can be naturally strengthened, where the cumulative weaker versions are equivalent to the cumulative stronger versions. We will not use these strengthened conditions, so we omit the proof.
\begin{itemize}
\item[(2)] {\bf Constructible (strong version):}
The value of $\cF$ on each stratified homotopy equivalence $X\to Y$ is an equivalence between spaces.

\item[(3)] {\bf Cone-local (strong version):} For each conically smooth stratified space $K$ with a deepest stratum $K_d \subset K$, the value of $\cF$ on the blow-up square
\[
\xymatrix{
{\sf Link}_{K_d}(K)  \ar[r]  \ar[d]
&
\unzip_{K_d} (K)  \ar[d]  
\\
K_d  \ar[r]
&
K
}
\]
is a pullback diagram of spaces.  

\item[(4)] {\bf Consecutive (strong version):} For each compact conically smooth stratified space $L$, the value of $\cF$ on the diagram
\[
\xymatrix{
\oC(\emptyset)  \ar[r]  \ar[d]
&
\oC(L) \ar[d]
\\
\oC^2(\emptyset)  \ar[r]
&
\oC^2(L) 
}
\]
is a pullback diagram of spaces.

\item[(5)] {\bf Univalent (strong version):} 
Consider the simplicial set $E[p]$ that is the nerve of the minimal connected groupoid whose underlying set is $\{0,\dots,p\}$.
By way of $\bdelta\xra{\Delta^\bullet}\strat$ there results a simplicial stratified space $E^p$.  
We extend $\cF$ to simplicial stratified spaces via right Kan extension $\cF(Z_\bullet):=\underset{\Delta^q\to Z_\bullet}\limit \cF(\Delta^q)$.  
For each conically smooth stratified space $K$, and for each $p\geq 0$, the value of $\cF$ on the projection $K\times E^p\to K$ is an equivalence.

\end{itemize}
Each of these conditions specializes to its weaker version, and so a sheaf on $\strat$ satisfying these strong conditions is in particular a striation sheaf.
Conversely, every striation sheaf automatically satisfies these strengthened conditions. 

\end{remark}

\begin{remark}
Each of the defining properties of a striation sheaf $\cF$ gives a conceptual reduction of information.
\begin{itemize}
\item The {\bf sheaf} condition implies $\cF$ is determined by its values on basic singularity types: $\RR^i\times \sC(L)$.

\item The {\bf constructible} condition implies $\cF$ factors through $\Strat$.
Together these conditions imply $\cF$ is determined by its values on cones.  

\item The {\bf cone-local} condition implies the value of $\cF$ near a singularity is determined by the value of $\cF$ on the link of the singularity.
Together these conditions imply $\cF$ is determined by its values on standard simplices, $\Delta^p = \oC^{p+1}(\emptyset)$, which is to say $\cF$ is equivalent data as a simplicial space.

\item The {\bf consecutive} condition implies equivalences among such $\cF$ are detected by their values on $\ast$ and on $\Delta^1$.
Together these four conditions imply the simplicial space characterizing $\cF$ is a Segal space.  

\item The {\bf univalent} condition implies equivalences among such $\cF$ are detected by their values on $\Delta^1$ together with surjectivity of components of their values on $\ast$.  
Together with the preceding conditions, this implies that the Segal space characterizing $\cF$ is \emph{complete}.

\end{itemize}

\end{remark}

\subsection{Characterization}

In proving our characterization of striation sheaves, we will need the following result, which asserts that the value of a striation sheaf $\cF$ on the cone on a conically smooth stratified space $L$ is determined by the values of $\cF$ on the cones on a cover of $L$.

\begin{lemma}\label{cones-of-covers}
Let $\ov{\cU}$ be a collection of conically smooth maps $\DD^j\times \oC(W) \to L$ for which the collection $\cU$ consisting of the precompositions with interiors $\RR^j\times \sC(W) \to L$ forms an open hypercover of $L$.
For each cone-local constructible sheaf $\cF$, and each $p\geq 0$, the canonical map of spaces
\[
\cF\bigl(\oC^p(L)\bigr) \xra{~\simeq~} \underset{\ov{U}\in \ov{\cU}}\limit \cF\bigl(\oC^p(\ov{U})\bigr)
\]
is an equivalence.

\end{lemma}

\begin{proof}
We proceed by induction on $p\geq 0$.
The case $p=0$ is immediate from Lemma~\ref{regularbasis}.
We assume the case $p$ and deduce the case $p+1$ from the following sequence of canonical equivalences:
\begin{eqnarray}
\nonumber
\cF\bigl(\sC(\oC^p(L))\bigr)
&
\xra{\simeq}
&
\cF(\ast)\underset{\cF\bigl(\oC^p(L)\bigr)} \times \cF\bigl(\oC^p(L)\times \RR_{\geq 0}\bigr)
\\
\nonumber
&
\xra{\simeq}
&
\cF(\ast) \underset{\underset{\ov{U}\in \ov{\cU}} \limit \cF\bigl(\oC^p(\ov{U})\bigr)} \times \underset{\ov{U}\in \ov{\cU}} \limit \cF\bigl(\oC^p(\ov{U})\times \RR_{\geq 0}\bigr)
\\
\nonumber
&
\simeq
&
\underset{\ov{U}\in \ov{\cU}} \limit \Bigl(\cF(\ast) \underset{\cF\bigl(\oC^p(\ov{U})\bigr)} \times \cF\bigl(\oC^p(\ov{U})\times \RR_{\geq 0}\bigr)\Bigr)
\\
\nonumber
&
\xla{\simeq}
&
\underset{\ov{U}\in \ov{\cU}} \limit \cF\bigl(\sC(\oC^{p}(\ov{U}))\bigr)
\end{eqnarray}
The first map is an equivalence because $\cF$ is cone-local.
The second map is an equivalence by the inductive hypothesis for $p$, using that the collection $\bigl\{\oC^{p}(\ov{U})\bigr\}$ of subspaces of $\oC^p(L)$, as well as the collection $\bigl\{\oC^{p}(\ov{U})\times \RR_{\geq 0}\bigr\}$ of subspaces of $\oC^p(L)\times \RR_{\geq 0}$, is of the form to which the statement of the lemma applies.
The third map is an equivalence because, formally, limits commute.
The fourth map is an equivalence because $\cF$ is cone-local.
Finally, Lemma~\ref{homotopy-boundary} gives that the canonical map $\cF\bigl(\sC(\oC^p(Z))\bigr)\xra{\simeq}\cF\bigl(\oC^{p+1}(Z)\bigr)$ is an equivalence for each conically smooth stratified space $Z$.

\end{proof}

We now prove the main result of this section: $\oo$-categories are striation sheaves. Recall the functor ${\st}\colon \bdelta \to \Strat$ from~\S\ref{recall}.
There results an adjunction
\begin{equation}\label{exit-adjunction}
{\st}^\ast \colon \Psh(\Strat) \rightleftarrows \Psh(\bdelta)\colon {\st}_\ast
\end{equation}
given by restriction and right Kan extension.  
Explicitly, this right Kan extension evaluates on $\cF$ as
\begin{equation}\label{right-kan-exit}
{\st}_\ast \cF\colon X\mapsto  \Map_{\Psh(\bdelta)}\bigl(\exit(X),\cF\bigr)~.
\end{equation}
Through Theorem~\ref{cbl-sheaves}, we identify the cone-local constructible sheaves $\shv^{\sf cone,cbl}(\strat)\subset \shv(\Strat)$ as a full $\infty$-subcategory.

\begin{lemma}\label{cone-sheaves}
The adjunction~(\ref{exit-adjunction}) restricts as an equivalence of $\infty$-categories
\[
\shv^{\sf cone, \cbl}(\strat)~\simeq~\Psh(\bdelta)~.
\]

\end{lemma}
\begin{proof}
Lemma~\ref{standard-ff} grants that ${\st}_\ast$ is fully faithful.
It remains to verify that the unit $\cF \to {\st}_\ast {\st}^\ast \cF$ is an equivalence if and only if $\cF$ is a cone-local constructible sheaf.

We first point out that ${\st}_\ast$ takes values in cone-local constructible sheaves.
Constructibility follows upon inspecting~(\ref{right-kan-exit}) because the projection $X\times \RR \to X$ induces an equivalence $\exit(X\times \RR) \xra{\simeq} \exit(X)$, by definition of $\Strat$.
By inspecting~(\ref{right-kan-exit}), Proposition~\ref{exit-facts} gives that ${\st}_\ast\cF$ is cone-local.

It remains to prove that this unit map $\cF \to {\st}_\ast {\st}^\ast \cF$ is an equivalence whenever $\cF$ is a cone-local constructible sheaf.
So let $\cF$ be a cone-local constructible sheaf.
By definition, each conically smooth stratified space admits a hypercover by basic singularity types; so it is enough to show that, for each integer $i$ and each compact conically smooth stratified space $Z$, that the unit map $\cF\bigl(\RR^i\times \sC(Z)\bigr) \to {\st}_\ast {\st}^\ast \cF\bigl(\RR^i\times \sC(Z)\bigr)$ is an equivalence between spaces.
Because $\cF$ is constructible, then it is enough to show this for the case $i=0$.

So let $Z$ be a compact conically smooth stratified space.
Consider the maximal $p\geq 0$ for which $Z\cong \oC^p(L)$ for some compact conically smooth stratified space $L$.
We proceed by downward induction on $p$.
In the case $p > {\sf dim}(Z)$ then necessarily $L=\emptyset$ and $Z = \oC^{d+1}(\emptyset) = \Delta^{d}$, so the unit is an equivalence because ${\st}_\ast$ is fully faithful.
We now give the inductive step.

From the proof of Lemma~\ref{regularbasis}, there is a collection $\ov{\cU}$ of conically smooth maps $\DD^j\times \oC(W) \to L$ for which the collection $\cU$ consisting of the precompositions with interiors $\RR^j\times \sC(W) \to L$ forms an open hypercover of $L$.
So, through Lemma~\ref{cones-of-covers}, we can assume $L=\DD^j\times \oC(W)$ for some compact conically smooth stratified space $W$ and integer $j$.
Because both $\cF$ and ${\st}_\ast {\st}^\ast \cF$ are constructible, we can assume $j=0$, so that $L=\oC(W)$.
This case $Z=\oC^p(L)=\oC^{p+1}(W)$ follows by induction on $p$.

\end{proof}

\begin{theorem}\label{striation-gcat}
The adjunction~(\ref{exit-adjunction}) restricts to an equivalence of $\infty$-categories
\[
\Stri~\simeq~\Cat_\infty~.  
\]

\end{theorem}

\begin{proof}
Through the equivalence $\shv^{\cone, \cbl}(\strat)\simeq\psh(\bdelta)$ of Lemma~\ref{cone-sheaves}, the consecutive condition is identical to the Segal condition.
Through the resulting equivalence $\shv^{\sf cons,\cone, \cbl}(\strat)\simeq\Psh^{\sf Segal}(\bdelta)$, the univalent condition is identical to the completeness condition.  
\end{proof}

\begin{remark}
It was a choice to prove Theorem~\ref{striation-gcat} by way of the \emph{standard} cosimplicial stratified space ${\st} \colon \bdelta \ra \strat$
as opposed to the stratification of the topological simplices $\Delta^p \to [p]$ given by $t\mapsto {\sf Min}\bigl\{ i\mid t_i\neq 0\bigr\}$. This choice resulted in our use of the exit-path $\oo$-category rather than its opposite, the enter-path $\oo$-category.

\end{remark}

We recall the notion of {\it strong generation} from, for example, \cite{pkd}.

\begin{definition}
A functor $g:\cC \ra \cD$ {\it strongly generates} if the diagram
\[
\xymatrix{
\cC\ar[r]^g\ar[d]_g&\cD\\
\cD\ar@{-->}[ur]_{{\sf id}_\cD}
}
\] 
witnesses ${\sf id}_\cD$ as the left Kan extension of $g$ along $g$.
\end{definition}

\begin{remark}
The following are examples of strongly generating functors: the inclusion of the terminal category into spaces, $\{\ast\}\hookrightarrow \spaces$; the fully faithful inclusion of a rank-1 free $R$-module into all $R$-modules, $\{R\} \hookrightarrow \m_R$, for any associative ring $R$; the inclusion of free algebras into algebras; the Yoneda functor $\cC \ra \psh(\cC)$ for any $\oo$-category $\cC$; any localization $\cC \ra \cC[J^{-1}]$. See \S3.5 of \cite{pkd} for an extended discussion.
\end{remark}

\begin{lemma}\label{lem.stronggen}
The standard stratification functor ${\st}\colon\bdelta \ra \Strat$ strongly generates. In particular, for every conically smooth stratified space $X$, the canonical morphism
\[
\underset{[p]\in \bdelta_{/X}}\colim \Delta^p \longrightarrow X
\]
exists in $\Strat$ and is an equivalence, where $\bdelta_{/X}:=\bdelta\underset{\Strat}\times\Strat_{/X}$ is an $\oo$-category in which an object is a pair $([p]\in \bdelta, \Delta^p \ra X)$.
\end{lemma}
\begin{proof}
In this proof we notate the Yoneda functor by ${\sf y}: \Strat\hookrightarrow \psh(\Strat)$. 
We prove that the canonical morphism $\lkan_{\st}(\st)\ra {\sf id}_{\Strat}=:{\sf id}$ is an equivalence. 
To do so is equivalent to showing, for each pair $X$ and $Y$ of stratified spaces, that the canonical map between spaces
\begin{equation}\label{06}
\Strat(X,Y) \longrightarrow \Strat\bigl(\lkan_{\st}(\st)(X),Y\bigr)\simeq \limit \bigl((\bDelta_{/X})^{\op}\to \bDelta^{\op} \xra{\st} \Strat^{\op} \xra{\Strat(-,Y)} \Spaces\bigr)
\end{equation}
is an equivalence. 
This will imply that $X$ has the universal property of the left Kan extension $\lkan_{\st}(\st)(X)$; consequently, it simultaneously implies that the left Kan extension $\lkan_{\st}(\st)$ indeed exists. 
Through the following string of equivalences among spaces
\begin{eqnarray}
\nonumber
\underset{\bigl([p]\in\bdelta_{/X}\bigr)^{\op}}\limit \Map_{\psh(\Strat)}({\sf y}_{\Delta^p}, {\sf y}_Y)
&
\simeq
&
\Map_{\psh(\Strat)}\Bigl(\underset{[p]\in\bdelta_{/X}}\colim {\sf y}_{\Delta^p}, {\sf y}_Y\Bigr)
\\
\nonumber
&
\simeq
&
\Map_{\psh(\Strat)}(\st_!\st^\ast {\sf y}_X, {\sf y}_Y)
\\
\nonumber
&
\simeq
&
\Map_{\psh(\bdelta)}(\st^\ast {\sf y}_X, \st^\ast {\sf y}_Y)
\\
\nonumber
&
\simeq
&
\Map_{\psh(\Strat)}({\sf y}_X, \st_\ast\st^\ast {\sf y}_Y)
\end{eqnarray}
we recognize the map~(\ref{06}) as the canonical map
\[
\Strat(X,Y) \ra \Map_{\psh(\Strat)}({\sf y}_X, \st_\ast\st^\ast {\sf y}_Y)~.
\]
Consequently, the map~(\ref{06}) is an equivalence if and only if, for each stratified space $Y$, the canonical natural transformation ${\sf y}_Y \ra  \st_\ast\st^\ast {\sf y}_Y$ is an equivalence.
Note that as a consequence of Lemma~\ref{cone-sheaves}, any natural transformation $\cF \ra \cG$ of functors in $\psh(\Strat)$ is an equivalence if both of the following are true.
\begin{itemize}
\item The restriction $\st^\ast \cF \ra \st^\ast \cG$ is an equivalence in $\psh(\bdelta)$.
\item Both $\cF$ and $\cG$ are cone-local sheaves, i.e., satisfy (1) and (3) from Definition \ref{def.striation}.
\end{itemize}
We are therefore reduced to showing that these criteria hold for the natural transformation ${\sf y}_Y \ra  \st_\ast\st^\ast {\sf y}_Y$.

Fully faithfulness of $\st$, which is Lemma~\ref{standard-ff}, immediately implies the natural transformation $\st^\ast {\sf y}_Y \ra \st^\ast \st_\ast\st^\ast {\sf y}_Y$ is an equivalence.
The representable presheaf ${\sf y}_Y$ is a cone-local sheaf because both diagram types (1) open covers and (3) cone-quotients are colimit diagrams in $\Strat$ (by Lemma \ref{lem.opens-cover} and Lemma \ref{lem.blow-up-colims}). 
It remains to verify $\st_\ast\st^\ast({\sf y}_Y)$ is a cone-local sheaf. 
Note that, by definition, $\exit(Y):=\st^\ast({\sf y}_Y)$.
Corollary~\ref{exit-cplt-segal} states that $\exit(Y)$ is an $\infty$-category.  
Theorem~\ref{striation-gcat} states that, for each $\infty$-category $\cC$, the presheaf $\st_\ast \cC$ is a striation sheaf.  
In particular, $\st_\ast \exit(Y)\simeq \st_\ast \st^\ast({\sf y}_Y)$ satisfies (1) and (3) of Definition~\ref{def.striation}.

\end{proof}

\begin{theorem}\label{main}
The exit-path functor
\[
\exit: \Strat\longrightarrow \Cat_{\oo}
\]
is fully faithful.
\end{theorem}

\begin{proof}
Lemma~\ref{lem.stronggen} gives that the, a priori, lax commutative diagram among $\infty$-categories:
\[
\xymatrix{
\Strat\ar[r]^{\exit}\ar@{^{(}->}[d]&\Cat_{\oo}\ar@{^{(}->}[d]\\
\psh(\Strat)&\psh(\bdelta)\ar@{_{(}->}[l]_{\st_!}
}
\]
in fact commutes.  
Lemma~\ref{standard-ff} gives that $\st\colon \bDelta \to \Strat$ is fully faithful, from which it follows that left Kan extension $\st_\ast$ along it is also fully faithful.  
We conclude that the functor $\exit\colon \Strat \ra \Cat_{\oo}$ is fully faithful, as desired.

\end{proof}

\begin{remark}
Theorem \ref{main} is an instance of the general fact that strong generation of a functor $\cC \ra \cD$ implies fully faithfulness of the restricted Yoneda functor $\cD \ra \psh(\cC)$ (\S3.5 of \cite{pkd}), combined with the particular result that the essential image of the restricted Yoneda functor $\Strat\ra \psh(\bdelta)$ lies in $\Cat_{\oo}$.
\end{remark}

\section{Transversality sheaves}\label{sec.trans-sheaves}
Here we give a procedure for manufacturing examples of striation sheaves, hence $\oo$-categories, using even simpler stratified geometry mixed with ordinary category theory.

\subsection{Transversality sheaves}
We now present transversality sheaves, which give checkable conditions that guarantee that the topologizing diagram of an isotopy sheaf is a striation sheaf. 
Roughly, transversality sheaves are to striation sheaves as quasi-categories are to complete Segal spaces. 
We will use transversality sheaves to construct our $\oo$-categories of ultimate interest, $\Bun$ and $\Exit$.

\begin{definition}[Transversality sheaves]\label{def.transversality}
A transversality sheaf is a right fibration $\fF\ra \strat$ among ordinary categories for which its topologizing diagram $|\fF(-\times\Delta^\bullet_e)|$ is a striation sheaf. 
The category of transversality sheaves is the full subcategory 
\[
\trans~\subset~\Cat_{/\strat}
\] 
consisting of transversality sheaves.

\end{definition}

We now give practicable point-set sufficient conditions for checking that a right fibration is a transversality sheaf.

\begin{theorem}\label{transversality.conditions}
If a right fibration $\fF \ra \strat$ admits a straightening $\fF\colon \strat^{\op} \to {\sf Gpd}$ that satisfies the following conditions, it is a transversality sheaf, i.e., $|\fF(-\times\Delta^\bullet_e)|$ is a striation sheaf.

\begin{enumerate}
\item {\bf Sheaf:} For each covering sieve $\cU \subset \strat_{/K}$, the restriction of $\fF$ along the adjoint diagram $\cU^{\tr} \to \strat$ is a limit diagram of groupoids, which is to say that the canonical functor between groupoids
\[
\fF(K)\xra{~\simeq~} \underset{U\in \cU} \limit \fF(U)
\]
is an equivalence.

\item {\bf Isotopy extension:} 
For each stratified space $K$, the following two conditions are satisfied.
\begin{itemize}
\item
For each $p\geq 0$, the functor between groupoids
\[
\fF(K\times \Delta^q_e) \longrightarrow \fF(K\times \partial \Delta^q_e)
\]
is an isofibration.  

\item
For each $0\leq i \leq p$, the functor between groupoids
\[
\fF(K\times \Delta^q_e) \longrightarrow \fF(K\times (\Lambda^q_i)_e)
\]
is surjective on objects and on Hom-sets.

\end{itemize}

\item {\bf Cone-local:} For each compact conically smooth stratified space $L$, the value of $\fF$ on the diagram
\[
\xymatrix{
L  \ar[r]  \ar[d]
&
L\times \RR_{\geq 0}  \ar[d]
\\
\ast  \ar[r]
&
\sC(L)
}
\]
is a limit diagram of groupoids.

\item {\bf Consecutive:} 
For each $ 0<k<p$, and each $q\geq 0$, and for $\square$ the pullback groupoid as in the diagram
\[
\xymatrix{
\square  \ar[r]   \ar[d]
&
\fF(\Delta^{\{0<\dots<k\}}\times \Delta^q_e) \underset{\fF(\Delta^{\{k\}\times \Delta^q_e})}\times \fF(\Delta^{\{k<\dots<p\}}\times \Delta^q_e)    \ar[d]
\\
\fF(\Delta^p\times \partial \Delta^q_e)   \ar[r]
&
\fF(\Delta^{\{0<\dots<k\}}\times \partial \Delta^q_e)\underset{\fF(\Delta^{\{k\}}\times \partial \Delta^q_e)}\times\fF( \Delta^{\{k<\dots<p\}}\times \partial\Delta^q_e)~,
}
\]
the canonical functor $\fF(\Delta^p\times\Delta^q_e)  \to \square$ is surjective on objects and on $\Hom$-sets.

\item {\bf Univalent:}
The topologizing diagram $|\fF(-\times\Delta^\bullet_e)|$ carries the diagram in $\strat$
\[
\xymatrix{
\Delta^{\{0<2\}}\amalg \Delta^{\{1<3\}}  \ar[r] \ar[d]
&
\Delta^{\{0<1<2<3\}} \ar[d]
\\
\ast \thinspace\amalg \thinspace\ast  \ar[r]
&
\ast
}
\]
to a pullback diagram of spaces.  

\end{enumerate}
\end{theorem}
\begin{proof}
Notice that the first two conditions coincide with those of an isotopy sheaf.
So Theorem \ref{isot.cbl} grants that the topologizing diagram $|\fF(-\times\Delta^\bullet_e)|$ is a constructible sheaf.
Clearly, the {\bf univalent} condition on $\fF$ equals the univalent condition for $|\fF(-\times \Delta^\bullet_e)|$.  
It remains to show that {\bf cone-locality} and {\bf consecutivity} for $\fF$ imply the corresponding conditions for its topologizing diagram $|\fF(-\times\Delta^\bullet_e)|$.

\medskip

\noindent
{\bf Cone-local:}
From the cone-local condition, and since, for each $q\geq 0$, taking products with $\Delta^q_e$ commutes with pushouts, we have that the canonical diagram of groupoids
\[\xymatrix{
\fF(\sC(L)\times\Delta^p_e)\ar[r]\ar[d]&\fF(L\times\RR_{\geq 0}\times\Delta^p_e)\ar[d]\\
\fF(\Delta^p_e)\ar[r]&\fF(L\times\Delta^p_e)\\}\]
is a point-set pullback.
Taking nerves and restricting to diagonals preserves limits, therefore the canonical diagram of Kan complexes
\begin{equation}\label{some-diagram}
\xymatrix{
\delta^\ast\fF(\sC(L)\times\Delta^\bullet_e)_\star\ar[r]\ar[d]&\delta^\ast\fF(L\times\RR_{\geq 0}\times\Delta^\bullet_e)_\star\ar[d]\\
\delta^\ast\fF(\Delta^\bullet_e)_\star\ar[r]&\delta^\ast\fF(L\times\Delta^\bullet_e)_\star\\}
\end{equation}
is a point-set pullback. 
From the proof of Theorem \ref{isot.cbl}, the restriction map \[\delta^\ast\fF(L\times\RR_{\geq 0}\times\Delta^\bullet_e)_\star\longrightarrow \delta^\ast\fF(L\times\Delta^\bullet_e)_\star\] is a Kan fibration. 
The diagram~(\ref{some-diagram}) is therefore a homotopy pullback among Kan complexes in Quillen's model structure on simplicial sets.  
In particular, the diagram~(\ref{some-diagram}) gives a pullback diagram in the $\infty$-category $\Spaces$.  

\medskip

\noindent
{\bf Consecutive:}
We now show $|\fF(-\times\Delta^\bullet_e)|: \strat^{\op}\ra \spaces$ is consecutive. 
That is, we show that the composite
\[\xymatrix{\bdelta^{\op}\ar[rr]^{\st} &&\strat^{\op}\ar[rr]^{|\fF(-\times\Delta^\bullet_e)|}&&\spaces\\}\]
is a Segal space. 
Fix $0\leq k \leq p$, with associated maps $\Delta^k \ra \Delta^p$ and $\Delta^{p-k}\ra \Delta^p$ induced by the inclusions of $[k]=\{0<\ldots <k\}$ and $[p-k]=\{k<\ldots <p\}$ into $[p]=\{0<\ldots<p\}$.  We will show that the square among bisimplicial sets
\[
\xymatrix{
 \fF(\Delta^p\times \Delta^\bullet_e)_\star  \ar[rr]  \ar[d]
&&
 \fF(\Delta^{p-k}\times \Delta^\bullet_e)_\star  \ar[d]
\\
 \fF(\Delta^{k}\times \Delta^\bullet_e)_\star  \ar[rr] 
&&
 \fF(\Delta^{\{k\}}\times \Delta^\bullet_e)_\star 
}
\]
is a homotopy pullback in the diagonal model structure. 
We show that the canonical map to the pullback
\[
\fF(\Delta^p\times\Delta^\bullet_e)_\star\longrightarrow \square
\]
is an acyclic fibration for the diagonal model structure on bisimplicial sets.
By \cite{jardine}, acyclic fibrations for the diagonal model structure are detected by having the right lifting property with respect to the cofibrations
\begin{equation}\label{lift11}
\partial \Delta[q] \boxtimes \Delta[r]\underset{\partial \Delta[q] \boxtimes\partial \Delta[r]}\coprod \Delta[q]\boxtimes \partial \Delta[r]~\hookrightarrow~ \Delta[q]\boxtimes\Delta[r]~\qquad~
\end{equation}
where $\boxtimes:\Fun(\bdelta^{\op},\set)\times \Fun(\bdelta^{\op},\set)\ra \Fun(\bdelta^{\op}\times\bdelta^{\op},\set)$ is the external product of simplicial sets. 
In the case at hand, for each conically smooth stratified space $Z$ and each $q\geq 0$, the simplicial set $\fF(Z\times \Delta^q_e)_\star$ is $1$-coskeletal, since it is the nerve of a groupoid.
Thus, it suffices to verify the right lifting property in the cases $r=0,1$.

{\bf Case $r=0$:} We solve the lifting problem
\[
\xymatrix{
\partial\Delta[q]\boxtimes \Delta[0]\ar[rr]  \ar[d]
&&
 \fF(\Delta^p\times \Delta^\bullet_e)_\star  \ar[d]
\\
\Delta[q]\boxtimes\Delta[0]  \ar[rr] \ar@{-->}[urr]
&&
 \fF(\Delta^{k}\times \Delta^\bullet_e)_{\star} \underset{\fF(\Delta^{\{k\}}\times \Delta^\bullet_e)_{\star}}   \times \fF( \Delta^{p-k}\times \Delta^\bullet_e)_\star~.
}
\]
This problem is adjoint to the problem of surjectivity on objects of the functor
\[
\fF(\Delta^p\times \Delta^q_e) \longrightarrow \square~.
\]
Consequently, this problem is solved by the assumed {\bf consecutive} condition on $\fF$.  

{\bf Case $r=1$:} 
We solve the lifting problem
\[
\xymatrix{
\partial\Delta[q]\boxtimes \Delta[1]\underset{\partial\Delta[q]\boxtimes\partial\Delta[1]}\coprod\Delta[q]\boxtimes\partial\Delta[1]\ar[rr]  \ar[d]
&&
 \fF(\Delta^p\times \Delta^\bullet_e)_\star  \ar[d]
\\
\Delta[q]\boxtimes\Delta[1]  \ar[rr] \ar@{-->}[urr]
&&
 \fF(\Delta^{k}\times \Delta^\bullet_e)_{\star} \underset{\fF(\Delta^{\{k\}}\times \Delta^\bullet_e)_{\star}}   \times \fF( \Delta^{p-k}\times \Delta^\bullet_e)_\star~.
}
\]
This problem is adjoint to the problem of surjectivity on $\Hom$-sets of the functor
\[
\fF(\Delta^p\times \Delta^q_e) \longrightarrow \square~.
\]
Consequently, this problem is solved by the assumed {\bf consecutive} condition on $\fF$.

\end{proof}

\begin{remark}\label{class-vs-top}
Restriction along $\bdelta \xra{\Delta^\bullet} \strat$ gives the functor $\Psh(\strat)\to \Psh(\bdelta)$ to simplicial spaces.
By inspection, the map of cosimplicial stratified spaces $\Delta^\bullet \to \Delta^\bullet_e$, which is the standard map on underlying topological spaces, factors as
\[
\Delta^\bullet \longrightarrow E^\bullet\longrightarrow \Delta^\bullet_e
\]
where $E^p$ is the simplicial stratified space $[r]\mapsto \underset{[r]\to \sE^p} \coprod \Delta^r$ where $\sE^p$ is the ordinary category corepresenting a composible sequence of $p$ isomorphisms.  
There results, for each presheaf $\fF\in \Psh(\strat)$, a natural transformation
\[
|\fF(-\times\Delta^\bullet_e)|_{|\bdelta^{\op}} \longrightarrow |\fF(-\times E^\bullet)|_{|\bDelta^{\op}}
\]
from the restriction of the topologizing diagram to the classifying diagram.
After Rezk's reformulation of the completeness condition for Segal spaces (see~\S10 of~\cite{rezk-n}), our {\bf univalence} condition gives that this natural transformation is an equivalence for $\fF$ a \emph{transversality sheaf}.

\end{remark}

\section{Constructible bundles}
We will think of the category $\strat$ as a category of parametrizing objects for stratified manifold structures, just as affine schemes parametrize algebro-geometric structures.
In this section, we investigate a universal such example, $\Bun$, for which a $K$-point is a constructible bundle $X\to K$.   
This universal striation sheaf specializes to many interesting situations.  For instance, from it one can recover: the $\infty$-category of smooth $n$-manifolds and smooth embeddings among them, for any $n$; the simplicial category $\bdelta$; and the category of based finite sets.

\subsection{Closure properties of constructible bundles}
We show that constructible bundles are closed under composition and pullback, among other formations. We first recall the definition.
\begin{definition}\label{def.cbl}
Let $X\xra{\pi} K$ be a conically smooth map between stratified spaces.
\begin{itemize}
\item This map $\pi$ is a \emph{fiber bundle} if there is a basis $\{U_\alpha \subset Y\}$ for the topology of $Y$ for which, for each $\alpha$, there is a pullback diagram among stratified spaces
\[
\xymatrix{
F\times U_\alpha \ar[r]  \ar[d]^-{\sf pr}
&
X  \ar[d]
\\
U_\alpha \ar[r]
&
Y.
}
\]

\item 
This map $\pi$ is a \emph{weakly constructible bundle} if, for each stratum $K_q\subset K$, the restriction $X_{|K_q}\to K_q$ is a fiber bundle.

\item
This map $\pi$ is a \emph{constructible bundle} if it is weakly constructible, and if it satisfies the following condition which is inductive on the depth of $K$.
\begin{itemize}

\item In the case that $K$ has depth zero, there is no further condition on $\pi$.  

\item Suppose $K$ has finite depth $d>0$.  
For $K_d\subset K$ the deepest strata, the condition on $\pi$ is that the canonical map
\[
\Link_{X_{|K_q}}(X)\longrightarrow X_{|K_q}\underset{K_q}\times \Link_{K_q}(K)
\]
is a constructible bundle.

\item For general $K$, the condition on $\pi$ is that, for each open subset $U\subset K$ for which the inherited stratification on $U$ has finite depth, the restriction $X_{|U_\alpha}\to U_\alpha$ is a constructible bundle.  

\end{itemize}

\end{itemize}

\end{definition}

\begin{remark}
The added condition of constructibility on links is essential for the existence of pullbacks shown in Lemma \ref{cbls-pullback}. See Remark \ref{rem.wcbl.dont}.
\end{remark}

We observe several common classes of constructible bundles.
\begin{observation}\label{bdl-are-cbl}\label{cbl-are-bdl}\label{strata-are-cbl}\label{sub-cbl}

\begin{itemize}
\item[~]
\item
A fiber bundle $X\ra K$ is a constructible bundle: it is weakly constructible, and for each stratum $K_q\subset K$ the natural map
\[
\Link_{X_{|K_q}}(X)\longrightarrow X_{|K_q}\underset{K_q}\times \Link_{K_q}(K)
\] 
is an isomorphism.

\item
A weakly constructible bundle $X\ra K$ with a smooth base is necessarily a fiber bundle.

\item 
For each conically smooth stratified space $X = (X\to P)$, and each consecutive subposet $Q\subset P$, the inclusion of the preimage of $Q$
\[
X_Q~\hookrightarrow~X
\]
is a constructible bundle. 

\item
For a constructible bundle $X\to K$ and a stratum $X_p\subset X$, the composite $X_p\to K$ is again a constructible bundle.  

\end{itemize}

\end{observation}

The condition of being a constructible bundle is local in the base.
\begin{observation}\label{cbls-cover}
Let $X\ra K$ be a conically smooth map, and let $\cU_0\subset K$ be an open cover of $K$.
Then $X\ra K$ is a constructible bundle if and only if, for each $U\in \cU_0$, the restriction $X_{|U}\ra U$ is a constructible bundle.

\end{observation}

The following is routine.

\begin{lemma}[Lemma 7.1.3 of \cite{aft1}]\label{bdls-homot}
For each fiber bundle $X\xra{f} K\times \RR^i\times \sC(L)$, there is an isomorphism
\[
X~\cong~ X_{|K\times\{0\}}\times\RR^i\times\sC(L)
\]
over $K\times \RR^i\times \sC(L)$ and under $K\times\{0\}$. 

\end{lemma}

\begin{cor}\label{bdls-compose}
The composition of conically smooth maps $X\ra Y\ra Z$ is a fiber bundle if both $X\ra Y$ and $Y\ra Z$ are fiber bundles.

\end{cor}
\begin{proof}
It is sufficient to consider the case $Z=\RR^i\times\sC(L)$ is a basic.  
Lemma~\ref{bdls-homot} grants that the map $Y\to Z$ is isomorphic to a projection $Y_0\times \RR^i\times \sC(L) \to \RR^i\times \sC(L)$.
Another application of Lemma~\ref{bdls-homot} gives that $X\to Y$ is isomorphic to a product map $X_0\times\RR^i\times \sC(L)\to Y_0\times \RR^i\times \sC(L)$.  
It follows that the composition $X\to Z$ is isomorphic to a projection $X_0\times \RR^i\times \sC(L) \to \RR^i\times \sC(L)$.  

\end{proof}

The designed regularity within stratified spaces gives the following easy criterion that characterizes weak constructibility.
\begin{lemma}\label{to-R}
The following two conditions on a conically smooth map $f\colon X\to K$ are equivalent.
\begin{enumerate}
\item
The map $X\ra K$ is a weakly constructible bundle.

\item
The composite $X_p\to X\to K$ is a weakly constructible bundle for each stratum $X_p\subset X$.  

\end{enumerate}

\end{lemma}
\begin{proof}
Suppose~(1) is true.
Each composition $X_p\to X \to K$ factors through a stratum $K_q\subset K$.
Because $\emptyset \to L$ is weakly constructible for any $L$, it is enough to show that the factorized map $X_p\to K_q$ is a fiber bundle.
By assumption, there is an open cover $\cU_0$ of $K_q$ together with isomorphisms $X_{|U}\cong F\times U$ over $U$ for each $U\in \cU_0$. 
In particular, there are isomorphisms among fiberwise strata $(X_p)_{|U} \cong F_p\times U$ over $U$ for each $U\in \cU_0$.  
This implies $X_p\to K_q$ is a fiber bundle.

Now suppose~(2) is true.
We must show, for each stratum $K_q\subset K$, that the restriction $X_{|K_q} \to K_q$ is a fiber bundle.  
So we can assume $K$ is trivially stratified, in which case we are to show $X\ra K$ is a fiber bundle.

Each stratified space $Z$ admits an open cover $\{Z^{\leq n}\mid n\in \ZZ\}$ with each $Z^{\leq n}\subset Z$ the locus of those points $z\in Z$ whose local dimension ${\sf dim}_z(Z)$ is at most $n$.  
Evidently, each assignment $Z\mapsto Z^{\leq n}$ is functorial with respect to conically smooth open embeddings.
Therefore, we can assume $X$ has depth which is bounded above.
We will induct on the depth of $X$.  

Should this depth be zero, which is to say that each connected component of $X$ is trivially stratified, the assertion is true because $X\to K$ is assumed to be a fiber bundle.
Denote the union of deepest strata $X_d\subset X$, and its complement $X_{>d}$.  
By induction, each of the restrictions $X_d\to K$ and $X_{>d}\to K$ is a fiber bundle.  
Consider the link $\Link_{X_d}(X) \xra{\pi} X_d$ and the fiberwise open cone $\sC(\pi)$. This is equipped with the cone-locus section $X_d\to \sC(\pi)$.  
Choose a conically smooth open embedding $\sC(\pi) \hookrightarrow X$ under $X_d$, which exists by Proposition 8.2.5 of~\cite{aft1}.  
This witnesses a pushout 
\[
\sC(\pi) \underset{\Link_{X_d}(X)\times \RR_{>0}}\coprod X_{>d}~\cong~ X~.
\]
Both of the projections $\Link_{X_d}(X) \xra{\pi} X_d$ and $\sC(\pi) \to X_d$ are fiber bundles;
after Corollary~\ref{bdls-compose} we see that both of these projections further project to $K$ as fiber bundles.
Thus, to verify that $X\to K$ is a fiber bundle, it suffices to show that the conically smooth open embedding $\sC(\pi)\hookrightarrow X$ can be re-chosen as one over $K$.  
For this, we can work locally in $K$ and assume $K=\RR^i$ for some $i$.

Fix isomorphisms $X_d \cong F_d \times \RR^i$ and $X_{>d}\cong F_{>d}\times \RR^i$, each over $\RR^i$ -- such a choice is possible because $\RR^i$ is contractible.  
Likewise, choose an isomorphism $\Link_{X_d}(X) \cong  L_d\times \RR^i$ over the first of these isomorphisms.
Denote the projection $\pi_d\colon L_d \to F_d$.
There results an isomorphism $\sC(\pi)\cong \sC(\pi_d)\times \RR^i$.  
We have thus established a conically smooth open embedding
\[
e\colon \sC(\pi_d)\times\RR^i\cong \sC(\pi)\hookrightarrow X
\]
under $F_d\times\RR^i\cong X_d$.  
Optimistically, we seek a conically smooth open embedding $\phi\colon \sC(\pi_d)\times \RR^i\hookrightarrow \sC(\pi_d)\times\RR^i$ under $F_d\times\RR^i$ so that the diagram
\[
\xymatrix{
\sC(\pi_d)\times \RR^i  \ar[r]^-{\phi}  \ar[dr]_-{{\sf pr}}
&
\sC(\pi_d)\times \RR^i  \ar[r]^-{e}  
&
X  \ar[dl]^-f
\\
&
\RR^i
&
}
\]
commutes, though we will only approximate as much.

Consider the composition $fe\colon \sC(\pi_d)\times \RR^i\to \RR^i$. By construction, for each point $x_d\in F_d$ of the cone-locus, the restriction $fe_{|\{x_d\}}\colon \RR^i\to \RR^i$ is the identity map.
The inverse function theorem of \cite{aft1} grants that, for each $R>0$, there is an open neighborhood of this cone-locus $F_d \subset O\subset \sC(\pi_d)$ for which, for each $x\in O$, the restriction $fe_{|\{x\}}\colon \RR^i\to \RR^i$ is a smooth open embedding whose image contains the ball $\sB_R(0)$ of radius $R$ about the origin.  
Because conically smooth open self-embeddings $\sC(L_d)\hookrightarrow \sC(L_d)$ form a basis for the topology about the origin $\ast \in \sC(L_d)$, there is a conically smooth open self-embedding $\un{\phi}\colon \sC(\pi_d) \hookrightarrow \sC(\pi_d)$ over $F_d$ whose image lies in $O$.  
Consider $\phi_R\colon \sC(\pi_d) \times \sB_R(0)\hookrightarrow \sC(\pi_d)\times \RR^i$ whose projection to $\sC(\pi_d)$ is $\un{\phi}$ and whose projection to $\RR^i$ is $(fe_{|\{-\}})^{-1}$. This map is a conically smooth open embedding.
By construction, the diagram
\[
\xymatrix{
\sC(\pi_d)\times \sB_R(0)  \ar[r]^-{\phi_R}  \ar[d]_-{{\sf pr}}
&
\sC(\pi_d)\times \RR^i  \ar[r]^-{e}  
&
X  \ar[d]^-f
\\
\sB_R(0)  \ar[rr]
&&
\RR^i
}
\]
commutes, where the bottom arrow is the standard inclusion.  
This is our approximate solution indicated earlier.  
This implies that the restriction $X_{|\sB_R(0)}\to \sB_R(0)$ is a fiber bundle.  
Because the collection $\{\sB_R(0)\mid R>0\}$ is an open cover of $\RR^i$, we conclude that $X\to \RR^i$ is a fiber bundle. This proves~(1).

\end{proof}

\begin{prop}\label{cbl-compose}
The composition of conically smooth maps $ X\ra Y\ra Z$ is a weakly constructible bundle if both $X\ra Y$ and $Y\ra Z$ are weakly constructible bundles. If, furthermore, $X\ra Y$ and $Y\ra Z$ are constructible bundles, then the composite $X\ra Z$ is a constructible bundle.

\end{prop}
\begin{proof}
We first show that the composite $X\ra Z$ is a weakly constructible bundle. Using Lemma~\ref{to-R}, we need only verify that, for each stratum $X_p\subset X$, the restriction of the composite $X_p\ra Z$ is weakly constructible.
From the definition of conically smooth maps, this composition factors as a composition through strata: $X_p\to Y_q \to Z_r$.  
By assumption, each of these factored maps is a fiber bundle among smooth manifolds.
The result follows because the composition of smooth fiber bundles is again a smooth fiber bundle. 

Now, assuming $X\ra Y$ and $Y\ra Z$ are constructible bundles, we show further that $X\ra Z$ is constructible. 
By base change along a basic neighborhood in $Z$, we can assume $Z$ has finite depth.
We prove this by induction on the depth of $Z$. 
The depth zero case is immediate, since the composite of fiber bundles is again a fiber bundle. For the inductive step, let $Z_r\subset Z$ be any stratum. We have a commutative diagram
\[
\xymatrix{
\Link_{X_{|Z_r}}(X)\ar[d] \ar[r]&X_{|Z_r}\underset{Z_r}\times\Link_{Z_r}(Z)\ar@{=}[d]\\
X_{|Z_r}\underset{Y_{|Z_r}}\times\Link_{Y_{|Z_r}}(Y)\ar[r]&X_{|Z_r}\underset{Y_{|Z_r}}\times Y_{|Z_r}\underset{Z_r}\times\Link_{Z_r}(Z)}
\]
factoring the map in question as a composite of constructible bundles. 
Since the depth has decreased, the inductive hypothesis implies that the composite map is a constructible bundle.

\end{proof}

In order to state Lemma \ref{lemma.cbltokanfib}, we recall in some detail the notion of parallel vector fields from \cite{aft1}.
For $Z$ a stratified space, recall from~\S7 is its \emph{total unzip}, $\Unzip(Z)$, which is a smooth manifold with corners equipped with a surjective weakly constructible bundle
\[
q_Z\colon \Unzip(Z) \longrightarrow Z~.
\]
We next recall from~\S8 of~\cite{aft1} the sheaf $\Theta_Z$ of \emph{parallel vector fields on $Z$}:
first suppose that $Z$ is a manifold with corners, with $\sT Z \to Z$ the ordinary tangent bundle of $Z$.
For $U\subset Z$ an open subset, the vector space $\Theta_Z(U)$ is the vector subspace of those smooth sections
\[
\xymatrix{
&&
\sT  Z  \ar[d]
\\
U  \ar@{-->}[urr]^-{V}  \ar[rr]
&&
Z
}
\]
such that for each point $z\in U\cap \partial_S Z$ belonging to a face of $Z$, the vector $V_z \in \sT_z \partial_S Z \subset \sT_z Z$ is tangent to that face.
Restricting smooth sections along open inclusions defines the functor $\Theta_Z$, which is evidently a sheaf.

Now suppose $Z$ is an arbitrary stratified space.
For each stratum $\Unzip(Z)_{\w{p}} \subset \Unzip(Z)$, the restriction
\[
(q_Z)_{|\w{p}}    \colon \Unzip(Z)_{\w{p}} \longrightarrow  Z_p\subset Z
\]
necessarily factors through a single stratum of $Z$, as indicated.
Consequently, this restriction $(q_Z)_{|\w{p}}$ is a smooth map between smooth manifolds (possibly with corners).
Taking its derivative determines the bundle map
\[
\sD (q_Z)_{|\w{p}} \colon \sT \Unzip(Z)_{\w{p}}\longrightarrow {(q_Z)_{|\w{p}}}^\ast \sT Z_p
\]
over $(q_Z)_{|\w{p}}$.
Consider the subsheaf ${\sf Ker}(\sD q_Z) \subset \Theta_{\Unzip(Z)}$ consisting, for each open subset $U\subset \Unzip(Z)$, of those parallel vector fields $V$ on $U$ for which, for each point $\w{z}\in \Unzip(Z)$ belonging to the $\w{p}$-stratum, the tangent vector $V_{\w{z}}$ belongs to the kernel of the linear map
\[
\sD_{\w{z}} (q_Z)_{|\w{p}} \colon \sT_{\w{z}} \Unzip(Z)_{\w{p}}\longrightarrow \sT_{q_Z(\w{z})} Z_p~.
\]
Finally, the sheaf $\Theta_Z$ of \emph{parallel vector fields on $Z$} is the cokernel among sheaves on $Z$ of vector spaces
\[
(q_Z)_\ast {\sf Ker}(\sD q_Z)  \longrightarrow
(q_Z)_\ast \Theta_{\Unzip(Z)} \longrightarrow
\Theta_Z
\]
involving pushforwards along the continuous map $q_Z$.

Before stating our the next result, recall also the simplicial enrichment
\[
\Strat(V,X)_p := \strat_{/\Delta_e^p}(V\times\Delta_e^p,V\times\Delta_e^p)
\]
with respect to which the simplicial set $\Strat(V,X)$ is a Kan complex for any conically smooth stratified spaces $V$ and $X$ (Lemma 4.1.4 of \cite{aft1}).

\begin{lemma}\label{lemma.cbltokanfib}
Let $f:X\ra K$ be a weakly constructible bundle.
\begin{enumerate}
\item\label{one.kan}
For any conically smooth stratified space $U$, the induced postcomposition
\[
\Strat(U,X)\longrightarrow \Strat(U,K)
\]
is a Kan fibration.
\item\label{two.vect}
For each parallel vector field $V$ on $K$, there is a parallel vector field $\w{V}$ on $X$ such that for each $x\in X$, the derivative
\[
\sD_x f \colon \sT_x X \longrightarrow \sT_{f x}K
\]
carries $\w{V}_x$ to $V_{f x}$.

\end{enumerate}
\end{lemma}
\begin{proof}
Statement (\ref{two.vect}) is an infinitesimal version of statement (\ref{one.kan}).
To show (\ref{one.kan}), we show the existence of a dashed lift for every commutative diagram
\[
\xymatrix{
\Lambda_k[p]\ar[r]\ar[d]&\Strat(U,X)\ar[d]
&{\rm or}&U\times (\Lambda_k^p)_e\ar[r]\ar[d]&X\ar[d]\\
\Delta[p]\ar[r]\ar@{-->}[ur]&\Strat(U,K)
&&U\times \Delta^p_e\ar[r]\ar@{-->}[ur]&K}
\]
where the first diagram is of simplicial sets, and the second diagram is of conically smooth stratified spaces.
In the first diagram, $\Lambda_k[p]$ is the $k$-horn of the $p$-simplex $\Delta[p]$, for each $0\leq k\leq p$.
In the second diagram, $(\Lambda_k^p)_e$ is the extended $k$-horn in the extended topological $p$-simplex $\Delta^p_e$.
(A conically smooth map $U\times (\Lambda_k^p)_e\ra X$ is understood to mean a stratified map such that each of the $p$ restrictions $U\times\Delta^{p-1}_e\ra X$ is conically smooth.)
The equivalence of the two lifting problems is immediate from the definition of the simplicial enrichment $\Strat$.
Note that the construction of this lift is local in $U\times\Delta^p_e$.
Consequently, by reduction to a cover, it suffices to show the existence of a diagonal lift
\[
\xymatrix{
Z\times \{0\}\ar[r]^{\ov{g}_0}\ar[d]&X\ar[d]\\
Z\times \RR\ar[r]^g\ar@{-->}[ur]&K}
\]
for every compact conically smooth stratified space $Z$.
We delineate the steps of our proof.
\begin{enumerate}
\item\label{step.one} We obtain from the map $g:Z\times\RR\ra K$ a parallel vector field $V$ on $K\times Z\times\RR$.
\item\label{step.two} We lift $V$ to a parallel vector field $\widetilde{V}$ on $X\times Z\times\RR$.
\item\label{step.three} We construct a lift $\ov{g}: Z\times\RR \ra X$ by flowing $\widetilde{V}$. 
\end{enumerate}
We establish step (\ref{step.one}). The data of a conically smooth map $Z\times \RR \ra K$ is equivalent to a commutative diagram in $\strat$
\[
\xymatrix{
Z\times\RR \ar@{^{(}->}[rr]^{g\times {\sf id}_{Z\times\RR}}\ar@{=}[dr]&& K\times Z\times \RR\ar[dl]\\
&Z\times\RR}
\]
where $g\times{\sf id}:Z\times\RR \ra K\times Z\times\RR$ is a proper embedding and $K\times Z\times \RR\ra Z\times\RR$ is the projection map. By pushing forward the parallel vector field $\partial_t$ on $Z\times \RR$ we obtain a parallel vector field $\sD (g\times {\sf id}_{Z\times\RR})(\partial_t)$ defined on the graph of $g$.
By construction $\sD (g\times{\sf id}_{Z\times\RR})(\partial_t)$ lies over the standard vector field $\partial_t$ with respect to the projection.
By existence of regular neighborhoods and conically smooth partitions of unity (\S8.2 and \S7.1 of \cite{aft1}), we extend $\sD (g\times {\sf id}_{Z\times\RR})(\partial_t)$ to a vector field defined on all of $K\times Z\times \RR$ which projects to $\partial_t$ and which is properly supported with respect to projection. Denote this vector field $V$. The flow of $V$ is defined for all time and gives a map $K\times Z\times\RR\times\RR\ra K\times Z \times\RR$ such that the composite
\[
(Z\times\{0\})\times\RR\xra{(g_0\times{\sf id}_Z)\times{\sf id}_{\RR}}(K\times Z\times\RR)\times\RR\longrightarrow K\times Z\times\RR
\]
is equal to the original map $g\times {\sf id}_{Z\times\RR}$. That is, we have expressed $g$ from the $t=0$ condition $g_0$ and the flow of a vector field $V$.

We now establish step (\ref{step.two}), which is a form of assertion (\ref{two.vect}).
By the functoriality of the $\Unzip$ construction, any conically smooth map $f\colon Y\to S$ determines another conically smooth map
\[
\w{f}\colon \Unzip(Y) \longrightarrow \Unzip(S)
\]
which lies over $f$.
By direct inspection, if $f$ is a weakly constructible bundle, then $\w{f}$ is also weakly constructible bundle.
The map $f\colon Y \to S$ also determines a morphism between sheaves of vector spaces on $Y$,
\[
\sD \w{f}\colon \Theta_{\Unzip(Y)} \longrightarrow \w{f}^\ast \Theta_{\Unzip(S)}~,
\]
which descends to a morphism between sheaves of vector spaces on $Y$:
\[
\sD f \colon \Theta_Y \longrightarrow f^\ast \Theta_S~.
\]
Should $f\colon Y \to S$ be a weakly constructible bundle, then the map $\sD \w{f}$ is surjective.
Using that both $q_Y$ and $q_S$ are surjective weakly constructible bundles, the asymmetric 2-of-3 property for surjections gives that $\sD f$ is surjective.
By setting $Y=X\times Z\times \RR$ and $S=K\times Z\times\RR$, this surjectivity establishes step (\ref{two.vect}), that there exists a lift $\widetilde{V}$ of $V$.

In step (\ref{step.three}), we now flow the parallel vector field $\widetilde{V}$ on $X\times Z\times\RR$. Since $V$ lies over $\partial_t$, the flow is defined for all time. Thus, we obtain a map $X\times Z\times\RR\times\RR \ra X\times Z\times\RR$. Define $\ov{g}$ to be the composite
\[
Z\times\{0\}\times\RR \xra{(\ov{g}_0\times{\sf id}_Z)\times{\sf id}_{\RR}}X\times Z\times\RR\times\RR\longrightarrow X\times Z\times\RR
\xra{\sf proj}
X
\]
Since $\widetilde{V}$ lies over $V$, their flows commute, and so by construction $\ov{g}$ lifts the map $g:Z\times\RR \ra K$ and at $t=0$ is equal to $\ov{g}_0$.

\end{proof}

\begin{lemma}\label{lemma.PBstratStrat}
Consider a commutative diagram 
\[
\xymatrix{
X'\ar[r]\ar[d]&X\ar[d]^{\wcbl}\\
K'\ar[r]&K}
\]
which has the properties that
\begin{itemize}
\item the diagram is a limit diagram in the ordinary category $\strat$; and
\item the map $X\ra K$ is a weakly constructible bundle.
\end{itemize}
Then the diagram is again a limit in the $\oo$-category $\Strat$.
\end{lemma}
\begin{proof}
It suffices to show that for every $Z\in \Strat$, the diagram
\[
\xymatrix{
\Strat(Z,X')\ar[r]\ar[d]&\Strat(Z,X)\ar[d]\\
\Strat(Z,K')\ar[r]&\Strat(Z,K)}
\]
is a homotopy limit diagram of simplicial sets, with respect to the Kan model structure. Note that since $X'$ is the limit in the ordinary category $\strat$, the diagram is a point-set limit in the ordinary category of simplicial sets; this follows directly from the definition of the enrichment of $\Strat$.
By Lemma \ref{lemma.cbltokanfib}, the right vertical map is a Kan fibration.
Thus, the diagram is a point-set limit of Kan complexes in which the vertical arrows are Kan fibrations;
consequently, it is a homotopy limit with respect to the Kan model structure.

\end{proof}

\begin{lemma}\label{cbls-pullback}
Constructible bundles pull back along conically smooth maps. 
That is, for each constructible bundle $X\xra{f} K$ and each conically smooth map $K'\xra{g} K$, the pullback 
\[
\xymatrix{
X\underset{K}\times K'  \ar[r]  \ar[d]_-{f'}
&
X  \ar[d]^f
\\
K'  \ar[r]^-g
&
K
}
\]
exists in $\strat$, and the map $X\times_KK'\xra{f'} K'$ is again a constructible bundle.
Furthermore, $f'$ is proper whenever $f$ is proper.

\end{lemma}

\begin{proof}
The properness assertion is immediate. We first prove the existence of such pullbacks.
Open covers give colimit diagrams in $\strat$, and open covers pull back along conically smooth maps.  
Given these considerations, because $K$ admits an open cover by stratified spaces of bounded depth, we can assume $K$ has bounded depth.
We proceed by induction on the depth of $K$.
The statement is standard provided $X\ra K$ is a fiber bundle, which proves the case that $K$ has depth zero.  

We now prove the inductive step.
Consider a deepest stratum $K_0\subset K$.
Consider the preimages $K'_0:=g^{-1}K_0$ and $X_0:=f^{-1}K_0$.  
Provided the existence of the fiber product $X\times_K K'$ in $\strat$, there is a canonical conically smooth isomorphism
\begin{equation}\label{pullback-blowups}
\bigl(X_0\underset{K_0}\times K'_0\bigr)
\underset{\bigl(\Link_{X_0}(X)\underset{\Link_{K_0}(K)}\times \Link_{K'_0}(K')\bigr)} 
\coprod 
\bigl(\Unzip_{X_0}(X)\underset{\Unzip_{K_0}(K)}\times \Unzip_{K'_0}(K')\bigr)
~\longrightarrow~ 
X\underset{K} \times K'~.
\end{equation}
We will show that $X\times_KK'$ exists in $\strat$ by showing that this lefthand pushout exists.
Being a constructible cover, this lefthand pushout exists provided the existence of the three fiber products in the lefthand expression.
By inspection, should each of the fiber products in the lefthand expression exist, this pushout exhibits the desired pullback, and the above arrow will be an isomorphism.
The base case of our induction assures the existence of the left term of the pushout.
We next show existence for the intermediate term of the pushout, that the fiber product of links exists: we have a commutative diagram
\[
\xymatrix{
\Link_{X_0}(X)\ar[rr]\ar[dr]&&\Link_{K_0}(K)\\
&X_0\underset{K_0}\times\Link_{K_0}(K)\ar[ur]}
\]
where the first map is a constructible bundle by assumption and the second map is the pullback of the constructible bundle $X_0\ra K_0$ along the map $\Link_{K_0}(K)\ra K_0$. Consequently, using the inductive hypothesis since this is a case of lesser depth, we obtain that the second map is a constructible bundle. By Proposition \ref{cbl-compose}, we obtain that the composite map $\Link_{X_0}(X)\ra \Link_{K_0}(K)$ is a constructible bundle. Again applying the inductive hypothesis now proves the existence of the intermediate term, as it is a pullback of a constructible bundle of stratified spaces of lesser depth.

To complete the proof, we lastly show that the fiber product of unzippings exists. Due to the existence of regular neighborhoods, there is a conically smooth map $\unzip_{K_0}(K)\xra{\alpha} \RR_{\geq 0}$ with the following properties:
\begin{itemize}
\item The preimage 
\[
\Link_{K_0}(K)=\alpha^{-1}0
\]
is the link.
Thereafter, we identify $\Link_{K'_0}(K')=(\alpha g)^{-1}0$ and $\Link_{X_0}(X)=(\alpha f)^{-1}0$.
\item There is an open conically smooth map 
\[
\Link_{K_0}(K)\times [0,1) \hookrightarrow \Unzip_{K_0}(K)
\]
under $\Link_{K_0}(K)$ and over the standard inclusion $[0,1)\to \RR_{\geq 0}$. 
Likewise, there are open conically smooth maps $\Link_{K'_0}(K')\times [0,1) \hookrightarrow \Unzip_{K'_0}(K')$ and $\Link_{X_0}(X)\times [0,1) \hookrightarrow \Unzip_{X_0}(X)$, respectively under $\Link_{K'_0}(K')$ and $\Link_{X_0}(X)$, each over the standard inclusion $[0,1)\to \RR_{\geq 0}$.
\end{itemize}
So, after Observation~\ref{cbls-cover}, to prove the existence of the fiber product of unzippings we are reduced to proving that a fiber product $X\underset{K}\times K'$ exists in the special case in which: 
\begin{itemize}
\item $K = L\times [0,1) = \Link_{K_0}(K)\times [0,1)$;

\item the map $X\xra{f} K=L\times[0,1)$ fits into a diagram among stratified spaces
\[
\xymatrix{
E\times [0,1)  \ar[r]^-r  \ar[d]_-{f_0\times {\sf id}} 
&
X  \ar[d]^-f
\\
L\times [0,1)  \ar[r]
&
K
}
\]
with $r$ a refinement and $f_0$ constructible;

\item the map $K'\xra{g} K=L\times[0,1)$ fits into a diagram among stratified spaces
\[
\xymatrix{
L'\times [0,1)  \ar[d]_-s  \ar[r]^-{g_0\times {\sf id} } 
&
L\times [0,1)  \ar@{=}[d]
\\
K'  \ar[r]^-g
&
K
}
\]
with $s$ a refinement.

\end{itemize}
By commuting limits, we have the identification
\[
X\underset{L\times [0,1)} \times K'~\cong~ (X\underset{L}\times K')\underset{L\times [0,1)} \times (X\underset{[0,1)}\times K')
\]
provided existence of each pullback on the right.
Applying this to our case at hand,
by induction on depth, we are reduced to the case that $L=\ast$.  
The identity
\[
\bigl(E\times[0,1)\bigr)\underset{[0,1)}\times \bigl(L'\times [0,1)\bigr) \cong \bigl(E\times L')\times [0,1)
\]
verifies the existence of the pullback term in this expression.
Finally, we recognize the desired pullback
\[
\bigl(E\times[0,1)\bigr)\underset{[0,1)}\times \bigl(L'\times [0,1)\bigr) \longrightarrow X\underset{[0,1)}\times K'
\]
as the domain of a refinement, thereby demonstrating its existence.

We have now established the existence of pullbacks of constructible bundles. To conclude the proof, we show that the left vertical map $X\times_KK' \ra K'$ is a constructible bundle. 
We first show that this map is a weakly constructible bundle.
By definition, this immediately reduces to the case that $K'$ consists of a single stratum.  
In this case, the map $K'\to K$ factors through a single stratum, $K_q\subset K$. 
The problem thereby reduces to the case that $K$ too consists of a single stratum and so that $X\to K$ is a fiber bundle.  
This problem is local in $K'$, so we can assume that both $K'$ and $K$ are Euclidean spaces.
In this case, there is an isomorphism $X\cong F\times K$ over $K$, which thereafter determines an isomorphism $X\times_KK' \cong F\times K'$ over $K'$.
We conclude that $X\times_KK'  \to K'$ is a weakly constructible bundle, as desired.

It remains to show that $X\times_KK' \to K'$ is constructible. 
For the duration of the proof, fix the notation $X' := X\times_KK'$ and $X_q:=X_{|K_q}$ and $X'_q := X_q \times_{K_q}K_q'$. 
We must show that the conically smooth map
\[
\Link_{X_q'}(X')\longrightarrow \Link_{K_q'}(K')
\]
is a constructible bundle. 
To prove this, we show that the natural diagram in $\strat$
\begin{equation}\label{97}
\xymatrix{
\Link_{X_q'}(X') \ar[r]\ar[d]&\Link_{X_q}(X)\ar[d]\\
\Link_{K_q'}(K')\ar[r]&\Link_{K_q}(K)}
\end{equation}
is a pullback.
By definition of $X\to K$ being a constructible bundle, the right vertical map in this diagram is a constructible bundle.
Ths proof is then complete by induction on the depth of the bottom terms in this diagram, which are strictly less than the respective depths of $K'$ and $K$.

Using that the right vertical map in~(\ref{97}) is a constructible bundle, the argument above ensures the existence of the pullback $\Link_{K'_q}(K')\underset{\Link_{K_q}(K)}\times \Link_{X_q}(X)$ in the category $\strat$.
So we must show that the conically smooth map 
\begin{equation}\label{94}
\Link_{X_q'}(X') \longrightarrow \Link_{K'_q}(K')\underset{\Link_{K_q}(K)}\times \Link_{X_q}(X)
\end{equation}
is an isomorphism in the category $\strat$.
We do this by first arguing that this map is a refinement, then by arguing that this map induces a homotopy equivalence between strata.

The assumption that $X\to K$ is a constructible bundle implies $\Unzip_{X_q}(X) \to \Unzip_{K_q}(K)$ is a constructible bundle.
As argued above, this ensures that the pullback $\Unzip_{K'_q}(K')\underset{\Unzip_{K_q}(K)}\times \Unzip_{X_q}(X)$ in $\strat$ exists.  
This pullback is a stratified space with boundary $\Link_{K'_q}(K')\underset{\Link_{K_q}(K)}\times \Link_{X_q}(X)$ whose interior is $X'\smallsetminus X'_q$.
Consider the diagram in $\strat$:
\begin{equation}\label{95}
\xymatrix{
\Link_{K'_q}(K')\underset{\Link_{K_q}(K)}\times \Link_{X_q}(X)\ar@{^{(}->}[r]\ar[d]&\ar[d] \Unzip_{K'_q}(K')\underset{\Unzip_{K_q}(K)}\times \Unzip_{X_q}(X)\\
K'_q\underset{K}\times X\ar@{^{(}->}[r]& X' .
}
\end{equation}
Using the commutation of pushouts and fiber products of topological spaces, 
we identify the pushout term for the diagram~(\ref{95}) as the fiber product
\[
\Bigl(K'_q\underset{\Link_{K'_q}(K')}\amalg\Unzip_{K'_q}(K')\Bigr)
\underset{\bigl(K_q\underset{\Link_{K_q}(K)}\amalg\Unzip_{K_q}(K)\bigr)}\times
\Bigl(X_q\underset{\Link_{X_q}(X)}\amalg\Unzip_{X_q}(X)\Bigr)
\cong
K'
\underset{K}\times
X
=:X'~.
\]
This is to say that the diagram~(\ref{95}) is a pushout diagram in $\strat$.  
In particular, there is an open map from a deleted collar-neighborhood of the boundary in $\Unzip_{K'_q}(K')\underset{\Unzip_{K_q}(K)}\times \Unzip_{X_q}(X)$,
\[
\bigl(\Link_{K'_q}(K')\underset{\Link_{K_q}(K)}\times \Link_{X_q}(X) \bigr)\times (0,1) \longrightarrow X'\smallsetminus X'_q~,
\]
onto a deleted collar-neighborhood of $X'_q\subset X'$.
Necessarily, this open map factors the likewise open map
\[
\bigl(\Link_{X_q'}(X') \bigr)\times (0,1) \longrightarrow X'\smallsetminus X'_q~.
\]
We conclude that the canonical map $\Link_{X_q'}(X')\to \Link_{K'_q}(K')\underset{\Link_{K_q}(K)}\times \Link_{X_q}(X)$ is a refinement, as desired.

It remains to show that the map~(\ref{94}) induces a homotopy equivalence
\begin{equation}\label{92}
\underset{p'}\coprod \Link_{X_q'}(X')_{|X_{p'}} \xra{~\simeq~}
\underset{q'}\coprod \Link_{K'_q}(K')_{|K'_{q'}}\underset{\underset{q} \coprod \Link_{K_q}(K)_{|K_q}}\times \underset{p}\coprod \Link_{X_q}(X)_{|X_p}
\end{equation}
between disjoint union of strata.
Through the identification of exit-path $\infty$-categories offered by Lemma~\ref{exit-description}, we recognize this map as that between spaces of morphisms induced by the functor
\begin{equation}\label{93}
\exit(X')\longrightarrow \exit(K')\underset{\exit(K)}\times \exit(X)~.
\end{equation}
Being a restricted Yoneda functor, $\exit: \Strat\ra \Cat_{\oo}$ preserves limit diagrams.
The canonical map
$X' \to K'\underset{K}\times X$
is an equivalence in $\Strat$ by applying Lemma \ref{lemma.PBstratStrat}. Consequently, the functor~(\ref{93}) is an equivalence between $\infty$-categories.
In particular, this functor~(\ref{93}) is an equivalence on spaces of morphisms.
We conclude that the map~(\ref{92}) is a homotopy equivalence, as desired, thereby completing this proof.

\end{proof}

\begin{remark}\label{rem.wcbl.dont}
Weakly constructible bundles do not pull back in $\strat$. For example, let $M$ be a compact smooth manifold with $x\in M$ a point which is not isolated.
Consider the map to the open cone on $M$
\[
[0,1) \overset{(x,-)}\longrightarrow M\times[0,1)\longrightarrow \sC(M)~.
\]
This gives rise to a limit diagram of stratified topological spaces
\[
\xymatrix{
[0,1)\underset{\sC(M)}\times M\times [0,1)\ar@{^{(}->}[r]\ar[d]&M\times [0,1)\ar[d]\\
[0,1)\ar@{^{(}->}[r]&\sC(M)}
\]
where $M\times [0,1)\ra \sC(M)$ is a weakly constructible bundle which is not constructible.
The above fiber product is isomorphic as a stratified topological space to the stratified subspace
\[
M\times\{0\}\cup\{x\}\times[0,1) \subset M\times[0,1)~.
\]
By inspection, there is no basic neighborhood about the point $(x,0) \in M\times\{0\}\cup\{x\}\times[0,1)$.
Therefore, this fiber product is not even a $C^0$ stratified space in the sense of \S1.1. Consequently, the above pullback does not exist in $\strat$.
\end{remark}

From the previous lemma we conclude that constructible covers pull back along constructible bundles.

\begin{cor}\label{cbl-covers-pullback}
For each constructible bundle $X\to K$ and each constructible cover
\[
\xymatrix{
L  \ar[r]  \ar[d]
&
\w{K} \ar[d] &\text{the pullback diagram}& X\underset{K}\times L  \ar[r]  \ar[d]
&
X\underset{K}\times\w{K} \ar[d]
\\
K_0 \ar[r]
&
K, && X\underset{K}\times K_0 \ar[r]
&
X
}
\]
is a constructible cover.

\end{cor}

\begin{proof}
Lemma~\ref{cbls-pullback} grants that the pullback square exists.
Because the first displayed diagram in the statement is a pullback, then so is the second displayed diagram.
Inspection of expression~(\ref{pullback-blowups}) verifies that the second displayed diagram is a pushout.  

\end{proof}

\begin{lemma}\label{cbls-with-R}
For each weakly constructible bundle $X\to Z\times \RR$, there is an isomorphism
\[
X~\cong~X_{|Z\times\{0\}}\times \RR
\]
under $X_{|Z\times\{0\}}\times \{0\}$ and over $Z\times \RR$.

\end{lemma}

\begin{proof}

Fix a weakly constructible bundle $X\xra{f} Z\times \RR$, and denote the stratification $X=(X\to P)$.
There is the map of posets $P\xra{\sf depth} \ZZ_{\geq 0}$ that reports the codimension of $X_p\subset X$.
For $S\subset \ZZ_{\geq 0}$, we notate $X_S:= X_{{\sf depth}^{-1}S} \subset X$ for the union of strata whose depth is an element of $S$.

Consider the parallel vector field $\partial_t$ on $Z\times \RR$. Its flow is a conically smooth map $(Z\times\RR)\times  \RR\to Z\times\RR$ given by $(z,s,t)\mapsto (z,s+t)$. This flow is defined for all time.
In a moment, we will argue that there exist a parallel vector field $V$ on $X$ lifting $\partial_t$.
Provided this, the flow of $V$ is defined for all time.
In particular, we have an isomorphism $\gamma\colon X_{|Z\times\{0\}}\times \RR\ra X$ over $Z\times \RR$, as desired.

Lemma~\ref{to-R} implies that, for each $k\geq 0$, the restriction $X_k\to Z\times \RR$ factors through some stratum $Z_q\times \RR\subset Z\times \RR$ as a fiber bundle.
Through Lemma~\ref{bdls-homot}, choose an isomorphism $X_k\cong (X_k)_{|Z_q\times\{0\}} \times \RR$ under $(X_k)_{|Z_q\times\{0\}}\times \{0\}$ and over $Z_k\times \RR$.
By way of the above isomorphism, there results a parallel vector field $V_k$ on $X_k$ lifting $\partial_t$ on $Z_q\times \RR$. For each $k\geq 0$, choose a tubular neighborhood $X_k\subset \nu_k\subset X$. (Such exists through the results of~\cite{aft1} discussed in~\S\ref{recall} of this article).
This open neighborhood $\nu_k\subset X$ is equipped with a conically smooth retraction $\nu_k\to X_k$ which is a fiber bundle, equipped with a section.
Another application of Lemma~\ref{bdls-homot} gives the isomorphism $\nu_k\cong (\nu_k)_{|Z_q\times\{0\}} \times \RR$ over and under $X_k\cong (X_k)_{|Z_q\times\{0\}} \times \RR$.
In particular, there is a parallel vector field $\w{V}_k$ on $\nu_k$ extending $V_k$ on $X_k$.
The chosen collection of tubular neighborhoods $\{\nu_k\}_{k\geq 0}$ is an open cover of $X$.
Choose a partition of unity $\{\psi_k\}_{k\geq 0}$ subordinate to this open cover.
Consider the parallel vector field on $X$
\[
V'~:=~\sum_{k\geq 0} \psi_k \w{V}_k~.
\]
We now modify $V'$ to ensure that it lies over $\partial_t$ on $Z\times \RR$.

We define a sequence $\{V_k'\}_{k\geq 0}$ of parallel vector fields on $X$ with the following properties:
\begin{itemize}
\item For each $0\leq i\leq k \leq l$, the restrictions $(V'_k)_{|X_i} = (V'_l)_{|X_i}$ agree.

\item For each $0\leq i\leq k$, the restriction $(V'_k)_{|X_i}$ lies over $\partial_t$ on $Z\times \RR$.

\end{itemize}
We define this sequence by induction on $k\geq 0$.
For $k=0$, set $V'_k=V'$.
It is immediate to verify that the restriction $V'_{|X_0}$ lies over $\partial_t$, because $V'_{|X_0} = V_0$ (for $(\psi_k)_{|X_0} = 0$ unless $k=0$).
Now suppose $V'_k$ has been constructed for $k<d>0$.
The projection $X_d \to Z\times \RR$ factors through some stratum $Z_q\times \RR$ as a fiber bundle.
In particular, the map of sheaves of parallel vector fields ${\sf D}f_{|d}:\Theta_{X_d}\ra f^\ast \Theta_{Z_q\times \RR}$ has locally constant rank, and it admits a splitting $\Theta_{X_d} \simeq {\sf Ker}({\sf D}f_{|d}) \oplus \Theta_{Z_q\times \RR}$.
By way of this splitting, choose a vector field $W_d$ on $X_d$ for which $(V'_{d-1})_{|X_{d}}-W_d$ lies over $\partial_t$ on $Z_q\times \RR$.
Because the restriction $(V'_{d-1})_{|X_{<d}}$ lies over $\partial_t$, we can take $W_d$ to vanish conically smoothly as it approaches $X_{<d}$.
In particular $(V'_{d-1})_{|X_d}\smallsetminus W_d$ gives an extension of $(V'_{d-1})_{|X_{d-1}}$ to $X_{\leq d}$.
Call this extension $W'_d$. It is a parallel vector field on $X_{\leq d}$ that agrees with $(V'_{d-1})$ on $X_{<d}$.
As performed previously, choose a regular neighborhood $X_{\leq d}\subset O\subset X$ as in~\S\ref{recall}, and extend $W'_d$ to a parallel vector field $\w{W}'_d$ on $O$.
Choose a partition of unity $\{\phi_O,\phi_{>d}\}$ subordinate to the open cover $O\cup (X\smallsetminus X_{\leq d})$ of $X$.
Take
\[
V_d~:=~ \phi_O \w{W}'_d+ \phi_{>d}V'_{d-1}~.
\]
This is a parallel vector field on $X$, and it satisfies the two required properties by construction.

The desired parallel vector field on $X$ is given by the expression
\[
V~:=~\underset{k\geq 0} \limit V'_k
\]
which is defined because $X$ is locally compact, so in particular the map ${\sf depth}\colon X\to \ZZ_{\geq 0}$ is locally bounded.
By construction, this parallel vector field lifts $\partial_t$ on $Z\times \RR$.

\end{proof}

\subsection{Decomposing constructible bundles}\label{sec.decompose}
We study how to break up total spaces of constructible bundles.
The main result of this section is Corollary~\ref{Cbl.ess}.

\begin{observation}\label{unzips-over-cones}
Each conically smooth map $X\to K\times \oC(Z)$ to the product of a smooth manifold $K$ and a closed cone determines a map between inclusions
\[
\Bigl(\Link_{X_{|K}}(X) \to \Unzip_{X_{|K}}(X)\Bigr)  \longrightarrow \Bigl(K\times Z \xra{\{0\}} K\times Z\times[0,1]\Bigr)~,
\]
over the inclusion $K \to K\times \oC(Z)$.
Furthermore, should $X\to K\times \oC(Z)$ be a constructible bundle, then each of the maps above is a constructible bundle.
\end{observation}

\begin{lemma}\label{cbls-over-cones}
For each constructible bundle $X\to K\times \oC(Z)$ over the product of a smooth manifold and a closed cone, there is an open conically smooth map under $\Link_{X_{|K}}(X)$
\[
\Link_{X_{|K}}(X) \times [0,1] \longrightarrow \Unzip_{X_{|K}}(X)
\]
over $K\times Z\times[0,1]$.  
Furthermore, for each open conically smooth map $e:X \ra Y$ of constructible bundles over $K\times \oC(Z)$, the diagram 
\[
\xymatrix{
\Link_{X_{|K}}(X)      \ar[r]^-{\{0\}}    \ar[dr]
&
\Link_{X_{|K}}(X)\times[0,1]  \ar[rr]^-{\Link(e)\times {\sf id}}     \ar@{-->}[d]
&&
\Link_{Y_{|K}}(Y)\times [0,1]  \ar@{-->}[d] 
&
\Link_{Y_{|K}}(Y)   \ar[l]_-{\{0\}}    \ar[ld] 
\\
&\Unzip_{X_{|K}}(X)  \ar[rr]^-{\Unzip(e)}  
&&
\Unzip_{Y_{|K}}(Y)
}
\]
over $K\times Z\times [0,1]$ can be filled with the dashed arrows are open conically smooth maps.  

\end{lemma}
\begin{proof}
Consider the vector field $\partial_t$ on $K\times Z\times [0,1]$. 
Its flow is the conically smooth map $\chi\colon (K\times Z\times[0,1])\times  \RR_{\geq 0}\dashrightarrow K\times Z\times[0,1]$ given by $(x,z,s,t)\mapsto (x,z,s+t)$. 
This is defined on $K\times Z\times\{(s,t)\in [0,1]\times \RR_{\geq 0}\mid 0\leq s+t\leq 1\}$.
We now now argue that there exists a vector field $V$ on $\Unzip_{X_{|K}}(X)$ which lifts $\partial_t$.  

From its construction, and after Lemma~\ref{cbls-with-R}, the constructible bundle $\Unzip_{X_{|K}}(X)\to K\times Z\times [0,1]$ admits an atlas consisting of maps among basics of the form 
$f_\alpha\times {\sf pr}:\RR^{i_\alpha}\times \sC(L_\alpha)\times I_\alpha \ra \RR^{j_\alpha}\times U_\alpha\times I_\alpha$ 
with $I_\alpha\subset [0,1]$ an open subspace which is an interval, and $\RR^{j_\alpha} \hookrightarrow K$ and $U_\alpha \hookrightarrow Z$ members of the respective atlases of $K$ and $Z$.  
In particular, there is a standard lift $V_\alpha$ of the restriction of $\partial_t$ to $\RR^{j_\alpha}\times U_\alpha \times I_\alpha$.  
Choose a conically smooth partition of unity $\{\phi_\alpha\}$ subordinate to this open cover for $\Unzip_{X_{|K}}(X)$ (such exists, as established in~\cite{aft1}).  
The vector field $V:= \sum_\alpha \phi_\alpha  V_\alpha$ defines a vector field on $\Unzip_{X_{|K}}(X)$.
By design, $V$ lifts $ \partial_t$ on $K\times Z\times [0,1]$.  

The flow of $V$ is a conically smooth map $\gamma\colon \Unzip_{X_{|K}}(X) \times \RR_{\geq 0}\dashrightarrow \Unzip_{X_{|K}}(X)$ over $\chi$, and is defined on the preimage of the domain of $\chi$.   
Furthermore, for each $t\in \RR_{\geq 0}$, the conically smooth map $V\colon \Unzip_{X_{|K}}(X) \dashrightarrow \Unzip_{X_{|K}}(X)$ is open, where it is defined. 
In particular, the restriction
\[
\Link_{X_{|K}}(X)\times [0,1] \longrightarrow \Unzip_{X_{|K}}(X)_{|K\times Z} 
\]
is an open conically smooth map over the isomorphism $\chi_1\colon (K\times Z\times\{0\})\times [0,1]\xra{\cong} K\times Z\times[0,1]$.  

Finally, the relative assertion is manifest from the above construction.  

\end{proof}

\begin{lemma}\label{cbls-open-over-cones}
For each constructible bundle $X\to K\times \oC(Z)$ over the product of a smooth manifold and a closed cone, there is a pushout diagram in $\strat$
\[
\xymatrix{
\Link_{X_{|K}}(X) \times (0,1]    \ar[rr]  \ar[d]
&&
X_{|K\times Z} \times (0,1] \ar[d]
\\
\Link_{X_{|K}}(X)\times [0,1]  \ar[rr]
&&
\Unzip_{X_{|K}}(X)
}
\]
over the trivial pushout diagram $K\times Z\times [0,1] \underset{K\times Z\times(0,1]} \coprod K\times Z\times(0,1]\cong K\times Z\times[0,1]$.  

\end{lemma}

\begin{proof}
For each open neighborhood $\Link_{X_{|K}}(X) \subset \nu\subset \Unzip_{X_{|K}}(X)$ there is an open cover in $\strat$
\[
\xymatrix{
\nu\smallsetminus \Link_{X_{|K}}(X)    \ar[rr]  \ar[d]
&&
\Unzip_{X_{|K}}(X) \smallsetminus \Link_{X_{|K}}(X) \ar[d]
\\
\nu \ar[rr]
&&
\Unzip_{X_{|K}}(X),
}
\]
thereby demonstrating a pushout diagram.  
Now, take $\nu$ to be the image of a conically smooth map
\[
\Link_{X_{|K}}(X)\times[0,1] \longrightarrow \Unzip_{X_{|K}}(X)
\]
over $K\times Z\times[0,1]$, as produced by Lemma~\ref{cbls-over-cones}.  
The result follows because this map is, in particular, a monomorphism of underlying sets.  

\end{proof}

\begin{cor}\label{cbl-cone-pushout}
For each constructible bundle $X\to K\times \oC(Z)$ over the product of a smooth manifold and a closed cone, there is a diagram among stratified spaces
\[
\xymatrix{
&&
\Link_{X_{|K}}(X) \times (0,1]    \ar[rr]  \ar[d]
&&
X_{|K\times Z} \times (0,1] \ar[dd]
\\
\Link_{X_{|K}}(X)  \ar[rr]^-{\{0\}}  \ar[d]
&&
\Link_{X_{|K}}(X)\times [0,1]   \ar[rrd]
&&
\\
X_{|K} \ar[rrrr]
&&
&&
X
}
\]
witnessing a colimit.
\end{cor}

\begin{remark}
The above span $X_{|K} \xla{\pi} \Link_{X_{|K}}(X) \xra{\gamma} X_{|K\times Z}$ has the property that $\pi$ is proper and constructible, and $\gamma$ is open; furthermore, any refinement $\Link_{X_{|K}}(X) \to L'$ factoring both $\pi$ and $\gamma$ is in fact an isomorphism.  

\end{remark}

The following is a converse to Corollary~\ref{cbl-cone-pushout}.
\begin{observation}\label{building-cbl}
For each span among stratified spaces
$
\bigl(X_0 \xla{\pi} L \xra{\gamma} X_1\bigr)
$
in which $\pi$ is proper and constructible and $\gamma$ is open, the colimit of the diagram
\[\xymatrix{X_0
&\ar[l] L \ar[r]^-{\{0\}}& L\times[0,1]&\ar[l]L\times(0,1]\ar[r]&X_1\times (0,1] \\}
\]
exists; call it $X$. Furthermore, for each natural transformation $(X_0\xla{\pi}L \xra{\gamma}X_1)\to (K\xla{\sf pr} K\times Z \xra{=}K\times Z)$ by constructible bundles in which $Z$ is compact, the resulting canonical map 
$
X\to
K\times \oC(Z)
$
too is constructible, manifestly.

\end{observation}

The next result is phrased in terms of the following categories.
Fix a compact conically smooth stratified space $Z$ and a smooth manifold $K$.
Consider the full subcategory 
\[
{\sf Cbl}(K\times Z)~\subset ~ \strat_{/K\times Z}
\]
consisting of the constructible bundles over $K\times Z$.
Consider the category
\[
{\sf Burn}'_1(K,Z)
\]
for which an object is a diagram in $\strat$
\[
\xymatrix{
X_K  \ar[d]
&
L   \ar[l]_-{\pi}  \ar[r]^-{\gamma}  \ar[d]
&
X_{K\times Z}  \ar[d]
\\
K
&
K\times Z  \ar[l]_-{\sf pr}  \ar[r]^-{=}
&
K\times Z
}
\]
in which $\pi$ is a proper constructible bundle, $\gamma$ is an open map, and the downward maps are constructible. 
A morphism is a natural transformation of such diagrams that restricts as the identity natural transformation on the bottom span; composition is given by composing natural transformations. 
Consider likewise the category over ${\sf Burn}'_1(K,Z)$,
\begin{equation}\label{11}
{\sf Burn}_1(K,Z) \longrightarrow {\sf Burn}'_1(K,Z)~,
\end{equation}
defined as follows.  
An object of ${\sf Burn}_1(K,Z)$ over an object $(X_K\xla{\pi} L{\xra{\gamma}} X_{K\times Z})\in {\sf Burn}'_1(K,Z)$ is a colimit diagram in $\strat_{/K\times \oC(Z)}$:
\[
\xymatrix{
&&
L \times (0,1]    \ar[rr]^-{\gamma \times{\sf id}}  \ar[d]
&&
X_{K\times Z} \times (0,1] \ar[dd]
\\
L \ar[rr]^-{\{0\}}  \ar[d]_-{\pi}
&&
L \times [0,1]   \ar[rrd]
&&
\\
X_K \ar[rrrr]
&&
&&
X.
}
\]
There is a unique morphism in ${\sf Burn}_1(K,Z)$ between two such, over a given morphism in ${\sf Burn}'_1(K,Z)$.
Composition in ${\sf Burn}_1(K,Z)$ is composition in ${\sf Burn}'_1(K,Z)$.
Observation~\ref{building-cbl} gives that the functor~(\ref{11}) is surjective on objects, and is therefore an equivalence between categories.
Evaluation at the cone-point determines a functor
\begin{equation}\label{0-burn-to-bun}
{\sf Burn}_1(K,Z) \longrightarrow {\sf Cbl}\bigl(K\times \oC(Z)\bigr)~.
\end{equation}

\begin{remark}\label{terminal-refinement}
The functor~(\ref{0-burn-to-bun}) is not an equivalence; it is not even conservative.
In particular, for each span $(X_0\xla{\pi} L \xra{\gamma}X_1)$ with $\pi$ proper and constructible and $\gamma$ open, any refinement $L\to L'$ factoring both $\pi$ and $\gamma$ induces the same colimit of Observation~\ref{building-cbl}.  
Nevertheless, there is a terminal such $L'$ under $L$.  

\end{remark}

We consider a relative version of ${\sf Burn}(K,Z)$ and of ${\sf Cbl}\bigl(K\times \oC(Z)\bigr)$.  
We do this through the following notion.
\begin{definition}\label{mutually.transverse}
Let $K$ be a smooth manifold.
A \emph{transverse collection in $K$} is a finite collection $K_0:=\{W\xra{f}K\}$ of smooth proper embeddings with the property that it is either empty or each proper subset $S\subset K_0$ is transverse and, for each $(W\xra{f} K) \in K_0\smallsetminus S$, this map $f$ is transverse to the smooth map 
\[
\underset{(W'\xra{f'}K)\in S} \bigcap W' \longrightarrow K
\]
from the iterated fiber product over $K$.

\end{definition}

\begin{remark}
Let $K$ be a smooth manifold.
Let $\{W_i\xra{f_i}K\}_{i=1}^n$ be a collection of smooth proper embeddings, indexed by the set $\{1,\dots,n\}$.
If $n=0$, this collection is empty, and, by definition, it is transverse.
If $n=1$, this collection is transverse.
If $n=2$, this collection is transverse if and only if $f_1\pitchfork f_2$.
If $n=3$, this collection is transverse if and only if $f_i\pitchfork f_j$ for each $i\neq j$ and also $f_i\pitchfork f_{jk}$ for each $i\neq j \neq k \neq i$ where
$f_{jk} \colon W_j \underset{K}\times W_k \to K$.  

\end{remark}

\begin{example}\label{ex.partial}
Let $q>0$.
Consider the collection of smooth maps $\{\Delta^{S}_e \to \Delta^q_e\}$, indexed by the set of subsets $S\subset \{0,\dots,q\}$ with cardinality $q$.
This collection is transverse, and we denote it simply as $K_0 =\partial \Delta^q_e$.

\end{example}

\begin{notation}\label{on.K0}
Let $\fI \colon \strat^{\op} \to \cX$ be a functor to a category with finite limits.  
For each stratified space $Y$, each smooth manifold $K$, and each transverse collection $K_0$, the value
\[
\fI(Y\times K_0) ~:=~\underset{S\in \cP(K_0)}\limit~\fI\bigl(Y\times \underset{(W\xra{f} K)\in S} \bigcap W\bigr)
\]
is the limit, indexed by the poset $\cP(K_0)$ of non-empty subsets of $K_0$, ordered by inclusion.

\end{notation}

Fix a smooth manifold $K$ and a transverse collection $K_0$ in $K$.  
For a fixed constructible bundle $\bigl(X_0\to K_0\times \oC(Z)\bigl)\in \bun\bigl(K_0\times \oC(Z)\bigr)$, 
denote the categories in upper left and upper middle of the diagram
\[
\xymatrix{
{\sf Burn}_1(K,Z; X_0)  \ar[rr]   \ar[d]
&&
{\sf Cbl}\bigl(K\times \oC(Z);X_0\bigr)  \ar[rr]  \ar[d]
&&
\ast  \ar[d]^-{\lag X_0\to K_0\times \oC(Z)\rag}
\\
{\sf Burn}_1(K,Z)     \ar[rr]^-{(\ref{0-burn-to-bun})}
&&
{\sf Cbl}\bigl(K\times \oC(Z)\bigr)  \ar[rr]
&&
{\sf Cbl}\bigl(K_0\times \oC(Z)\bigr)
}
\]
in which each square is a pullback.
In the case that $K_0=\emptyset$, or $K_0 = \{\emptyset \to K\}$, note that each of the canonical functors
\[
{\sf Burn}_1(K , Z; \emptyset) \xra{~\cong~} {\sf Burn}_1(K ,Z)
\qquad \text{ and }\qquad
{\sf Cbl}\bigl(K \times  \oC(Z);\emptyset\bigr)  \xra{~\cong~}{\sf Cbl}\bigl(K \times \oC(Z)\bigr) 
\]
is an isomorphism between categories.

In the next result we denote by $\Ar^{\open}(-)$ the full subcategory of those arrows that are by \emph{open} conically smooth maps.  
\begin{cor}\label{0burn-bun-ess}
Let $K$ be a smooth manifold, let $Z$ be a compact stratified space, and let $K_0$ be a transverse collection in $K$.
For each constructible bundle $X_0\to K_0\times \oC(Z)$, each of the functors
\[
{\sf Burn}_1(K,Z;X_0) \xra{(\ref{0-burn-to-bun})} {\sf Cbl}\bigl(K\times \oC(Z);X_0\bigr)
\]
and
\[
\Ar^{\open}\bigl({\sf Burn}_1(K,Z;X_0)\bigr) \xra{~\Ar^{\open}(\ref{0-burn-to-bun})~} \Ar^{\open}\bigl({\sf Cbl}\bigl(K\times \oC(Z);X_0\bigr)\bigr)
\]
is surjective on objects.
\end{cor}

\begin{proof}
The absolute case (in which $K_0=\emptyset$) is implied by Corollary~\ref{cbl-cone-pushout}, 
using Lemma~\ref{cbls-open-over-cones}.
The general relative case is implied likewise, by inspecting the functor~(\ref{0-burn-to-bun}).  

\end{proof}

We now show how Corollary~\ref{0burn-bun-ess} can be iterated.
Fix a smooth manifold $K$, a conically smooth stratified space $Z$, and an integer $p\geq 0$.
Consider the poset $\cP_p$ of non-empty convex subsets $S\subset \{0<\dots<p\}$ for which $p\in S$. 
Order $\cP_p$ by inclusion.  
Denote the functor 
\[
(K,Z)_p\colon \cP_p^{\op} \longrightarrow \strat
\]
defined as follows.
The value on $S$ is $K$ if $p\notin S$ and is $K\times Z$ if $p\in S$.
The value on an inclusion $(S\subset T)$ is $(K\xra{=}K)$ if $p\notin T$, is $(K\times Z\xra{=}K\times Z)$ if $p\in S$, and is the projection $(K\times Z\xra{\sf pr}K)$ if $p\in T\smallsetminus S$.  
Consider the full subcategory
\[
{\sf Burn}_{p}(K,Z)~\subset~\Fun\bigl(\cP_p^{\op},\strat\bigr)_{/(K,Z)_p}
\]
consisting of those functors $\cP_p^{\op} \xra{X_-} \strat$ over $(K,Z)_p$ that satisfy the following conditions.
\begin{itemize}
\item 
For each $S\in \cP_p$, the conically smooth map 
\[
X_S \longrightarrow (K,Z)_p(S)
\]
is a constructible bundle.  

\item 
For each relation $S\subset T$ in $\cP_p$ for which the minima ${\sf Min}(S) = {\sf Min}(T)$ agree, the conically smooth map 
\[
X_T \longrightarrow X_S
\]
is a proper constructible bundle.

\item For each relation $S\subset T$ in $\cP_p$ for which the maxima ${\sf Max}(S) = {\sf Max}(T)$ agree, the conically smooth map
\[
X_T \longrightarrow X_S
\]
is \emph{open}.

\item For each pair $S,T\in \cP_p$ for which $S\cap T \neq \emptyset$, the square among stratified spaces 
\[
\xymatrix{
X_{S\cup T}  \ar[r]  \ar[d]
&
X_T  \ar[d]
\\
X_S  \ar[r]
&
X_{S\cap T}
}
\]
is a pullback.  

\end{itemize}

Note the identification
\begin{equation}\label{79}
{\sf Burn}_0(K,Z)\xra{~=~} {\sf Cbl}(K\times Z)~.
\end{equation}
For $p>0$, this category fits into an evident square
\[
\xymatrix{
{\sf Burn}_p(K,Z)  \ar[r]  \ar[d]
&
{\sf Burn}_{p-1}(K,Z)  \ar[d]
\\
{\sf Burn}_1(K,\ast)  \ar[r]
&
{\sf Burn}_0(K,\ast)
}
\]
in which the horizontal arrows are given by restriction along the standard inclusion $\{1<\dots<p\}\hookrightarrow\{0<\dots<p\}$, and the vertical arrows are induced from the standard inclusion $\{0<1\} \hookrightarrow \{0<\dots<p\}$.
Directly so, all of the functors in this square are isofibrations.
Because constructible bundles admit base-change (Lemma~\ref{cbls-pullback}), the canonical functor to the pullback,
\begin{equation}\label{burn-segal}
{\sf Burn}_p(K,Z) \longrightarrow  {\sf Burn}_1(K,\ast)  \underset{{\sf Burn}_0(K,\ast)}\times {\sf Burn}_{p-1}(K,Z)~,
\end{equation}
is an equivalence between categories that is surjective on objects.

We now construct a functor
\begin{equation}\label{p-burn-to-bun}
{\sf Burn}_p(K,Z) \longrightarrow {\sf Cbl}\bigl(K\times \oC^p(Z)\bigr)
\end{equation}
by induction on $p$.
For $p=0$, this is the isomorphism~(\ref{79}).
For $p=1$ this is the functor~(\ref{0-burn-to-bun}).  
Now suppose $p>1$.  
There is the diagram among categories
\[
\xymatrix{
{\sf Burn}_p(K,Z)    \ar[rrr]^-{{\sf ev}_{\{0,1\}}}  \ar[d]
&&&
\Ar^{\cbl}(\strat)  \ar[d]^-{{\sf ev}_1}
\\
\Ar^{\open}\bigl({\sf Burn}_{p-1}(K,Z)\bigr) \ar[r]^-{{\sf ev}_0}
&
{\sf Burn}_{p-1}(K,Z)  \ar[r]^-{(Z\to \ast)}
&
{\sf Burn}_{p-1}(K,\ast)  \ar[r]^-{{\sf ev}_{\{0,1\}}}
&
\strat
}
\]
with the left vertical functor given by restriction along the inclusion of posets
\[
\cup \colon \cP_{p-1}\times [1]\cong \cP_{\{1<\dots<p\}}\times \bigl\{\emptyset\subset\{0\}\bigr\} ~\cong~ \bigl\{S\in \cP_p\mid S \neq \{0\}\bigr\} ~\subset ~\cP_p~.
\] 
The canonical functor to the fiber product,
\begin{equation}\label{another-pb}
{\sf Burn}_p(K,Z) \xra{~\simeq~} \Ar^{\open}\bigl({\sf Burn}_{p-1}(K,Z)\bigr) \underset{\strat} \times \Ar^{\cbl}(\strat)~,
\end{equation}
is a surjective equivalence between categories
since there is a pushout expression among posets,
\[
\bigl\{0\subset \{0,1\}\bigr\}\underset{\{0,1\}}\coprod   \bigl\{S\in \cP_p\mid S \neq \{0\}\bigr\}\xra{~\cong~}\cP_p~,
\]
and because constructible bundles admit base-change (Lemma~\ref{cbls-pullback}).  
We now define the functor~(\ref{p-burn-to-bun}) as composite of the sequence of functors
{\Small
\begin{equation}\label{p-diag}
\xymatrix{
{\sf Burn}_p(K,Z)    \ar[d]^-{(\ref{another-pb})}_-{\simeq}
&&
{\sf Burn}_1\bigl(K,\oC^{p-1}(Z)\bigr)   \ar[r]^-{(\ref{0-burn-to-bun})}   \ar[d]^-{(\ref{another-pb})}_-{\simeq}
&
{\sf Cbl}\bigl(K\times \oC^p(Z)\bigr)
\\
\Ar^{\open}\bigl({\sf Burn}_{p-1}(K,Z)\bigr) \underset{\strat} \times \Ar^{\cbl}(\strat)  \ar[rr]^-{\rm induction}
&&
\Ar^{\open}\bigl({\sf Cbl}\bigl(K\times \oC^{p-1}(Z)\bigr) \underset{\strat} \times \Ar^{\cbl}(\strat)   
&
.
}
\end{equation}
}

Now, fix a transverse collection $K_0$ in $K$, and fix a constructible bundle $X_0 \to K_0\times \oC^p(Z)$.  
Denote the categories in upper left and upper middle of the diagram
\[
\xymatrix{
{\sf Burn}_p(K,Z; X_0)  \ar[rr]   \ar[d]
&&
{\sf Cbl}\bigl(K\times \oC^p(Z);X_0\bigr)  \ar[rr]  \ar[d]
&&
\ast  \ar[d]^-{\lag X_0\to K_0\times \oC(Z)\rag}
\\
{\sf Burn}_p(K,Z)     \ar[rr]^-{(\ref{p-burn-to-bun})}
&&
{\sf Cbl}\bigl(K\times \oC^p(Z)\bigr)  \ar[rr]
&&
{\sf Cbl}(K_0\times \oC^p(Z))
}
\]
in which each square is a pullback.
In the case that $K_0=\emptyset$, note that each of the canonical functors
\[
{\sf Burn}_p(K,Z; \emptyset) \xra{~\cong~} {\sf Burn}_p(K,Z)
\qquad \text{ and }\qquad
{\sf Cbl}\bigl(K\times \oC^p(Z);\emptyset\bigr)  \xra{~\cong~}{\sf Cbl}\bigl(K\times \oC^p(Z)\bigr) 
\]
is an isomorphism between categories.

This inductive definition of the functor~(\ref{p-burn-to-bun}) facilitates the next result.
\begin{cor}\label{gpd-ess}
For each transverse collection $K_0$ in a smooth manifold $K$, each compact conically smooth stratified space $Z$, each $p\geq 0$, and each constructible bundle $X_0\to K_0\times \oC^p(Z)$, each of the functors
\[
{\sf Burn}_p(K,Z;X_0) \xra{(\ref{p-burn-to-bun})} {\sf Cbl}\bigl(K\times \oC^p(Z);X_0\bigr)
\]
and
\[
\Ar^{\open}\bigl({\sf Burn}_p(K,Z;X_0)\bigr) \xra{~\Ar^{\open}(\ref{p-burn-to-bun})~} \Ar^{\open}\bigl({\sf Cbl}\bigl(K\times\oC^p(Z);X_0\bigr)\bigr)
\]
are surjective on objects.

\end{cor}

\begin{proof}
Surjectivity on objects of the second functor follows a similar argument as for the first functor.
We proceed by induction on $p$.  
The case $p=0$ is the identification~(\ref{79}).
The case $p=1$ is Corollary~\ref{0burn-bun-ess}.
Now suppose $p>1$.  
From the definition of~(\ref{p-burn-to-bun}) in terms of the diagram~(\ref{p-diag}), it is enough to verify that each of the arrows in~(\ref{p-diag}) is surjective on objects.
The vertical functors are equivalences between categories.
The upper right horizontal functor is surjective on objects, as the $p=1$ case.
The functor labeled as ``induction'' is base change of the functor
\[
\Ar^{\open}\bigl({\sf Burn}_{p-1}(K,Z)\bigr)
\longrightarrow 
\Ar^{\open}\bigl({\sf Cbl}\bigl(K\times \oC^{p-1}(Z)\bigr)~.
\]
By induction, this latter functor is surjective on objects.
We conclude that the functor ${\sf Burn}_p(K,Z;X_0) \to {\sf Cbl}\bigl(K\times \oC^p(Z);X_0\bigr)$ is surjective on objects, as desired.

\end{proof}

We insert the following observation, which is implied by Corollary~\ref{gpd-ess}.
\begin{cor}\label{link-link}
Let $X\to K\times \oC^2(Z)$ be a constructible bundle over a double closed cone.  
There are open conically smooth maps $\gamma_{1Z}$, $\gamma_{0Z}$, $\gamma_{01Z}$, and $\gamma_{010Z}$ over $K\times Z$, as well as $\gamma_{01}$,  fitting into a commutative diagram
\[
\Small
\xymatrix{
\Link_{X_{|K}}(X_{|K\times \oC^{\{0\}}(Z)})  \ar@(u,u)[drrr]^-{\gamma_{0Z}}  \ar@(l,l)[dddr]_-{\pi_{0Z}}  
&&&
\\
&
\Link_{\Link_{X_{|K}}(X_{|K\times \oC^2(\emptyset)})}\bigl(\Link_{X_{|K}}(X)\bigr)  \ar[r]_-{\gamma_{01Z}}  \ar[d]^-{\pi_{01Z}}  \ar[ul]^-{\gamma_{010Z}}
&
\Link_{X_{|\oC^{\{1\}}(\emptyset)}}\bigl(X_{|K\times \oC^{\{1\}}(Z)}\bigr)   \ar[d]^-{\pi_{1Z}}  \ar[r]_-{\gamma_{1Z}}
&
X_{|K\times Z}
\\
&
\Link_{X_{|K}}(X_{|K\times \oC^2(\emptyset)})  \ar[r]_-{\gamma_{01}}    \ar[d]^-{\pi_{01}}
&
X_{|K\times \oC^{\{1\}}(\emptyset)}  
&
\\
&
X_{|K}
}
\]
in which the square is a pullback, and the map $\gamma_{010Z}$ is a refinement from the pullback.

\end{cor}

\begin{cor}\label{Cbl.ess}
For $K_0$ a transverse collection in a smooth manifold $K$ and a compact conically smooth stratified space $Z$. For each $0\leq k <p$ and each constructible bundle $X_0\to K_0\times \oC^p(Z)$, the canonical functor
\[
{\sf Cbl}\bigl(K\times \oC^p(Z);X_0\bigr)    \longrightarrow    
{\sf Cbl}\bigl(K\times \oC^k(\ast);(X_0)_{|K_0\times \oC^k(\ast)}\bigr)
\underset{{\sf Cbl}(K;(X_0)_{|\ast})}\times 
{\sf Cbl}\bigl(K\times \oC^{p-k}(Z);(X_0)_{|K_0\times \oC^{p-k}(Z)}\bigr)
\]
is surjective on objects.

\end{cor}

\begin{proof}
The functors~(\ref{p-burn-to-bun}) supply the downward functors in the diagram of categories
{\Small
\begin{equation}\label{25}
\xymatrix{
{\sf Burn}_p(K,Z;X_0) \ar[d]  \ar[dr]
&
\\
{\sf Cbl}\bigl(K\times \oC^p(Z); X_0\bigr) \ar[dr] 
&
{\sf Burn}_k(K,\ast; (X_0)_{|K_0\times \oC^k(\ast)})\underset{{\sf Burn}_0(K,\ast;(X_0)_{|\ast}))}\times {\sf Burn}_{p-k}(K,Z;(X_0)_{|K_0\times \oC^{p-k}(Z)})  \ar[d]
\\
&
{\sf Cbl}\bigl(K\times \oC^k(\ast);(X_0)_{|K_0\times \oC^k(\ast)}\bigr)
\underset{{\sf Cbl}(K;(X_0)_{|\ast})}\times 
{\sf Cbl}\bigl(K\times \oC^{p-k}(Z);(X_0)_{|K_0\times \oC^{p-k}(Z)}\bigr)
.
}
\end{equation}
}
The top diagonal functor is a surjective equivalence of categories, as implied by the identification~(\ref{burn-segal}).  
Corollary~\ref{gpd-ess} gives that the left downward functor is surjective on objects.
The right downward functor is given by taking fiber products of the horizontal diagrams in the commutative diagram among categories:
\[
\xymatrix{
{\sf Burn}_k(K,\ast; (X_0)_{|K_0\times \oC^k(\ast)})  \ar[r]  \ar[d]
&
{\sf Burn}_0(K,\ast;(X_0)_{|\ast}))  \ar[d]
&
{\sf Burn}_{p-k}(K,Z;(X_0)_{|K_0\times \oC^{p-k}(Z)})  \ar[l]  \ar[d]
\\
{\sf Cbl}\bigl(K\times \oC^k(\ast);(X_0)_{|K_0\times \oC^k(\ast)}\bigr)  \ar[r]
&
{\sf Cbl}(K;(X_0)_{|\ast})
&
{\sf Cbl}\bigl(K\times \oC^{p-k}(Z);(X_0)_{|K_0\times \oC^{p-k}(Z)}\bigr).  \ar[l]
}
\]
The middle downward functor is~(\ref{79}), which is an isomorphism between categories.
Corollary~\ref{gpd-ess} gives that the outer downward functors are surjective on objects. 
It follows that the right downward functor in~(\ref{25}) is surjective on objects.
We conclude that the bottom diagonal functor is surjective on objects, as desired.

\end{proof}

\subsection{Classifying constructible bundles}
We examine a category of constructible bundles, through which we define the $\infty$-category $\Bun$. First, recall from the end of~\S\ref{sec.stratified-spaces} the notion of \emph{finitary}.  
\begin{definition}[$\bun$]\label{def.bun}
An object of the category $\bun$ is a constructible bundle $f:X\ra K$ whose fibers are finitary.  
A morphism from $(X\xra{f} K)$ to $(X'\xra{f'} K')$ is a pullback diagram among stratified spaces
\[
\xymatrix{
X  \ar[d]_-f     \ar[r]^{g'}
&
X'\ar[d]^-{f'}
\\
K   \ar[r]_-g  
& 
K'~.
}
\] 
Composition is given by concatenating such squares horizontally, then composing horizontal maps.  

\end{definition}

\begin{lemma}\label{bun-right-fibration}
The forgetful functor 
\[
\bun\longrightarrow \strat~,\qquad (X\to K)\mapsto K~,
\]
is a right fibration, and as so it is a cone-local sheaf.  
\end{lemma}
\begin{proof} 
Lemma~\ref{cbls-pullback} immediately implies that this projection is a Cartesian fibration.  By design, the fibers are groupoids.  It follows that it is a right fibration.
Observation~\ref{cbls-cover}, together with the fact that open covers pull back, implies $\bun \to \strat$ is a sheaf.  
Corollary~\ref{cbl-covers-pullback} implies $\bun\to \strat$ is cone-local.  

\end{proof}

Recall the notion of a transversality sheaf from Definition~\ref{def.transversality}.
\begin{theorem}\label{isaquasi-cat}
The functor $\bun\ra \strat$ is a transversality sheaf.
\end{theorem}

\begin{notation}
Lemma~\ref{bun-right-fibration} states that $\bun \to \strat$ is a right fibration and a sheaf.
So we fix, once and for all, a straightening $\bun\colon \strat^{\op} \to \gpd$. 
We will write $\bun(K)$ for the fiber of $\bun\ra \strat$ over $K\in \strat$.
This is the value on $K$ of the functor $\bun: \strat^{\op} \ra \gpd$ which is the straightening the right fibration $\bun\ra \strat$.
\end{notation}

We establish the following isofibration property; recall Remark \ref{rem.isofib} regarding our use of isofibrations.

\begin{lemma}\label{iso.fib.opens}
For each open subset $U\subset Z$ of a stratified space, the restriction functor
\[
\bun(Z) \longrightarrow \bun(U)
\]
is an isofibration.
\end{lemma}
\begin{proof}
Let $X\to Z$ and $Y_U \ra U$ be constructible bundles, and let $\alpha_U\colon X_{|U} \cong Y_U$ be an isomorphism over $U$.
We must construct an isomorphism $\alpha \colon X\cong Y$ over $Z$ extending $\alpha_U$.  
Choose an open subset $V\subset Z$ for which $\{U,V\}$ is an open cover of $Z$.
This determines a commutative diagram
\[
\xymatrix{
X_{|U}  \ar[d]_-{\alpha_U}
&&
X_{|U\cap V}  \ar[rr]  \ar[ll]  \ar[d]^-{=}
&&
X_{|V}  \ar[d]^-{=}
\\
Y_U  
&&
X_{|U\cap V}  \ar[rr]  \ar[ll]_-{(\alpha_{U})_{|X_{|U\cap V}}}
&&
X_{|V}
}
\]
in which the unlabeled arrows are the canonical inclusions.  
There results an isomorphism between pushouts in the category $\strat$ (which, in this case, exist):
\[
\alpha\colon Y~:=~ Y_U \underset{X_{|U\cap V}} \coprod X_{|V}  ~\cong~  X_{|U} \underset{X_{|U\cap V}} \coprod X_{|V}~=~X~,
\]
as desired.

\end{proof}

\begin{proof}[Proof of Theorem \ref{isaquasi-cat}]

We will employ the conditions of Theorem~\ref{transversality.conditions} in showing that $\bun\ra \strat$ is a transversality sheaf.

\noindent
{\bf Isotopy extension:}

We are given the right fibration $\bun \to \strat$ between ordinary categories.
We  must show, for each stratified space $K$ and each $0\leq i \leq q$, that the functor between groupoids
\begin{equation}\label{39}
\bun(K\times \Delta^q_e)
\longrightarrow
\bun(K\times \partial \Delta^q_e) 
\end{equation}
is an isofibration, and that the functor between groupoids
\begin{equation}\label{40}
\bun(K\times \Delta^q_e)
\longrightarrow
\bun(K\times (\Lambda^q_i)_e) 
\end{equation}
is essentially surjective and surjective on Hom-sets. (See Notation \ref{simp.not} for the righthand terms.)

We first establish that the functor~(\ref{39}) is an isofibration. 
Consider a finite open cover $\cO$ of $\Delta^q_e$ that is closed under finite intersections, and for which, for each $U\in \cO$, there is a diffeomorphism $\phi_U\colon U \cong \RR^q$ with the following property:
\begin{itemize}
\item[~]
for each non-empty subset $S\subset \{0,\dots,r\}$ for which $U\cap \Delta^S_e \neq \emptyset$ is not empty, the composition 
\[
U\cap \Delta^S_e \hookrightarrow U \xra{~\phi_U~} \RR^q
\]
is a diffeomorphism onto a coordinate vector subspace $\RR^{|S|-1}\subset \RR^q$.  

\end{itemize}
(Such an open cover exists, by choosing complements of skeleta of the topological simplex $\Delta^q$.)
For each $0\leq a \leq  q$, consider the poset $\cP_a$ of proper subsets of $\{0,\dots,q\}$, ordered by inclusion.
Consider the functor 
\[
\cP_q  \longrightarrow   \strat
~,\qquad
S\mapsto \RR^S~,
\]
carrying inclusions to coordinate inclusions.  
By design then, for each $U\in \cO$ there is an $0\leq a \leq q$ for which such a $\phi_U$ identifies the restriction functor $\bun(K\times U) \to \bun(K\times \partial \Delta_e \cap U)$ as the canonical functor
\begin{equation}\label{35}
\bun(K\times \RR^q) \longrightarrow \underset{S\in \cP_a}\limit \bun(K\times \RR^{q-a}\times \RR^S)~.
\end{equation}
Through Lemma~\ref{iso.fib.opens}, the fact that $\bun$ is a sheaf on $\strat$ reduces the problem of showing~(\ref{39}) is an isofibration to that of showing the functor~(\ref{35}) is an isofibration.

So let $X\to K\times \RR^q$ be a constructible bundle, and let $(X_{|K\times \RR^{q-a}\times \RR^S}\xra{\alpha_S} Y_S)_{S\in \cP_a}$ be an isomorphism in the codomain of~(\ref{35}).  
We must extend this system of isomorphisms to an isomorphism $X\cong Y$ over $K\times \RR^q$.  
Denote the restriction $X_0:=X_{|K\times\{0\}}$. 
Lemma~\ref{cbls-with-R} gives an isomorphism $X\cong X_0\times \RR^q$ over $K\times \RR^q$ and under $X_0$.  
Composing with this isomorphism, we can assume that $X=X_0\times \RR^q$.  
Consider the functor $\cP_a^{\op}\to \strat$ whose value on $S\in \cP_a$ is $Y_S$ and whose value on an inclusion $S'\subset S$ is the composition 
\[
Y_S \xra{\alpha_S^{-1}} X_0\times \RR^{q-a}\times \RR^S \xra{{\sf id}_{X_0\times \RR^{q-a}}\times {\sf pr}} X_0\times \RR^{q-a}\times \RR^{S'} \xra{\alpha_{S'}} Y_{S'}~;
\]
here, ${\sf pr}$ is the evident coodinate projection.
By construction, the system of isomorphisms $(\alpha_S)$ determines a map betwen limits (which exist, in this case):
\[
\alpha\colon X_0\times \RR^q  \cong   X_0\times  \underset{S\in \cP_a}\limit \RR^{q-a}\times \RR^S  \cong  \underset{S\in \cP_a}\limit X_0\times \RR^{q-a}\times \RR^S 
\xra{~{}~\cong~{}~}
\underset{S\in \cP_a}\limit Y_S ~=:~ Y
\]
over the limit $K \times \RR^q  \cong   K \times  \underset{S\in \cP_a}\limit \RR^{q-a}\times \RR^S  \cong  \underset{S\in \cP_a}\limit K \times \RR^{q-a}\times \RR^S$.
This is the sought isomorphism extending the system $(\alpha_S)$.

We now show that the functor~(\ref{40}) is surjective on objects and on Hom-sets.
The case $q=0$ is immediate, since $\bun(\emptyset) =\{\emptyset\}$ is terminal.
In the case that $q=1$, there is an isomorphism $((\Lambda^q_i)_e \subset \Delta^q_e) \cong (\{0\} \subset \RR)$.  
Surjectivity on objects and on Hom-sets of~(\ref{40}) follows from the observation that for any stratified space $Z$, the map ${\sf id}\times\{0\}Z\ra Z\times \RR$ is a section of the projection $Z\times \RR \ra Z$.

We proceed by induction on $q> 1$.
Note the isomorphism $(\Lambda^q_i)_e \cong (\Lambda^q_0)_e$ between diagrams of smooth manifolds; we therefore only consider the case $i=0$, for simplicity. 
We begin by showing that surjectivity on Hom-sets is implied by surjectivity on objects.
Let $\w{X}_0 \to K\times \Delta^q_e$ and $\w{X}_1 \to K\times \Delta^q_e$ be constructible bundles, and let 
\[
\alpha\colon X_0 := (\w{X}_1)_{|K\times (\Lambda^q_i)_e}~\cong~(\w{X}_0)_{|K\times (\Lambda^q_i)_e}=: X_1
\]
be an isomorphism between their restrictions over $K\times (\Lambda^q_i)_e$.  
We must extend the isomorphism $\alpha$ to an isomorphism $X_0\cong X_1$ between constructible bundles over $K\times \Delta^q_e$.  
The isomorphism $\alpha$ determines a constructible bundle $X\to K\times (\Lambda^q_i)_e \times \RR$ whose restrictions are identified
\[
X_{|K\times (\Lambda^q_i)_e \times \{0\}}~=~X_0
\qquad \text{ and }\qquad
X_{|K\times (\Lambda^q_i)_e \times \{1\}}~ = ~X_1~.
\]
The triple $(X,X_0\coprod X_1,\w{X}_0\coprod \w{X}_1)$ is then an object of the codomain of the functor between groupoids
\[
\bun(K\times \Delta^q_e\times \RR) 
\longrightarrow
\bun(K\times (\Lambda^q_i)_e \times \RR) 
\underset{\bun(K\times (\Lambda^q_i)_e \times \{0,1\})}\times
\bun(K\times \Delta^q_e \times \{0,1\})~.
\]
Provided the functor~(\ref{40}) is surjective on objects, Observation~\ref{two.formals} gives that this functor is surjective on objects.  
We can therefore choose a constructible bundle $\w{X} \to K\times \Delta^q_e\times \RR$ whose restrictions are identified
\[
\w{X}_{|K\times (\Lambda^q_i)_e \times \RR }~=~X~
\qquad ~,~\qquad
\w{X}_{|K\times (\Lambda^q_i)_e \times \{0,1\}}~=~X_0\coprod X_1~,
\qquad \text{ and }\qquad
X_{|K\times \Delta^q_e \times \{0,1\}}~ = ~\w{X}_0\coprod \w{X}_1~.
\]
Lemma~\ref{cbls-with-R} gives an isomorphism $\w{\alpha}\colon \w{X} \cong \w{X}_0\times \RR$ between constructible bundles over $K\times \Delta^q_e\times \RR$ under $\w{X}_0$.  
The sought isomorphism is the restriction $\w{\alpha}_{|\{1\}}\colon \w{X}_1\cong \w{X}_0\times \{1\}$.

We are therefore reduced to showing that the functor~(\ref{40}) is surjective on objects.
Note the isomorphism among diagrams of smooth manifolds
\[
(\Lambda^q_0)_e~\cong~ \Bigl(\Delta^{q-1}_e \leftarrow (\Lambda^{q-1}_0)_e \rightarrow (\Lambda^{q-1}_0)_e \times \RR    \Bigr)~.
\]
Note that this isomorphism lies over an isomorphism of diagrams of smooth manifolds
\[
\Delta^q_e~\cong~ \Bigl(\Delta^{q-1}_e \leftarrow \Delta^{q-1}_e \rightarrow \Delta^{q-1}_e \times \RR    \Bigr)~.
\]
Through these isomorphisms, the functor~(\ref{40}) is recognized as the functor
\begin{eqnarray}
\nonumber
\bun(K\times \Delta^q_e)
&
\xra{~\cong~}
&
\limit \Bigl(\bun(K\times \Delta^{q-1}_e) \to \bun(K\times \Delta^{q-1}_e) \leftarrow \bun(K\times \Delta^{q-1}_e \times \RR)    \Bigr)
\\
\nonumber
&
\longrightarrow
&
\limit \Bigl(\bun(K\times \Delta^{q-1}_e) \to \bun(K\times (\Lambda^{q-1}_0)_e) \leftarrow \bun(K\times (\Lambda^{q-1}_0)_e \times \RR)    \Bigr)
\\
\nonumber
&
\xla{~\cong~}
&
\bun(K\times (\Lambda^q_i)_e) 
\end{eqnarray}
in which the middle functor is that induced by taking limits of the horizontal diagrams in the diagram
\begin{equation}\label{41}
\xymatrix{
\bun(K\times \Delta^{q-1}_e)  \ar[r]  \ar[d]
&
\bun(K\times \Delta^{q-1}_e)  \ar[d]
&
\bun(K\times \Delta^{q-1}_e \times \RR)  \ar[l]  \ar[d]
\\
\bun(K\times \Delta^{q-1}_e) \ar[r]
&
\bun(K\times (\Lambda^{q-1}_0)_e) 
&
\bun(K\times (\Lambda^{q-1}_0)_e \times \RR)    \ar[l]
}
\end{equation}
among groupoids.
Note that the left downward functor and the upper rightward functor are the identity functors.
Therefore, surjectivity of the functor~(\ref{40}) on objects is implied by surjectivity on objects of the functor
\begin{equation}\label{42}
\bun(K\times \Delta^{q-1}_e \times \RR) 
\longrightarrow
\bun(K\times \Delta^{q-1}_e)  \underset{\bun(K\times (\Lambda^{q-1}_0)_e) }\times  \bun(K\times (\Lambda^{q-1}_0)_e \times \RR) ~.
\end{equation}
By Observation~\ref{two.formals}, this functor is an isofibration. 
Therefore, surjectivity of~(\ref{42}) on objects is implied by essential surjectivity of~(\ref{42}), which we now establish.
Choose a constructible bundle $\w{X}_0 \to K\times \Delta^{q-1}_e$ and an extension $X\to K\times (\Lambda^{q-1}_0)_e \times \RR$ of its restriction $X_0:=(\w{X}_0)_{|K\times (\Lambda^{q-1}_0)_e}$.
Consider the product constructible bundle $\w{X}_0\times \RR \to K\times \Delta^{q-1}_e\times \RR$.
Lemma~\ref{cbls-with-R} provides an isomorphism between constructible bundles over $K\times (\Lambda^{q-1}_0)_e \times \RR$,
\[
\alpha\colon X~\cong~ X_0 \times \RR~\cong~(\w{X}_0\times \RR)_{|K\times (\Lambda^q_i)_e\times \RR}~,
\]
under $X_0$.  
We conclude that the functor~(\ref{42}) is essentially surjective, as desired.

\noindent
{\bf Consecutive:}
In this proof, for $D$ and $E$ finite simplicial sets, we use the notation
\[
\bun(D\times E_e)~:=~\limit \bigl((\bDelta_{/D})^{\op}\times (\bDelta_{/E})^{\op} \to \bDelta^{\op}\times \bDelta^{\op} \xra{\bun(\Delta^\bullet\times \Delta^{\bullet'}_e)} {\sf Gpd}  \bigr)~.
\]
Because, for each simplicial set $C$, the full subcategory of $\bDelta_{/C}$ consisting of the non-degenerate simplices of $C$ is final, this limit can be computed as a finite limit.  
Also, in this proof, for each $0<k<p$, we denote the finite simplicial set 
\[
D^p_k~:=~\Delta[\{0<\dots<k\}]\underset{\Delta[\{k\}]}\amalg \Delta[\{k<\dots<p\}]~.
\]

Let $0<k<p$ and $q\geq 0$.
With the notation from just above, we must show that the canonical functor between groupoids,
\begin{equation}\label{20}
\bun(\Delta^p\times \Delta^q_e) \longrightarrow \bun(D^p_k \times \Delta^q_e)\underset{\bun(D^p_k\times \partial \Delta^q_e)}\times \bun(\Delta^p\times \partial \Delta^q_e)~,
\end{equation}
is surjective on objects and on Hom-sets.

We first establish surjectivity of~(\ref{20}) on objects.
Let $(X_0\to \Delta^p\times \partial \Delta^q_e) \in \bun(\Delta^p\times \partial \Delta^q_e)$ be an object.
The desired surjectivity is implied by such surjectivity on fiber groupoids over $X_0$.  
We are therefore reduced to showing that the canonical functor
\begin{equation}\label{21}
\bun(\Delta^p\times \Delta^q_e;X_0) \longrightarrow 
\bun(D^p_k \times \Delta^q_e;(X_0)_{| D^p_k \times \partial \Delta^q_e})
\end{equation}
is surjective on objects.
By definition, $\bun(Y) := {\sf Cbl}(Y)^\sim$ is the maximal subgroupoid for each stratified space $Y$.  
Surjectivity of~(\ref{21}) on objects therefore follows directly from Corollary~\ref{Cbl.ess}, 
applied to the case $Z=\ast$ so that $\oC^p(\ast) \cong \Delta^p$ and $K=\Delta^q_e$ with $K_0 = \partial \Delta^q_e$ (as in Example~\ref{ex.partial}).

We now establish surjectivity of~(\ref{20}) on Hom-sets.
Let $X \to \Delta^p\times \Delta^q_e$ and $Y\to \Delta^p\times \Delta^q_e$ be constructible bundles, and let $\alpha_{|}\colon X_{|} \cong Y_{|}$ be an isomorphism between their images by the functor~(\ref{20}).  
We must extend $\alpha_{|}$ as an isomorphism $\alpha\colon X\cong Y$ over $\Delta^p\times \Delta^q_e$.
Using the given isomorphism $\alpha_{|}$, choose an object of 
\[
Z\in \bun(D^p_k \times \Delta^q_e\times \Delta^1_e)\underset{\bun(D^p_k\times \partial \Delta^q_e\times \Delta^1_e)}\times \bun(\Delta^p\times \partial \Delta^q_e\times \Delta^1_e)
\]
whose restriction
\[
Z_{|\partial \Delta^1_e}~=~X_{|} \amalg Y_{|}~
\in
~\bun(D^p_k \times \Delta^q_e\times \partial \Delta^1_e)\underset{\bun(D^p_k\times \partial \Delta^q_e\times \partial \Delta^1_e)}\times \bun(\Delta^p\times \partial \Delta^q_e\times \partial \Delta^1_e)
\]
is identical over $D^p_k\times \Delta^q_e\times \partial \Delta^1_e\cong D^p_k\times \Delta^q_e \amalg D^p_k\times \Delta^q_e$ to the coproduct of the restrictions of $X$ and $Y$.
The constructible bundle $(X\amalg Y \to \Delta^p \times \Delta^q_e\times \partial \Delta^1_e)$ together with $Z$ thus define an object of the codomain of the canonical functor
\begin{equation}\label{22}
\bun(\Delta^p\times \Delta^q_e\times\Delta^1_e) 
\longrightarrow
 \bun(D^p_k\times \Delta^q_e\times \Delta^1_e)\underset{\bun(D^p_k\times \Delta^q_e\times \partial \Delta^1_e)}\times \bun(\Delta^p\times \Delta^q_e\times \partial \Delta^1_e)~.
\end{equation}
Surjectivity on objects of this functor follows directly from Corollary~\ref{Cbl.ess}, applied to the case $Z=\ast$ so that $\oC^p(\ast) \cong \Delta^p$ and $K=\Delta^q_e\times \Delta^1_e$ with $K_0 =  \bigl\{\Delta^q_e\times \Delta^{\{0\}}_e~,~\Delta^q_e\times \Delta^{\{1\}}_e\bigr\}$.
We can therefore choose a constructible bundle $\w{Z} \to \Delta^p\times \Delta^q_e\times \Delta^1_e$ extending $Z$ and $X$ and $Y$.  
Lemma~\ref{cbls-with-R} grants an isomorphism $\w{\alpha}\colon \w{Z}\cong X\times \Delta^1_e$ over $\Delta^p\times \Delta^q_e\times \Delta^1_e$.
Given by flowing a parallel vector field, this isomorphism $\w{\alpha}$ can be ensured to extend the given isomorphism $\alpha_{|}$.  
The desired isomorphism $\alpha\colon X \cong Y$ is the restriction of $\w{\alpha}_{|\Delta^p\times \Delta^q_e\times \Delta^{\{1\}}_e}$.

\noindent
{\bf Univalent:}
As in Remark~\ref{class-vs-top}, to prove univalence for $\bun$ amounts to proving that the map of spaces
\[
\Bun(\Delta^0)  \longrightarrow   \Bun(E^1)
\]
is an equivalence, where $E^1$ is the simplicial stratified space corresponding, via $\Delta^\bullet$, to the nerve of the free isomorphism.
Noncanonically, we can identify these spaces as disjoint unions indexed by the collection of isomorphism classes $[X] \in \pi_0\Bun(\Delta^0)$, as
\[\xymatrix{
\Bun(\Delta^0) \simeq\underset{[X]}\coprod\Aut(X)&{\rm and}&\Bun(E^1) \simeq \underset{[X]}\coprod\Emb(X)~.\\}
\]
Here $\Aut(X)$ is the space of (conically smooth) automorphisms of $X$, and $\Emb(X) := \Emb^\sim(X,X)$ is the space of (conically smooth) self-embeddings of $X$ which are isotopic to an automorphism (via a conically smooth isotopy consisting of embeddings). To show the desired equivalence, we therefore reduce to showing that for any finitary stratified space $X$, the natural map
\[\Aut(X) \longrightarrow \Emb(X)\]
is a homotopy equivalence of Kan complexes. 
We now prove this equivalence.

Since $X$ is finitary, it is the interior of a compact conically smooth stratified space with boundary $\ov{X}$. We choose an open collar $C \cong \partial \ov{X}\times (0,1]$ of the boundary $\partial \ov{X}$, and then set $Z := X\smallsetminus C$ to be the complement of the open collar. By sucking in along the parametrization of the collar, we can define a 1-parameter family of automorphisms
\[
\phi: (0,1]\longrightarrow \Aut(X)
\]
such that:
\begin{itemize}
\item $\phi_1 = {\sf id}_X$;
\item  $\phi_{t|_Z} = {\sf id}_Z$ for each $t\in (0,1]$;
\item for each $x\in C$ and each open $U\subset X$ containing $Z$, there exists $t$ such that $\phi_s(x)\in U$ for all values $0<s<t$.
\end{itemize}  
Fix such a $Z$ and such a $\phi$. Denote by $\Emb_Z(X)\subset \Emb(X)$ the submonoid of those self-embeddings of $X$ that are the identity map on $Z$, and set $\Aut_Z(X) = \Aut(X)\cap \Emb_Z(X)$.
Since $Z$ is compact, the isotopy extension theorem grants that the natural restrictions $\Aut(X) \ra \Emb(Z,X)$ and $\Emb(X)\ra \Emb(Z,X)$ are surjective Kan fibrations. We thus have a natural map of homotopy fiber sequences
\[\xymatrix{
\Aut_Z(X) \ar[r]\ar[d]&\Aut(X) \ar[r]\ar[d]&\Emb(Z,X)\ar@{=}[d]\\
\Emb_Z(X) \ar[r]&\Emb(X) \ar[r]&\Emb(Z,X)\\}
\]
which is the identity on the base $\Emb(Z,X)$. Consequently, to prove the equivalence of $\Aut(X)\ra \Emb(X)$, we can reduce to proving that the map of fibers $\Aut_Z(X)\ra \Emb_Z(X)$ is an equivalence. To do so, we construct an explicit homotopy inverse using our previously chosen $\phi$.

Consider the map
\[
\gamma\colon \Emb_Z(X) \times (0,1] \longrightarrow \Emb_Z(X)~,\qquad (f,t)\mapsto \phi_t^{-1} \circ f \circ \phi_t
\]
given by conjugating with $\phi$.
Since $\Emb_Z(X)$ consists of conically smooth maps, there is an extension of $\gamma$ to a map $\ov{\gamma}\colon \Emb_Z(X)\times [0,1]\ra \Strat_Z(X,X)$. We denote the restriction of this extension along $\{0\}$ as
\[
D_Z \colon \Emb_Z(X) \longrightarrow \Strat_Z(X,X)~.
\]
To conclude the argument, we show that the values of $D_Z$ are automorphisms of $X$.
Let $f\colon X\to X$ be an embedding that is the identity on $Z$.
Then there is an open neighborhood $Z\subset O\subset X$ on which there is a conically smooth inverse $f^{-1}\colon X\dashrightarrow X$.
Denote the pre-image $O':=f^{-1}(O)\subset X$.
The chain rule of~\cite{aft1} gives the identities ${\sf id}_{O} = D_Z(f\circ f^{-1}) = D_Z f \circ D_Zf^{-1}$ and ${\sf id}_{O'} = D_Z(f^{-1}\circ f) = D_Z f^{-1} \circ D_Zf$.
This demonstrates an inverse to $D_Zf$.
We thus conclude that the map factors as $D_Z\colon \Emb_Z(X) \to \Aut_Z(X)$, and therefore that the isotopy $\ov{\gamma}$ witnesses $D_Z$ as a homotopy inverse to the inclusion $\Aut_Z(X) \to \Emb_Z(X)$.

\end{proof}

The efforts of this article culminate in the following example: an $\infty$-category $\Bun$ that classifies constructible bundles.  
Recall the Definition~\ref{def.bun} of $\bun$, which Theorem~\ref{isaquasi-cat} verifies is a transversality sheaf.  
\begin{definition}[$\Bun$]\label{def.Bun}
$\Bun$ is the $\oo$-category associated to the transversality sheaf $\bun$ by way of the functor $\trans\ra \Stri$ and the equivalence $\Stri\simeq \Cat_{\oo}$ of Theorem~\ref{striation-gcat}.  

\end{definition}

\begin{remark}
For each conically smooth stratified space $K$, there is an identification of the space of functors
\[
\Map\bigl(\exit(K),\Bun\bigr)~\simeq~\Bigl([q]\mapsto \bigl\{X\underset{\sf cbl} \longrightarrow K\times \Delta^q_e \bigr\}  \bigr)\Bigr)
\]
as the space associated to the simplicial set for which a $q$-simplex is a constructible bundle over $K\times \Delta^q_e$.
Each object of the $\infty$-category $\Bun$ is represented by a conically smooth stratified space; the space of morphisms in $\Bun$ from $X_0$ to $X_1$ is represented by the simplicial set for which a $q$-simplex is a constructible bundle $X\to \Delta^1\times\Delta^q_e\cong [0,1]\times\RR^q$ with identifications $X_0\times\RR^q \cong X_{|\{0\}\times\RR^q}$ and $X_1\times\RR^q\cong X_{|\{1\}\times\RR^q}$ over $\RR^q$.  

\end{remark}

\begin{remark}
For each conically smooth stratified space $K$, the set of components of the space of functors
\[
\pi_0\Map\bigl(\exit(K),\Bun\bigr)~\cong~ \Bigl\{X\underset{\sf cbl}\longrightarrow K\Bigr\}_{/\sf conc}
\]
is identified as the set of smooth concordance classes of constructible bundles over $K$.  
In this sense, the $\infty$-category $\Bun$ can be interpreted as one that classifies constructible bundles.  

\end{remark}

\begin{remark}\label{as-spans}
Corollary~\ref{0burn-bun-ess} implies a surjection
\[
\bigl\{X_0 \xla{\sf p.cbl} L  \xra{\open} X_1\bigr\}_{/\sf iso}\longrightarrow \pi_0 \Bun(\Delta^1)
\]
to the set of path components of the space of morphisms of $\Bun$,
from the set of isomorphism classes of spans among stratified spaces for which the leftward map is proper and constructible and the rightward map is open.  
Corollary~\ref{gpd-ess} generalizes this as a surjection
\[
\bigl\{X_0 \xla{\sf p.cbl} L_{0,1}  \xra{\open} X_1\xla{\sf p.cbl} L_{1,2}\xra{\open} \dots \xla{\sf p.cbl} L_{p-1,p} \xra{\open}X_p\bigr\}_{/\sf iso}\longrightarrow \pi_0 \Bun(\Delta^p)
\]
for each $p\geq 0$.  
Entertain the existence of a pre-Burnside category of $\strat$ (with leftward factor \emph{proper and constructible} and with rightward factor \emph{open}); composition in this pre-Burnside category is defined precisely because constructible bundles admit base-change.
The surjections above can be improved as a functor from this pre-Burnside category to the $\infty$-category $\Bun$ which, on the level of homotopy categories, is essentially surjective and full.  
Because it is not important for our purposes, we forgo justifying this assertion.

\end{remark}

There are subcategories
\[
\cbun ~\subset ~\bun^{\sf cpt}~\subset~ \bun 
\]
consisting of those constructible bundles $X\to K$ that are \emph{proper}, and those that have compact fibers.

\begin{observation}\label{variants-are-trans}
Following the proof that $\bun$ is a transversality sheaf (Theorem~\ref{isaquasi-cat}), the subcategories $\cbun$ and $\bun_{\leq n}$ and $\bun^{\sf cpt}$ too are transversality sheaves.

\end{observation}

After Observation~\ref{variants-are-trans}, there results $\infty$-subcategories
\[
\cBun~\subset~\Bun^{\sf cpt}~\subset~ \Bun ~\supset~ \Bun_{\leq n}
\]
as well as their intersections $\cBun_{\leq n}\subset\Bun^{\sf cpt}\subset \Bun$, the second inclusion being full, where $\Bun_{\leq n}$ consists of those constructible bundles whose fiberwise dimension is bounded by $n$. (Note that the objects of $\cBun$ are compact, which is why we notate with the {\sf c}.)
\begin{lemma}\label{Fin}
There are equivalences of $\infty$-categories
\[
\cBun_{\leq 0}\xra{~\simeq~} \Fin^{\op}\qquad \text{ and }\qquad \Bun_{\leq 0}^{\sf cpt} \xra{~\simeq~} \Fin_\ast^{\op}
\] 
where $\Fin$ is the category of finite sets and $\Fin_\ast$ is the category of pointed finite sets.
\end{lemma}

\begin{proof}
We prove the second equivalence.
The first equivalence is a restriction of the second equivalence, and follows by inspecting the construction of the second equivalence.  

First, we show that $\Bun^{\sf cpt}_{\leq 0}$ is a discrete $\oo$-category. That is, we show that, for any two objects $I$ and $J$ of $\Bun^{\sf cpt}_{\leq 0}$ (which are compact 0-manifolds, hence finite sets), the space of morphisms $\Bun^{\sf cpt}(I, J)$ has the homotopy type of a set. This is to show that, for any smooth manifold $K$, each map from its underlying space $K\ra \Bun^{\sf cpt}(I, J)$ factors through $K \ra \pi_0 K$, its set of path components. 
Such a map from $K$ is the data of the a diagram among stratified spaces
\[
\xymatrix{
I\times K\ar[r]\ar[d]&X\ar[d]^{\cbl}&\ar[l]J\times K\ar[d]\\
\Delta^{\{0\}}\times K \ar[r]&\Delta^1 \times K&\ar[l]\Delta^{\{1\}}\times K}
\]
in which the two squares are pullback.
Consider the link $\Link_{I\times K}(X)$ of the fiber over $\Delta^{\{0\}}\times K$ in the total space $X$.
This link is equipped with a conically smooth map $\Link_{I\times K}(X) \to \Link_{\Delta^{\{0\}}\times K}(\Delta^1\times K) \cong K$, which is a constructible bundle with compact fibers. 
Because $K$ is trivially stratified, this constructible bundle $\Link_{I\times K}(X)  \to K$ is a finite sheeted covering space.  
Furthermore, this link is equipped with maps over $K$:
\begin{equation}\label{98}
I\times K  \xla{~\pi~}  \Link_{I\times K}(X) \xra{~\gamma~}J\times K
\end{equation}
in which $\pi$ is a proper constructible bundle and $\gamma$ is an open map. 
Because the codomain of $\gamma$ is a trivial finite-sheeted covering space over $K$, we conclude that the finite sheeted covering space $\Link_{I\times K}(X)  \to K$ is trivializable.
Consequently, the span~(\ref{98}) is pulled back from a likewise span 
\[
I \times \pi_0 (K) \longleftarrow L\times \pi_0 (K)\longrightarrow J\times \pi_0 (K)~.
\]
It follows that the proper constructible bundle $X\to K$ is pulled back from a proper constructible bundle $\Delta^1\times \pi_0(K)$. 
We conclude that the given map $K\to \Bun_{\leq 0}^{\sf cpt}(I,J)$ factors through the canonical map $K\to \pi_0(K)$, as desired.

Now consider a $K$-point $\exit(K)\ra \Bun_{\leq 0}^{\sf cpt}$ classifying a constructible bundle $X\to K$ with compact fibers.
Consider the topological space $X^\ast$ over $K$, equipped with an embedding $X\hookrightarrow X^\ast$ over $K$, as well as a section $K\to X^\ast$ for which $X^\ast\to K$ is proper and the map $K\amalg X \to X^\ast$ is a bijection.  
There exists a unique such $X^\ast$, and it inherits the structure of a conically smooth stratified space with respect to which the maps $K\amalg X\to X^\ast \to K$ are conically smooth.  
The map $X^\ast \to K$ has the property that each diagram among stratified spaces
\[
\xymatrix{
Z   \ar[r]  \ar[d]_-{\{1\}}
&
X^\ast \ar[d]
\\
Z\times \Delta^1  \ar[r]  \ar@{-->}[ur]
&
K
}
\]
has a unique filler for $Z$ compact.  

The functor $\Bun^{\sf cpt}_{\leq 0} \to \Fin_\ast^{\op}$ is given by assigning to such a $K$-point the functor
\[
\exit(X^\ast) \longrightarrow \exit(K)
\]
together with its section.
The path lifting property implies this functor is a right fibration, and therefore is classified by a functor $(\Bun_{\leq 0}^{\sf cpt})^{\op} \to \Fin_\ast$.
This functor is evidently essentially surjective.  

To argue that this functor is fully faithful we construct an inverse on mapping spaces.  
Notice that, for each map $I_+\xra{f} J_+$ among based finite sets, the reversed mapping cylinder (Definition \ref{def-cylr})
\[
{\sf Cylr}(f) \longrightarrow \Delta^1
\]
is equipped with a standard map to the topological $1$-simplex, which is in fact a constructible bundle of stratified spaces, and the base-point determines a section. 
The map ${\sf Cylr}(f) \smallsetminus \Delta^1 \to \Delta^1$ is a constructible bundle with compact 0-dimensional fibers, and thus describes a morphism of $\Bun^{\sf cpt}_{\leq 0}$.  By direct inspection, the first equivalence is a restriction of the second.

\end{proof}

\subsection{The absolute exit-path $\oo$-category}
We introduce another interesting striation sheaf, the absolute exit-path $\oo$-category. This example ties together the previously considered $\Bun$ and exit-paths.

\begin{definition}\label{def.big-exit}
The category of \emph{absolute exit-paths} $\fexit$
has as an object a constructible bundle $X\xra{f} K$ together with a section $K\xra{s} X$. A morphism from $(X\overset{s}{\underset{f}\leftrightarrows} K)$ to $(X'\overset{s'}{\underset{f'}\leftrightarrows} K')$ is a pullback diagram in $\strat$
\[
\xymatrix{
X  \ar[r]^-{g'}  \ar[d]_f
&
X'  \ar[d]^{f'} &\text{such that the diagram of sections}& X   \ar[r]^-{g'} 
&
X' 
\\
K  \ar[r]^-{g}
&
K' && K    \ar[u]^s    \ar[r]^-g
&
K'     \ar[u]_{s'}
}
\]
commutes.
Composition is given by concatenating such squares horizontally and composing horizontal arrows.  
\end{definition}

Notice the projection $\fexit\to \bun$, given by $(X\overset{s}{\underset{f}\leftrightarrows} K)\mapsto (X\xra{f}K)$, and the composite functor $\fexit \to \bun \to \strat$.

\begin{theorem}\label{exit-transverse}
The functor $\fexit\ra \strat$ is a transversality sheaf.
\end{theorem}
\begin{proof}
We show the six conditions, in sequence.
Most of these arguments follow those proving that $\bun$ is a transversality sheaf (Theorem~\ref{isaquasi-cat}), so we will be terse.  

\noindent
{\bf Right fibration:} 
Since sections restrict contravariantly, the projection $\fexit \ra \bun$ is a category fibered in sets. In particular, this projection is a right fibration. 
It follows that the composition $\fexit\to \bun \to \strat$ is a right fibration, because $\bun\to \strat$ is.

\noindent
{\bf Isotopy extension:} 
This follows a nearly identical argument as the verification for $\bun$ (Theorem~\ref{isaquasi-cat}).

\noindent
{\bf Cone-local sheaf:} 
After the corresponding properties of $\bun \to \strat$, the sheaf condition, and also the cone-local condition, follows for $\fexit$ because covering sieves, as well as blow-up squares, are colimits in $\strat$.

\noindent
{\bf Consecutive:} 
Let $0<k<p$ and $q\geq 0$.
Recall the notational conventions from the proof of consecutivity for $\bun$ in Theorem~\ref{isaquasi-cat}.

Let 
\[
X_0 \to D^p_k\times \Delta^q_e \underset{D^p_k\times \partial \Delta^q_e}\amalg \Delta^p\times \partial \Delta^q_e
\]
be a constructible bundle, equipped with a section 
\[
\sigma_0\colon D^p_k\times \Delta^q_e \underset{D^p_k\times \partial \Delta^q_e}\amalg \Delta^p\times \partial \Delta^q_e\longrightarrow X_0~.
\]
We must extend this constructible bundle $X_0$ to one over $\Delta^p\times \Delta^q_e$, and the section $\sigma_0$ to one over $\Delta^p\times\Delta^q_e$.

Consider the terminal refinement $\w{X}_0\to X$ for which the conically smooth map $\sigma_0$ is a proper constructible embedding.
Because $\sigma_0$ is a section, the composition 
\[
\w{X}_0 \longrightarrow X \longrightarrow D^p_k\times \Delta^q_e \underset{D^p_k\times \partial \Delta^q_e}\amalg \Delta^p\times \partial \Delta^q_e
\]
is a constructible bundle.  
Using that $\bun$ is consecutive, choose a constructible bundle $\w{X}$ over $\Delta^p\times \Delta^q_e$ extending $\w{X}_0$.
Consider the initial proper constructible subspace of $\w{X}$ containing the image of $\sigma_0$.
This subspace projects to $\Delta^p\times \Delta^q_e$ as an isomorphism.  
The inverse of this isomorphism defines a section $\sigma \colon \Delta^p\times \Delta^q_e \to \w{X}$.  
Consider the initial refinement $\w{X}\to X$ that restricts over $D^p_k\times \Delta^q_e \underset{D^p_k\times \partial \Delta^q_e}\amalg \Delta^p\times \partial \Delta^q_e$ as $\w{X}_0\to X_0$. 
This constructible bundle $X\to \Delta^p\times \Delta^q_e$, together with the composite section $\sigma \colon \Delta^p\times \Delta^q_e \to \w{X} \to X$, is the desired extension.

\noindent
{\bf Univalent:}
For $\bun$, the proof of univalence is implied by the equivalence $\Aut(X)\simeq \Emb^\sim(X,X)$. The statement for $\fexit$ follows by the identical argument given the equivalence
\[\Aut_\ast(X)\simeq \Emb_\ast^\sim(X,X)\]
between pointed automorphisms and pointed self-embedding of a pointed stratified space $X$ which are isotopic to pointed automorphisms. The pointed assertion follows from the unpointed one, since these spaces are fibers of evaluation maps to $X$.
\end{proof}

The following result connects the three main striation sheaves considered in this paper, asserting that the fibers of $\Exit$ over $\Bun$ consist of exit-path $\oo$-categories. This explains our calling $\Exit$ the absolute exit-path $\oo$-category.

\begin{prop}
For each constructible bundle $X\xra{f} K$, classified by a functor $\exit(K) \xra{(X\xra{f} K)} \bu$, the diagram among $\infty$-categories
\[
\xymatrix{
\exit(X)   \ar[rr]^{(X\underset{K}\times X \rightleftarrows X)}   \ar[d]_-{\exit(f)}
&&
\Exit\ar[d]
\\
\exit(K)     \ar[rr]^{(X\xra{f} K)}
&&
\bu
}
\]
is a pullback.
\end{prop}

\begin{proof}
Let $Z\to K$ be a conically smooth map between stratified spaces.  Denote the pullback $X_{|Z}:= Z \times_K X$, which exists by Lemma \ref{cbls-pullback}.
The space of $Z$-points of $\exit(X)$ over $Z\to K$ is the space of sections $\Gamma\bigl(X_{|Z} \to Z\bigr)$.
\end{proof}

\subsection{Classes of constructible bundles}\label{sec.subcats}
We isolate several useful classes of constructible bundles, and show that some of them form a factorization system on the $\infty$-category $\Bun$.

Consider a constructible bundle $X\xra{f} K$, as well as a closed constructible subspace $K_0\subset K$.  
Denote the preimage $X_0:=f^{-1}K_0\subset X$.  
There are blow-up diagrams
\[
\xymatrix{
\Link_{X_0}(X)  \ar[r]  \ar[d]_-{\pi_{X_0}}
&
\Unzip_{X_0}(X)  \ar[d]
&&
\Link_{K_0}(K)  \ar[r]  \ar[d]
&
\Unzip_{K_0}(K)  \ar[d]
\\
X_0  \ar[r]
&
X
&\text{ and }&
K_0  \ar[r]
&
K
}
\]
and a constructible map from the left diagram to the right diagram.  
Also, there exists an open conically smooth map under $\Link_{X_0}(X)$
\[
\gamma_{X_0}\colon \Link_{X_0}(X)  \times [0,1) \longrightarrow \Unzip_{X_0}(X)
\]
over 
an open conically smooth map under $\Link_{K_0}(K)$
\[
\Link_{K_0}(K)  \times [0,1) \longrightarrow \Unzip_{K_0}(K)~.
\]

\begin{definition}[Class~$\psi$]\label{def.classes}
For the situation above:
\begin{enumerate}
\item The constructible bundle $f$ is \emph{closed} at $K_0$ if 
\begin{itemize}
\item the constructible proper map $\pi_{X_0}\colon \Link_{X_0}(X)\to X_0$ is an embedding;
\item the open conically smooth
$
\gamma_{X_0}\colon \Link_{X_0}(X)\times[0,1)\cong \Unzip_{X_0}(X)_{|\Link_{K_0}(K)\times[0,1)}
$
can be chosen to be an isomorphism.  

\end{itemize}

\item The constructible bundle $f$ is \emph{creating} at $K_0$ if
\begin{itemize}
\item the constructible proper map $\pi_{X_0}\colon \Link_{X_0}(X)\to X_0$ is surjective;
\item the open conically smooth
$
\gamma_{X_0}\colon \Link_{X_0}(X)\times[0,1)\cong \Unzip_{X_0}(X)_{|\Link_{K_0}(K)\times[0,1)}
$
can be chosen to be an isomorphism.  

\end{itemize}

\item The constructible bundle $f$ is \emph{refining} at $K_0$ if
\begin{itemize}
\item the constructible proper map $\pi_{X_0}\colon \Link_{X_0}(X)\xra{\cong} X_0$ is an isomorphism;
\item the open conically smooth
$
\gamma_{X_0}\colon \Link_{X_0}(X)\times[0,1)\to \Unzip_{X_0}(X)_{|\Link_{K_0}(K)\times[0,1)}
$
can be chosen to be a refinement.  

\end{itemize}

\item The constructible bundle $f$ is \emph{embedding} at $K_0$ if
\begin{itemize}
\item the constructible proper map $\pi_{X_0}\colon \Link_{X_0}(X)\xra{\cong} X_0$ is an isomorphism;
\item the open conically smooth
$
\gamma_{X_0}\colon \Link_{X_0}(X)\times[0,1)\hookrightarrow \Unzip_{X_0}(X)_{|\Link_{K_0}(K)\times[0,1)}
$
can be chosen to be an embedding.  

\end{itemize}

\end{enumerate}
There are some umbrella classes.
\begin{enumerate}

\item The constructible bundle $f$ is \emph{active} at $K_0$ if 
\begin{itemize}
\item the constructible proper map $\pi_{X_0}\colon \Link_{X_0}(X)\to X_0$ is surjective.
\end{itemize}

\item The constructible bundle $f$ is \emph{proper constructible} at $K_0$ if 
\begin{itemize}
\item the open conically smooth map
$
\gamma_{X_0}\colon \Link_{X_0}(X)\times[0,1)~\cong ~\Unzip_{X_0}(X)_{|\Link_{K_0}(K)\times[0,1)}
$
can be chosen to be an isomorphism.  
\end{itemize}

\item The constructible bundle $f$ is \emph{open} at $K_0$ if 
\begin{itemize}
\item the constructible proper map $\pi_{X_0}\colon \Link_{X_0}(X)\to X_0$ is an isomorphism.
\end{itemize}

\end{enumerate}
For economy of language, we will use the placeholder $``\psi"$ for any of the above classes of constructible bundles.  We say $f$ is of class $\psi$ if it is at each closed constructible subspace $K_0\subset K$. 

\end{definition}

For the next observation, we consider the full subcategory ${\sf Cbl}^\psi(Y)\subset {\sf Cbl}(Y)$ of those constructible bundles over $Y$ of class $\psi$.
\begin{observation}\label{burn-psi}
For each class $\psi$ of constructible bundle, the condition of a constructible bundle being of class $\psi$ is independent of isomorphism-type over the base.  
Furthermore, for each compact conically smooth stratified space $Z$, the preimage in ${\sf Burn}_1(Z) \to {\sf Cbl}\bigl(\sC(Z)\bigr)$ of ${\sf Cbl}^\psi\bigl(\sC(Z)\bigr)$ is the full subcategory consisting of those $(X_0 \xla{\pi} L \underset{{\rm over~}Z}{\xra{\gamma}} X_1)$ for which the constructible bundles $L\to Z\leftarrow X_1$ are of class $\psi$, and 
\begin{itemize}
\item should $\psi$ be \emph{closed}, then $\pi$ is an embedding and $\gamma$ is isotopic to an isomorphism;

\item should $\psi$ be \emph{creating}, then $\pi$ is surjective and $\gamma$ is isotopic to an isomorphism;

\item should $\psi$ be \emph{refining}, then $\pi$ is an isomorphism and $\gamma$ is isotopic to a refinement;

\item should $\psi$ be \emph{embedding}, then $\pi$ is an isomorphism and $\gamma$ is isotopic to an embedding;

\item should $\psi$ be \emph{active}, then $\pi$ is surjective;

\item should $\psi$ be \emph{proper constructible}, then $\gamma$ is isotopic to an isomorphism;

\item should $\psi$ be \emph{open}, then $\pi$ is an isomorphism.  

\end{itemize}

\end{observation}

\begin{lemma}\label{classes-pullback}
For each class $\psi$ of constructible bundle, the projection $\bun^\psi \to \strat$ is a transversality sheaf.

\end{lemma}

\begin{proof}
Being of class $\psi$ is preserved under base change, by inspection.
It follows that the projection $\bun^\psi \to \strat$ is a right fibration.
The proof that $\bun$ satisfies the conditions of Theorem~\ref{transversality.conditions} specializes to a proof that $\bun^\psi$ satisfies the conditions of Theorem~\ref{transversality.conditions}, with the essential surjectivity part of the {\bf consecutive} condition the only aspect that is not direct from inspection.  
To verify this essential surjectivity follows using Observation~\ref{burn-psi}.

\end{proof}

\begin{definition}\label{bun-psi}
For each class $\psi$ of constructible bundle, $\Bun^\psi$ is the $\infty$-category associated to the transversality sheaf $\bun^\psi$ by way of the functor $\trans\to \Stri$ and the equivalence $\Stri\simeq \Cat_\infty$ of Theorem~\ref{striation-gcat}.

\end{definition}

For each class $\psi$ of constructible bundle, the inclusion $\bun^\psi \to \bun$ of right fibrations over $\strat$ determines a functor among $\infty$-categories $\Bun^\psi \to \Bun$.

\begin{lemma}\label{psi-mono}
For each class $\psi$ of constructible bundle, the functor 
\[
\Bun^\psi \to \Bun
\]
is a monomorphism.

\end{lemma}
\begin{proof}
We must show that the map of spaces $\Bun^\psi(\Delta^p)\to \Bun(\Delta^p)$ is an inclusion of components for $p=0,1$. 
By construction, this map of spaces is represented by the inclusion of simplicial groupoids
\[
\bun^{\psi}(\Delta^p\times \Delta^\bullet_e)   ~\hookrightarrow~ \bun(\Delta^p\times \Delta^\bullet_e)~.
\]
We can therefore solve our problem by verifying the following assertion.
\begin{itemize}
\item[~]
Let $X_0 \to \Delta^p$ and $X_1\to \Delta^p$ be two constructible bundles of class $\psi$.
Suppose they are concordant as constructible bundles over $\Delta^p$, which is to say there is a constructible bundle $X\to \Delta^p\times \Delta^1_e$ restricting over $\Delta^{\{i\}}_e$ as $X_{i}$.  
Then $X\to \Delta^p\times \Delta^1_e$ too is of class $\psi$.  
\end{itemize}
Lemma~\ref{cbls-with-R} grants an isomorphism $X \cong X_{\frac{1}{2}}\times \Delta^1_e$ over $\Delta^p\times \Delta^1_e$, where $X_{\frac{1}{2}}$ is the fiber over $\frac{1}{2}\in\RR\cong \Delta^1_e$ of the composite map $X\ra \Delta^p\times \Delta^1_e \ra \Delta^1_e$.
The assertion follows using the first part of Observation~\ref{burn-psi}.

\end{proof}

We phrase the next result in terms of \emph{factorization systems}, which we first define.
For $\cC$ an $\infty$-category, we say a pair $(\cL,\cR)$ of $\infty$-subcategories of $\cC$ forms a \emph{factorization system} on $\cC$ if, for each morphism $[1]\xra{f} \cC$ the $\infty$-category of $(\cL,\cR)$-factorizations ${\sf Fact}_{\cL,\cR}(f)$ is terminal.
Here ${\sf Fact}_{\cL,\cR}(f)$ is defined as the pullback in the diagram
\[
\xymatrix{
{\sf Fact}_{\cL,\cR}(f)  \ar[rr]  \ar[d]
&&
\cC^{\{0<1<2\}}  \ar[d]
\\
\cL^{\{0<1\}} \times \cR^{\{1<2\}}  \ar[r]^-{\{f\}}
&
\cL^{\{0<1\}} \times\cC^{\{1<2\}}\times \cR^{\{1<2\}}  \ar[r]
&
\cC^{\{0<1\}}\times \cC^{\{0<2\}}\times \cC^{\{1<2\}}
}
\]
where $\cX^{\{0<\ldots<k\}}$ denotes the functor $\oo$-category $\Fun({\{0<\ldots<k\}}, \cX)$.
\begin{theorem}\label{Bun-factorization}
The pair of $\infty$-subcategories $(\Bun^{\cls},\Bun^{\act})$ forms a factorization system on $\Bun$, where these $\oo$-subcategories are given by Definition \ref{bun-psi} in the cases $\psi=\cls$ and $\psi = \act$.

\end{theorem}

\begin{proof}
Fix a constructible bundle $E\xra{f} \Delta^1$.
Choose an open conically smooth map $\gamma$ as in the span among stratified spaces 
\begin{equation}\label{concat-omega}
E_{|\Delta^{\{0\}}} \xla{\pi} \Link_{E_{|\Delta^{\{0\}}}}(E) \xra{\gamma} E_{|\Delta^{\{1\}}}
\end{equation}
in which $\pi$ is the standard projection from the link.  
By way of the equivalence~(\ref{burn-segal}), the concatenation of spans
\begin{equation}\label{concat-1}
E_{|\Delta^{\{0\}}} \xla{~\pi~} \Link_{E_{|\Delta^{\{0\}}}}(E) \xra{~=~}  \Link_{E_{|\Delta^{\{0\}}}}(E)\xla{~=~} \Link_{E_{|\Delta^{\{0\}}}}(E) \xra{~\gamma~} E_{|\Delta^{\{1\}}}
\end{equation}
determines an object of ${\sf Burn}_2(\ast)$.
The functor~(\ref{p-burn-to-bun}) thus assigns a constructible bundle $\ov{E} \to \oC^2(\ast) \cong \Delta^2$ to this concatenation of spans.

Fix a functor $\exit(Z)\xra{f_Z} {\sf Fact}_{\Bun^{\cls},\Bun^{\act}}(f)$, with $Z$ a compact conically smooth stratified space.
This functor $f_Z$ classifies a constructible bundle $X\xra{f_Z} Z\times \Delta^2$ with the following properties:
\begin{itemize}
\item $f_Z$ is \emph{closed} at $Z\times \Delta^{\{0<1\}}$,
\item $f_Z$ is \emph{active} at $Z\times \Delta^{\{1<2\}}$,
\item there is an isomorphism $X_{|Z\times \Delta^{\{0<2\}}} \cong Z\times E$ over $Z\times \Delta^{\{0<2\}}$. 
\end{itemize}
Choose an open conically smooth maps $\gamma_{01}$ over $Z\times \Delta^{\{0<1\}}$ and $\gamma_{12}$ over $Z\times \Delta^{\{1<2\}}$ as in the concatenation of spans that map constructibly to $Z$:
\begin{equation}\label{concat-2}
\Small
X_{|Z\times \Delta^{\{0\}}} \xla{\pi_{01}} \Link_{X_{|Z\times \Delta^{\{0\}}}}(X_{|Z\times \Delta^{\{0<1\}}})\xra{\gamma_{01}} X_{|Z\times \Delta^{\{1\}}}   
\xla{\pi_{12}} \Link_{X_{|Z\times \Delta^{\{1\}}}}(X_{|Z\times \Delta^{\{1<2\}}})\xra{\gamma_{12}} X_{|Z\times \Delta^{\{2\}}} .
\end{equation}
The equivalence~(\ref{burn-segal}) recognizes this concatenation of spans as an object of ${\sf Burn}_2(\ast)$ constructibly over $Z$.  
A last application of Corollary~\ref{gpd-ess} gives that, up to isomorphism over $Z\times \Delta^2$, the constructible bundle $f_Z$ is determined from the diagram~(\ref{concat-2}). 

The first condition on $f_Z$ in particular implies the map $\gamma_{01}$ can be taken to be an isomorphism, and so the composing the spans of~(\ref{concat-2}) gives a span
\begin{equation}\label{concat-0}
X_{|Z\times \Delta^{\{0\}}} \xla{~\pi_{12}\circ\gamma_{01}^{-1} \circ \pi_{01}~}  \Link_{X_{|Z\times \Delta^{\{1\}}}}(X_{|Z\times \Delta^{\{1<2\}}})\xra{~\gamma_{12}~} X_{|Z\times \Delta^{\{2\}}} .
\end{equation}
Through an ultimate application of Corollary~\ref{link-link}, there results a map from the diagram~(\ref{concat-0}) to the diagram
\begin{equation}\label{concat-3}
X_{|Z\times \Delta^{\{0\}}} \xla{\pi_{02}} \Link_{X_{|Z\times \Delta^{\{0\}}}}(X_{|Z\times \Delta^{\{0<2\}}})
\xra{\gamma_{02}} X_{|Z\times \Delta^{\{2\}}}
\end{equation}
under $X_{|Z\times \Delta^{\{0\}}}$ and $X_{|Z\times\Delta^{\{2\}}}$, and constructibly over $Z$.
The first condition on $f_Z$ further implies $\pi_{01}$ is an embedding, and the second condition on $f_Z$ implies $\pi_{12}$ is surjective.
Because there are no refinements from $\Link_{X_{|Z\times \Delta^{\{1\}}}}(X_{|Z\times \Delta^{\{1<2\}}})$ factoring $\pi_{12}$ and $\gamma_{12}$, then we conclude that there are no refinements from $\Link_{X_{|Z\times \Delta^{\{1\}}}}(X_{|Z\times \Delta^{\{1<2\}}})$ factoring $\pi_{12}\circ\gamma_{01}^{-1} \circ \pi_{01}$ and $\gamma_{12}$.  
Another ultimate application of Corollary~\ref{link-link} gives that this map of diagrams $(\ref{concat-0})\to(\ref{concat-3})$ is in fact an equivalence.  
Through the same ongoing reasoning,~(\ref{concat-3}) determines, up to isomorphism, the restricted constructible bundle $X_{|Z\times \Delta^{\{0<2\}}}\to Z\times \Delta^{\{0<2\}}$.  
The third condition on $f_Z$ thus implies~(\ref{concat-3}) is isomorphic to~(\ref{concat-omega}).
In this way, we conclude an isomorphism $X\cong Z\times \ov{E}$ over $Z\times \Delta^2$.
This isomorphism is classified by an equivalence between the functor $\exit(Z)\xra{f_Z}{\sf Fact}_{\Bun^{\cls},\Bun^{\act}}(f)$ and the constant functor at $\{f\}$.

\end{proof}

\subsection{Subcategories of $\Bun$}
We now realize certain $\oo$-subcategories of $\Bun$ in terms of subcategories of $\strat$.

\begin{definition}
The subcategory \[\strat^{\open}\subset \strat\]
has as morphisms those maps of stratified spaces which are open embeddings of underlying topological spaces. The categories
\[\strat^{\rf} \ {\rm and} \ \strat^{\emb}\]
are the further subcategories whose morphisms are: homeomorphisms of underlying topological spaces (for $\strat^{\rf})$; stratified open embeddings (for $\strat^{\emb}$).
\end{definition}

\begin{remark}
In \cite{aft1}, $\Strat^{\emb}$ is denoted $\Snglr$.
\end{remark}

\begin{definition}
The category \[\strat^{\pcbl}\] has as objects conically smooth stratified spaces and as morphisms those maps which are both proper and constructible. The categories
\[\strat^{\pcbl, \inj} \qquad {\rm and} \qquad  \strat^{\pcbl,\surj}\]
are the further subcategories whose morphisms are additionally injective or surjective, respectively.
\end{definition}

\begin{definition}
Given a subcategory $\strat^\psi\subset \strat$ and two constructible bundles $X\ra K$ and $Y\ra K$, the set
\[\strat^\psi_K(X,Y)\subset \strat_K(X,Y)\]
consists of those maps which belong to $\strat^\psi$ fiberwise: for every $t\in K$, the map of fibers $X_t\ra Y_t$ belongs to $\strat^\psi$.
A subcategory $\strat^\psi\subset\strat$ is parametrizable if for every pair of conically smooth stratified spaces $X$ and $Y$ the simplicial set
\[\strat^\psi_{\Delta^\bullet_e}(X\times\Delta^\bullet_e, Y\times\Delta^\bullet_e)\]
is a Kan complex. For $\strat^\psi$ parametrizable, the associated $\kan$-enriched category is $\Strat^\psi$.
\end{definition}

\begin{lemma}
For $\strat^\psi$ one of the subcategories $\strat^{\open}$, $\strat^{\rf}$, $\strat^{\emb}$, $\strat^{\pcbl}$, $\strat^{\pcbl,\surj}$, or $\strat^{\pcbl,\surj}$, then $\strat^\psi$ is parametrizable. That is, for any conically smooth stratified spaces $X$ and $Y$, the simplicial set
\[\strat^\psi_{\Delta^\bullet_e}(X\times\Delta^\bullet_e,Y\times\Delta^\bullet_e)\]
is a Kan complex.
\end{lemma}
\begin{proof}
The proof from \cite{aft1} that $\strat_{\Delta^\bullet_e}(X\times\Delta^\bullet_e,Y\times\Delta^\bullet_e)$ is itself a Kan complex immediately extends to these cases.
\end{proof}

We give a homotopy equivalent description of these $\kan$-enriched categories $\Strat^\psi$ using the model of simplicial objects in simplicial groupoids.

\begin{definition}
The value of the bisimplicial groupoid $\strat^\psi_{\star,\bullet}$ on $([p],[q])$ is the subcategory
\[\strat^\psi_{p,q}\subset \Fun([p], \strat_{/\Delta^q_e})^\sim\] consisting of those sequences
\[\xymatrix{
X_0\ar[drr]\ar[r]&X_1\ar[dr]\ar[r]&\ldots\ar[r]&X_p\ar[dl]\\
&&\Delta^q_e\\}
\]
for which each map $X_i \ra \Delta^q_e$ is a fiber bundle and each map $X_i\ra X_{i+1}$ belongs to $\strat^\psi_{\Delta^q_e}(X_i,X_{i+1})$.
\end{definition}

Note that the existence of pullbacks along constructible bundles ensures that $\strat_{\star,\bullet}^\psi$ is indeed a bisimplicial object.

\begin{definition}
The simplicial space $|\strat_{\star,\bullet}^\psi |$ is the vertical geometric realization of $\strat^\psi_{\star,\bullet}$ (i.e., in the $\bullet$ direction).
\end{definition}

\begin{lemma}
The simplicial spaces $|\strat^{\pcbl}_{\star,\bullet}|$ and $|\strat^{\open}_{\star,\bullet} |$ are complete Segal spaces.
\end{lemma}
\begin{proof}
The Segal condition is immediate. Completeness follows by observing that the spaces of objects are a disjoint union over all isomorphism classes $[X]$ of stratified spaces of the spaces of automorphisms $\Aut(X)$. This is given by identifying $\Delta^\bullet_e$ and the interior of $\Delta^\bullet$. For $\Strat^{\open}$, the underlying space of objects is a disjoint of union of spaces $\Emb^\sim(X,X)$, stratified embeddings which are isotopic to automorphisms, so we additionally use the equivalence $\Aut(X) \simeq\Emb^\sim(X,X)$ established in proving the univalence of $\bun$.

\end{proof}

\begin{lemma}
There is a natural equivalence
\[\Strat^\psi \simeq |\strat^\psi_{\star,\bullet}|\]
for $\Strat^\psi$ one of the $\oo$-categories $\Strat^{\open}$, $\Strat^{\rf}$, $\Strat^{\emb}$, $\Strat^{\pcbl}$, $\Strat^{\pcbl,\surj}$, or $\Strat^{\pcbl,\surj}$.\end{lemma}
\begin{proof}
There is a canonical map from $\Strat^\psi$, as it is a localization. The equivalence on spaces of objects follows immediately, since in both sides are disjoint unions of automorphism spaces. The equivalence on morphism spaces follows from the fact that every constructible bundle $X\ra \Delta^q_e$ splits as a product; once observed, the two sides have identical morphism spaces. 
\end{proof}

In the following definition of the open cylinder of a composable sequence of $p$ open stratified maps $X_0\ra\ldots \ra X_p$, we make use of the open complements of the standard filtration of the $p$-simplex,
\[\Delta^p_{> p-1} \subset \Delta^p_{> p-2}\subset \ldots \Delta^p_{>0}\subset \Delta^p~,\]
where $\Delta^p_{>i} \cong \Delta^p\smallsetminus \Delta^p_{\leq i}$.

\begin{definition}[Open cylinder]
Let $X_0 \ra X_1$ be an open conically smooth map. 
The \emph{open cylinder} $\cylo(X_0\ra X_1)$ is the pushout in $\strat$
\[\xymatrix{
X_0\times\Delta^1\smallsetminus\{0\}\ar[d]\ar[r]&X_1\times\Delta^1\smallsetminus\{0\}\ar[d]\\
X_0\times\Delta^1\ar[r]&\cylo\bigl(X_0\ra X_1\bigr)~.\\}
\]
Given a sequence of open conically smooth maps $X_0\ra \ldots \ra X_p$, the \emph{open cylinder} $\cylo(X_0\ra \ldots\ra X_p)$ is the iterated pushout in $\strat$
\[
X_0\times\Delta^p \underset{X_0\times\Delta^p_{>0}}\amalg X_1\times\Delta^p_{>0}
\ldots\underset{X_{p-2}\times\Delta^p_{>p-2}}\amalg X_{p-1}\times\Delta^p_{>p-2}\underset{X_{p-1}\times\Delta^p_{>p-1}}\amalg X_p\times\Delta^p_{>p-1}~.
\]

\end{definition}

\begin{observation}\label{obs.cylo}
Let $X_i \ra K$ be a collection of constructible bundles, for $0\leq i\leq p$, and let $X_0\ra \ldots X_p$ be a sequence of open conically smooth maps. The natural projection
\[\cylo\bigl(X_0\ra\ldots\ra X_p\bigr) \longrightarrow \Delta^p\times K\]
is a constructible bundle, where $\Delta^p$ carries the standard stratification.
\end{observation}

In the following definition of the cylinder of a composable sequence of $p$ stratified maps $X_p\ra\ldots \ra X_0$, we again make use of the standard filtration of the $p$-simplex, $\Delta^p_{\leq 0} \subset \Delta^p_{\leq 1}\subset \ldots \Delta^p_{\leq p}=\Delta^p$. For notational economy, we identify $\Delta^p_{\leq i} \cong\Delta^i$.

\begin{definition}[Reversed cylinder]\label{def-cylr}
Given a proper constructible bundle $X_1\ra X_0$, the reversed cylinder is the pushout in $\strat$
\[\cylr\bigl(X_1\ra X_0\bigr) := X_0 \underset{X_1\times\{0\}}\amalg X_1\times \Delta^1~.\]
Given a composable sequence of stratified maps $X_p \ra \ldots \ra X_0$, the reversed cylinder is
\[\cylr\bigl(X_p\ra \ldots \ra X_0\bigr) :=
X_0 \underset{X_1\times\{0\}}\amalg X_1\times \Delta^1\underset{X_2\times\Delta^1}\amalg X_2\times\Delta^2\underset{X_3\times\Delta^2}\amalg\ldots \underset{X_p\times\Delta^{p-1}}\amalg X_p\times\Delta^p~.\]
\end{definition}

\begin{remark}
The reversed cylinder admits an inductive description. There is a natural equivalence
\[\cylr\bigl(X_p\ra \ldots \ra X_0\bigr) \cong
\cylr\bigl(X_{p-1}\ra \ldots\ra X_0\bigr) \underset{X_{p}\times\Delta^{p-1}}\amalg X_p\times \Delta^p~.\]
\end{remark}

\begin{observation}\label{obs.cylr}
Let $X_i \ra K$ be a collection of constructible bundles, for $0\leq i\leq p$, and let $X_p \ra \ldots \ra X_0$ be a sequence of proper constructible bundles commuting with the maps to $K$. The natural projection
\[\cylr\bigl(X_p\ra\ldots\ra X_0\bigr) \longrightarrow \Delta^p\times K\]
is a constructible bundle, where $\Delta^p$ carries the standard stratification.
\end{observation}

\begin{theorem}
Taking reversed cylinders and open cylinders defines functors to $\Bun$
\[\xymatrix{
\Strat^{\open}\ar[rr]^{\cylo}&&\Bun&&\bigl(\Strat^{\pcbl}\bigr)^{\op}\ar[ll]_{\cylr}}
\]
which realize both sources as $\oo$-subcategories of $\Bun$. These functors restrict to subsidiary equivalences:
\begin{eqnarray}
\nonumber
\Strat^{\emb}\simeq \Bun^{\emb}
&
\qquad\text{ and }\qquad
&
\Strat^{\rf}\simeq \Bun^{\rf}~,
\\
\nonumber
\bigl(\Strat^{\pcbl,\inj}\bigr)^{\op}\simeq \Bun^{\cls}  
&
\qquad \text{ and }\qquad
&
\bigl(\Strat^{\pcbl,\surj}\bigr)^{\op}\simeq \Bun^{\creat} ~.
\end{eqnarray}
\end{theorem}
\begin{proof}

We first apply Observation \ref{obs.cylo} to show the open cylinder defines a map of bisimplicial groupoids
\[\cylr:\strat_{\star,\bullet}^{\open}\longrightarrow \bun(\Delta^\star\times\Delta^\bullet_e)~.\]
This sends an object of $\strat_{p,q}^{\open}$ given by
\[\xymatrix{
X_0\ar[drr]\ar[r]&X_{1}\ar[dr]\ar[r]&\ldots\ar[r]&X_{p-1}\ar[r]\ar[dl]&X_p\ar[dll]\\
&&\Delta^q_e\\}
\]
to the constructible bundle of stratified spaces
\[\cylo\bigl(X_0\ra\ldots \ra X_p\bigr) \longrightarrow \Delta^p\times\Delta^q_e\]
which is an object of the groupoid $\bun(\Delta^p\times\Delta^q_e)$. This is exactly Observation \ref{obs.cylo} applied in the case $K=\Delta^q_e$. We show that the map of spaces
\[\bigl|\strat^{\open}_{p,\bullet}\bigr|\longrightarrow \bigl|\bun(\Delta^p\times\Delta^\bullet_e)\bigr|\]
is an inclusion of components. First, it is immediate that this is an equivalence for $p=0$, since even more is true: the map
\[\strat^{\open}_{0,\bullet}\longrightarrow \bun(\Delta^\bullet_e)\]
is an isomorphism of simplicial groupoids. It now suffices, since both sides satisfy the Segal condition, to show the case of $p=1$.

To do so, we introduce an auxiliary space $\Bun^{\sf triv}$ of constructible bundles with choices of trivializations along strata, as given by a parallel vector field. That is, there is an fibration $\Bun^{\sf triv} \ra \bigl|\bun(\Delta^1\times\Delta^\bullet_e)\bigr|$ whose fiber over a point $X\ra \Delta^1$ is the space of isomorphisms
\[X_{|\Delta^1_{>0}}\cong X_1\times \Delta^1_{>0}\]
over $\Delta^1_{>0}$, where $X_1$ is the fiber over $\{1\}$. Note that the open cylinder functor naturally lifts to a map
\[\bigl|\strat^{\open}_{1,\bullet}\bigr|\longrightarrow\Bun^{\sf triv}(\Delta^1)\]
since such trivializations are present in the construction. Let $\Bun^{\sf triv}(\Delta^1)_o$ be subspace of components in the image of the open cylinder. Then we have a diagram
\[\xymatrix{
\Bun^{\sf triv}(\Delta^1)_o\ar@/_1pc/@{-->}[d]\ar@{^{(}->}[rr]&&\Bun^{\sf triv}(\Delta^1)\ar[d]^{\simeq}\\
\bigl|\strat^{\open}_{1,\bullet}\bigr|\ar[u]\ar[rr]&&\bigl|\bun(\Delta^1\times\Delta^\bullet_e)\bigr|\\}
\]
where the map above has an immediate retraction $\Bun^{\sf triv}(\Delta^1)_o \ra \bigl|\strat^{\open}_{1,\bullet}\bigr|$. This is immediately seen to be a homotopy equivalence, so the assertion follows. Further, this exchanges embedding/refining maps $X_0\ra X_1$ with maps which are embedding/refining, direct from the definition of the latter.

\medskip

We next apply Observation \ref{obs.cylr} to show the reversed cylinder defines a map of bisimplicial groupoids
\[\cylr: \bigl(\strat_{\star,\bullet}^{\pcbl}\bigr)^{\op}\longrightarrow \bun(\Delta^\star\times\Delta^\bullet_e)~.\]
This sends an object of $\strat_{p,q}^{\pcbl}$ given by
\[\xymatrix{
X_p\ar[drr]\ar[r]&X_{p-1}\ar[dr]\ar[r]&\ldots\ar[r]&X_1\ar[r]\ar[dl]&X_0\ar[dll]\\
&&\Delta^q_e\\}
\]
to the constructible bundle of stratified spaces
\[\cylr\bigl(X_p\ra\ldots \ra X_0\bigr) \longrightarrow \Delta^p\times\Delta^q_e\]
which is an object of the groupoid $\bun(\Delta^p\times\Delta^q_e)$. This is exactly Observation \ref{obs.cylr} applied in the case $K=\Delta^q_e$. We show that the map of spaces
\[\bigl|(\strat^{\pcbl}_{p,\bullet})^{\op}\bigr|\longrightarrow \bigl|\bun(\Delta^p\times\Delta^\bullet_e)\bigr|\]
is an inclusion of components. First, it is immediate that this is an equivalence for $p=0$, since even more is true: the map
\[(\strat^{\pcbl}_{0,\bullet})^{\op}\longrightarrow \bun(\Delta^\bullet_e)\]
is an isomorphism of simplicial groupoids. It now suffices, since both sides satisfy the Segal condition, to show the case of $p=1$.

To do so, we again use the auxiliary space $\Bun^{\sf triv}$ of constructible bundles with choices of trivializations along strata. Let $\Bun^{\sf triv}(\Delta^1)_r$ be subspace of components which are in the image of the reversed cylinder. Then we have a diagram
\[\xymatrix{
\Bun^{\sf triv}(\Delta^1)_r\ar@/_1pc/@{-->}[d]\ar@{^{(}->}[rr]&&\Bun^{\sf triv}(\Delta^1)\ar[d]^{\simeq}\\
\bigl|(\strat^{\pcbl}_{1,\bullet})^{\op}\bigr|\ar[u]\ar[rr]&&\bigl|\bun(\Delta^1\times\Delta^\bullet_e)\bigr|\\}
\]
where the map above has an immediate retraction $\Bun^{\sf triv}(\Delta^1)_r \ra \bigl|(\strat^{\pcbl}_{1,\bullet})^{\op}\bigr|$. This is immediately seen to be a homotopy equivalence, so the assertion follows. Further, this exchanges injective/surjective maps $X_1\ra X_0$ with maps which are closed/creation, direct from the definition of the latter.

\end{proof}


\begin{thebibliography}{99}

\bibitem[AF1]{fact} Ayala, David; Francis, John. Factorization homology of topological manifolds.  J. Topol. 8 (2015), no. 4, 1045--1084.  

\bibitem[AF2]{pkd} Ayala, David; Francis, John. Poincar\'e/Koszul duality. Preprint.  

\bibitem[AF3]{fibrations} Ayala, David; Francis, John. Fibrations of $\oo$-categories. Preprint.  

\bibitem[AF4]{bord} Ayala, David; Francis, John. The cobordism hypothesis. In preparation.

\bibitem[AFR]{emb1a} Ayala, David; Francis, John; Rozenblyum, Nick. Factorization homology I: higher categories. Preprint, 2015, arXiv: 1504.04007.

\bibitem[AFT]{aft1} Ayala, David; Francis, John; Tanaka, Hiro Lee. Local structures on stratified spaces. Adv. Math. 307 (2017), 903--1028.

\bibitem[BaDo]{baezdolan} Baez, John; Dolan, James.
Higher-dimensional algebra and topological quantum field theory.
J. Math. Phys. 36 (1995), no. 11, 6073--6105. 

\bibitem[BS]{clark-chris} Barwick, Clark; Schommer-Pries, Christopher. On the unicity of the homotopy theory of higher categories. Preprint, 2013.

\bibitem[Be]{bergner} Bergner, Julia. A model category structure on the category of simplicial categories. Trans. Amer. Math. Soc. 359 (2007), no. 5, 2043--2058.

\bibitem[BK]{holims.bousfield.kan} Bousfield, Aldridge; Kan, Daniel. Homotopy limits, completions and localizations. 
Lecture Notes in Mathematics, Vol. 304. Springer--Verlag, Berlin--New York, 1972.


\bibitem[Co]{cycat} Costello, Kevin. Topological conformal field theory and Calabi--Yau categories. Adv. Math. 210 (2007), no. 1, 165--214.

\bibitem[Cu]{curry} Curry, Justin Michael. Sheaves, cosheaves and applications. Thesis (Ph.D.)--University of Pennsylvania. 2014.

\bibitem[DI]{dugger-isaksen} Dugger, Daniel; Isaksen, Daniel. Topological hypercovers and $\mathbb{A}^1$-realizations. Math. Z. 246 (2004), no. 4, 667--689. 

\bibitem[Gr1]{groth} Grothendieck, Alexander. Pursuing stacks. Unpublished manuscript, 1983.

\bibitem[Gr2]{esquisse} Grothendieck, Alexandre. Esquisse d'un programme. London Math. Soc. Lecture Note Ser., 242, Geometric Galois actions, 1, 5--48, Cambridge Univ. Press, Cambridge, 1997.

\bibitem[Ha]{hatcher} Hatcher, Allen. Higher simple homotopy theory. Ann. of Math. (2) 102 (1975), no. 1, 101--137.

\bibitem[Hu]{hughes} Hughes, Bruce. The approximate tubular neighborhood theorem. Ann. of Math. (2) 156 (2002), no. 3, 867--889.

\bibitem[HTWW]{htww} Hughes, Bruce; Taylor, Laurence; Weinberger, Shmuel; Williams, Bruce.
Examples of exotic stratifications. Geom. Topol. 11 (2007), 1477--1505. 

\bibitem[Ja]{jardine} Jardine, John. Diagonal model structures. Theory Appl. Categ. 28 (2013), No. 10, 250--268.

\bibitem[Jo]{joyal1} Joyal, Andr\'e. Quasi-categories and Kan complexes. Special volume celebrating the 70th birthday of Professor Max Kelly. J. Pure Appl. Algebra 175 (2002), no. 1-3, 207--222.

\bibitem[JT]{joyaltierney}  Joyal, Andr\'e; Tierney, Myles. Quasi-categories vs Segal spaces. Categories in algebra, geometry and mathematical physics, 277--326, Contemp. Math., 431, Amer. Math. Soc., Providence, RI, 2007. 

\bibitem[Lu1]{HTT} Lurie, Jacob. Higher topos theory. Annals of Mathematics Studies, 170. Princeton University Press, Princeton, NJ, 2009. xviii+925 pp.

\bibitem[Lu2]{HA} Lurie, Jacob. Higher algebra. Preprint, 2014.

\bibitem[Lu3]{cobordism} Lurie, Jacob. On the classification of topological field theories. Current developments in mathematics, 2008, 129--280, Int. Press, Somerville, MA, 2009.

\bibitem[Ma1]{mather}  Mather, John. Stratifications and mappings. Dynamical systems (Proc. Sympos., Univ. Bahia, Salvador, 1971), pp. 195--232. Academic Press, New York, 1973.

\bibitem[Ma2]{mather.survey}  Mather, John. Notes on topological stability. Bulletin of the American Mathematical Society 49 (2012), 475--506.

\bibitem[Mi]{miller} Miller, David.
Strongly stratified homotopy theory.
Trans. Amer. Math. Soc. 365 (2013), no. 9, 4933--4962. 

\bibitem[Moe]{moerdijk} Moerdijk, Ieke. Bisimplicial sets and the group-completion theorem. Algebraic K-theory: connections with geometry and topology (Lake Louise, AB, 1987), 225--240,
NATO Adv. Sci. Inst. Ser. C Math. Phys. Sci., 279, Kluwer Acad. Publ., Dordrecht, 1989.

\bibitem[MV]{morel-voevodsky} Morel, Fabien; Voevodsky, Vladimir. $\mathbb{A}^1$-homotopy theory of schemes. Inst. Hautes \'Etudes Sci. Publ. Math. No. 90 (1999), 45--143 (2001).

\bibitem[Quil]{quillen.homotopical} Quillen, Daniel.
Homotopical algebra. 
Lecture Notes in Mathematics, No. 43 Springer--Verlag, Berlin--New York, 1967.

\bibitem[Qu1]{quinnends} Quinn, Frank. Ends of maps. II. Invent. Math. 68 (1982), no. 3, 353--424.

\bibitem[Qu2]{quinn} Quinn, Frank. Homotopically stratified sets.
J. Amer. Math. Soc. 1 (1988), no. 2, 441--499. 

\bibitem[RS]{normal} Rourke, Colin; Sanderson, Brian. An embedding without a normal microbundle.
Invent. Math. 3 1967, 293--299. 

\bibitem[Re1]{rezk}  Rezk, Charles. A model for the homotopy theory of homotopy theory. Trans. Amer. Math. Soc. 353 (2001), no. 3, 973--1007.

\bibitem[Re2]{rezk-n} Rezk, Charles. A Cartesian presentation of weak $n$-categories. Geom. Topol. 14 (2010), no. 1, 521--571.

\bibitem[Re3]{rezk3} Rezk, Charles. A model category for categories. Unpublished preprint, 2000.

\bibitem[Si]{siebenmann} Siebenmann, Laurence.
Deformation of homeomorphisms on stratified sets. I, II.
Comment. Math. Helv. 47 (1972), 123--136; ibid. 47 (1972), 137--163. 

\bibitem[SV]{SV} Suslin, Andrei; Voevodsky, Vladimir. Singular homology of abstract algebraic varieties. Invent. Math. 123 (1996), no. 1, 61--94.

\bibitem[Th]{thom} Thom, Ren\'e.
Ensembles et morphismes stratifi\'es. Bull. Amer. Math. Soc. 75 1969 240--284. 

\bibitem[To]{toen} To\"en, Bertrand. Vers une axiomatisation de la th\'eorie des cat\'egories sup\'erieures. K-Theory 34 (2005), no. 3, 233--263.

\bibitem[Tre]{treumann} Treumann, David. Exit paths and constructible stacks. Compos. Math. 145 (2009), no. 6, 1504--1532. 

\bibitem[Tro]{trotman} Trotman, David. Stability of transversality to a stratification implies Whitney (a)-regularity. Invent. Math. 50 (1978/79), no. 3, 273--277. 

\bibitem[Wa]{waldhausen}  Waldhausen, Friedhelm. Algebraic K-theory of spaces, a manifold approach. Current trends in algebraic topology, Part 1 (London, Ont., 1981), pp. 141--184, CMS Conf. Proc., 2, Amer. Math. Soc., Providence, R.I., 1982.

\bibitem[Wh1]{whitney0} Whitney, Hassler. Elementary structure of real algebraic varieties. Ann. of Math. (2) 66 (1957), 545--556.

\bibitem[Wh2]{whitney1} Whitney, Hassler. Local properties of analytic varieties. 1965 Differential and Combinatorial Topology (A Symposium in Honor of Marston Morse) pp. 205--244 Princeton Univ. Press, Princeton, N. J.

\bibitem[Wh3]{whitney2}  Whitney, Hassler. Tangents to an analytic variety. Ann. of Math. (2) 81 (1965), 496--549.

\bibitem[Wo]{woolf} Woolf, Jon. The fundamental category of a stratified space. J. Homotopy Relat. Struct. 4 (2009), no. 1, 359--387. 


\end{thebibliography}
\end{document}